\documentclass[12pt,reqno]{amsart}
\usepackage{amsmath,amsfonts,amssymb,amsthm,amsxtra,amscd}
\usepackage[all]{xy}
\usepackage{mathrsfs} 
\usepackage{graphicx} 
\numberwithin{equation}{section}

\theoremstyle{plain} 
\newtheorem{theorem}{Theorem}[section] 
\newtheorem{proposition}[theorem]{Proposition} 
\newtheorem{lemma}[theorem]{Lemma} 
\newtheorem{corollary}[theorem]{Corollary}

\theoremstyle{definition} 
\newtheorem{definition}{Definition}[section] 
 
\newtheorem{remark}[theorem]{Remark}

\newcommand{\p}{{\mathbb P}}

\newcommand{\z}{{\mathbb Z}} 
\newcommand{\pj}{{{\mathbb P}^1}}
\newcommand{\pii}{{{\mathbb P}^2}}
\newcommand{\piii}{{{\mathbb P}^3}}

\newcommand{\sce}{\mathscr{E}}
\newcommand{\scf}{\mathscr{F}} 
\newcommand{\scg}{\mathscr{G}}
\newcommand{\sco}{\mathscr{O}} 
\newcommand{\sch}{\mathscr{H}}
\newcommand{\sci}{\mathscr{I}}

\newcommand{\scl}{\mathscr{L}}
\newcommand{\scm}{\mathscr{M}}

\newcommand{\sct}{\mathscr{T}}

\newcommand{\fm}{{\mathfrak m}}

\newcommand{\tH}{\text{H}} 
\newcommand{\h}{\text{h}}

\newcommand{\izo}{\overset{\sim}{\rightarrow}} 
\newcommand{\Izo}{\overset{\sim}{\longrightarrow}} 
\newcommand{\ra}{\rightarrow} 
\newcommand{\lra}{\longrightarrow} 
\newcommand{\xra}{\xrightarrow}  
\newcommand{\vb}{\, \vert \, } 
\newcommand{\prim}{{\, \prime}} 
\newcommand{\secund}{{\prime \prime}}
\newcommand{\Ker}{\text{Ker}\, }
\newcommand{\Cok}{\text{Coker}\, }

\newcommand{\e}{\varepsilon}

\begin{document}

\title[Stable rank 3 vector bundles]{Stable rank 3 vector bundles on 
$\mathbb{P}^3$ with $c_1 = 0$, $c_2 = 3$}

\author[Coand\u{a}]{Iustin~Coand\u{a}} 
\address{Institute of Mathematics ``Simion Stoilow'' of the Romanian Academy, 
         P.O. Box 1-764, 
         RO--014700, Bucharest, Romania} 
\email{Iustin.Coanda@imar.ro} 





\begin{abstract}
We clarify the undecided case $c_2 = 3$ of a theorem of Ein, Hartshorne and 
Vogelaar [Math. Ann. 259 (1982), 541--569] about the restriction of a stable 
rank 3 vector bundle with $c_1 = 0$ on the projective 3-space to a general 
plane. It turns out that there are more exceptions to the stable restriction 
property than those conjectured by the three authors. One of them is a 
Schwarzenberger bundle (twisted by $-1$); it has $c_3 = 6$. There are also 
some exceptions with $c_3 = 2$ (plus, of course, their duals). We also prove, 
for completeness, the basic properties of the corresponding moduli spaces; 
they are all nonsingular and connected, of dimension 28.  
\end{abstract}

\maketitle 
\tableofcontents

\section*{Introduction}\label{S:intro} 

A basic technique for studying stable vector bundles on projective spaces
(over an algebraically closed field $k$ of characteristic 0) is 
to investigate their restrictions to general linear subspaces. The prototypes
are the Grauert-M\"{u}lich-Spindler theorem \cite{sp} asserting that if $E$ is
a semistable vector bundle on $\p^n$, $n \geq 2$, then, for the general line
$L \subset \p^n$, one has $E_L \simeq \bigoplus_{i=1}^r \sco_L(a_i)$, with
$a_1 \leq \cdots \leq a_r$ verifying $a_{i+1} - a_i \leq 1$, $i = 1 , \ldots ,
r-1$ and Barth's restriction theorem \cite{ba} asserting that if $E$ is a
stable rank 2 vector bundle on $\p^n$, $n \geq 3$, then its restricton to a
general hyperplane is stable unless $n = 3$ and $E$ is a (twist of a)
nullcorrelation bundle. After Gruson and Peskine reinterpreted Barth's
arguments, Ein, Hartshorne and Vogelaar \cite{ehv} were able to prove a
similar result for stable rank 3 vector bundles (actually, even reflexive
sheaves) on $\piii$. Their result is the following one$\, :$

\begin{theorem}\label{T:ehv}\emph{(Ein-Hartshorne-Vogelaar)}\quad   
If $E$ is a stable rank $3$ vector bundle with $c_1 = 0 $ on $\piii$ then the 
restriction of $E$ to a general plane is stable unless one of the following
holds$\, :$ 
\begin{enumerate}
\item[(1)] $c_2 \leq 3\, ;$   
\item[(2)] $E \simeq S^2N$, for some nullcorrelation bundle $N$ $($in which case
$c_2 = 4$ and $c_3 = 0)\, ;$ 
\item[(3)] There is an exact sequence$\, :$ 
\[
0 \lra \Omega_\piii(1) \lra E^\prim \lra \sco_{H_0}(-c_2+1) \lra 0\, ,
\]
for some plane $H_0 \subset \piii$, where $E^\prim$ is either $E$ or its dual
$E^\vee$.
\end{enumerate}
\end{theorem}

(The easier cases where $E$ has rank 3 and $c_1 = -1$ or $-2$ had been settled
earlier by Schneider \cite{sch}.) The three authors also show, in
\cite[Thm.~4.2]{ehv}, that, under the hypothesis of the above theorem, $c_2
\geq 2$ and $c_3 \leq c_2^2 - c_2$. Moreover, for $c_2 = 2$ they prove that 
the restriction of $E$ to any plane is not stable but in the case $c_2 = 3$ 
they assert, after the statement of \cite[Thm.~0.1]{ehv}, that they
"do not know exactly which bundles with $c_2 = 3$ have stable restrictions"
and conjecture that the only exceptions are again as in (3). 

The aim of this paper is to clarify the case $c_2 = 3$ of the theorem of Ein, 
Hartshorne and Vogelaar. Our main result is expressed by the next theorem.
As one can see from its statement, there are more exceptions than those
conjectured by the three authors (which shows that this case needs a special
treatment). 

\begin{theorem}\label{T:c10c23} 
Let $E$ be a stable rank $3$ vector bundle on $\piii$ with $c_1 = 0$, $c_2 = 3$ 
and $c_3 \geq 0$. 

\emph{(a)} If ${\fam0 H}^0(E_H) \neq 0$ for every plane $H \subset \piii$ then 
$c_3 = 6$ and there is an exact sequence$\, :$ 
\[
0 \lra \Omega_\piii(1) \lra E \lra \sco_{H_0}(-2) \lra 0\, , 
\]
for some plane $H_0$. 

\emph{(b)} If ${\fam0 H}^0(E_H^\vee) \neq 0$ for every plane $H \subset \piii$ 
then one of the following holds$\, :$ 
\begin{enumerate} 
\item[(i)] $c_3 = 6$ and, up to a linear change of coordinates in $\piii$, 
$E$ is the cokernel of the morphism
$\alpha \colon 3\sco_\piii(-2) \ra 6\sco_\piii(-1)$
defined by the transpose of the matrix$\, :$ 
\[
\begin{pmatrix} 
X_0 & X_1 & X_2 & X_3 & 0 & 0\\ 
0 & X_0 & X_1 & X_2 & X_3 & 0\\ 
0 & 0 & X_0 & X_1 & X_2 & X_3
\end{pmatrix}\, ; 
\]
\item[(ii)] $c_3 = 2$ and, up to a linear change of coordinates in $\piii$, 
$E$ is the cohomology sheaf of a monad of the form$\, :$ 
\[
0 \lra \sco_\piii(-2) \overset{\alpha}{\lra} 6\sco_\piii \overset{\beta}{\lra} 
2\sco_\piii(1) \lra 0\, , 
\]
with $\alpha = (X_2^2\, ,\, X_3^2\, ,\, -X_0X_2\, ,\, -X_1X_3\, ,\, X_0^2\, ,\, 
X_1^2)^{\fam0 t}$ 
and with $\beta$ defined by the matrix$\, :$ 
\[
\begin{pmatrix} 
X_0 & a_1X_1 & X_2 & a_1X_3 + a_3X_1 & 0 & a_3X_3\\ 
0 & b_1X_1 & X_0 & b_1X_3 + b_3X_1 & X_2 & b_3X_3
\end{pmatrix}\, , 
\]
where $a_1,\, a_3,\, b_1,\, b_3$ are scalars satisfying $a_1b_3 - a_3b_1 \neq 
0$. 
\end{enumerate}
\end{theorem} 

It is clear that the above theorem answers the question of Ein, Hartshorne and 
Vogelaar because if $c_3 < 0$ then $c_3(E^\vee) = -c_3 > 0$. 

We recall now, briefly, the results from \cite{ehv} that we use in the proof
of Thm.~\ref{T:c10c23} and then outline the extra arguments that we need to
complete its proof. A more detailed description of our method can be
found in Section~\ref{S:prelim}. 

Let $E$ be a stable rank 3 vector bundle on $\piii$ with
Chern classes $c_1 = 0$, $c_2$, $c_3$. The Riemann-Roch theorem asserts, in
this case, that $\chi(E(l)) = \chi(3\sco_\piii(l)) - (l+2)c_2 + c_3/2$,
$\forall \, l \in \z$. In particular, $c_3$ must be even. According to a
theorem of Spindler \cite{sp2} (see, also, \cite[Cor.~3.5]{ehv}), for the
general plane $H \subset \piii$ one has $\h^0(E_H) \leq 1$ and $\h^0(E_H^\vee)
\leq 1$. One deduces that $c_2 \geq 2$ and $c_3 \leq c_2^2 - c_2$ (see
\cite[Thm.~4.2]{ehv}). Assume, now, that $c_2 \geq 3$ and that, for the
general plane $H \subset \piii$, one has $\tH^0(E_H) = 0$ and
$\h^0(E_H^\vee) = 1$.
Applying to $E^\vee$ \cite[Prop.~7.4(a)]{ehv} and the first part of the
proof of \cite[Prop.~6.4]{ehv}, it follows that, for the general plane
$H \subset \piii$, there is an exact sequence$\, :$
\[
0 \lra \left(\Omega_\piii(1)\right)_H \lra E_H^\vee \lra \sco_{L_0}(-c_2+1)
\lra 0\, , 
\]
for some line $L_0 \subset H$. This implies that
$\h^1(E_{L_0}^\vee) \geq c_2 - 2$.

\vskip2mm 

Conversely, if a plane $H \subset \piii$ contains a line $L_0$ such that
$\h^1(E_{L_0}^\vee) \geq c_2 - 2$ then $E_H$ is not stable. [\emph{Indeed}, if
$E_H$ were stable then one would have $\h^1(E_H^\vee) = c_2 -3$, by
Riemann-Roch, and $\h^2(E_H^\vee(-1)) = \h^0(E_H(-2)) = 0$ and this would
imply that $\h^1(E_{L_0}^\vee) \leq c_2 - 3$.] This observation allows one to
show quickly that the bundles from Thm.~\ref{T:c10c23}(b) are exceptions to
the stable restriction property. \emph{Indeed}, if $E$ is the bundle from
Thm.~\ref{T:c10c23}(b)(i) let $(t_0 , t_1)$ and $(u_0 , u_1)$ be two linearly
independent elements of $k^2$ and let $L \subset \piii$ be the line of
equations $\sum t_0^{3-i}t_1^iX_i = \sum u_0^{3-i}u_1^iX_i = 0$. Since the kernel
of $\tH^0(\alpha_L^\vee(-1))$ contains the linearly independent elements
$(t_0^5 , \ldots , t_1^5)^{\text{t}}$ and $(u_0^5 , \ldots , u_1^5)^{\text{t}}$ it
follows that $\h^0(E_L^\vee(-1)) \geq 2$. Taking into account that
$E_L^\vee(-1)$ is a subbundle of $6\sco_L$, one deduces that
$E_L^\vee(-1) \simeq 2\sco_L \oplus \sco_L(-3)$ hence
$E_L^\vee \simeq 2\sco_L(1) \oplus \sco_L(-2)$. Since any plane
$H \subset \piii$ contains such a line [the planes of equation of the form
$\sum t_0^{3-i}t_1^iX_i = 0$, $(t_0 , t_1) \in k^2 \setminus \{(0 , 0)\}$, form
a twisted cubic curve $\Gamma$ in the dual projective space $\p^{3 \vee}$ and
the secants of $\Gamma$ fill the whole of $\p^{3 \vee}$], $E$ is an exception
to the stable restriction property.

If $E$ is one of the bundles from Thm.~\ref{T:c10c23}(b)(ii), let $L$ be a line
joining a point $P_0$ of the line $X_0 = X_2 = 0$ and a point $P_1$ of the line
$X_1 = X_3 = 0$. Since $X_0$ and $X_2$ (resp., $X_1$ and $X_3$) vanish at $P_0$
(resp., $P_1$) one deduces, using the restriction to $L$ of the dual of the
monad defining $E$, that $\h^1(E_L^\vee) = 1$. Since a general plane contains
a line of this kind, $E$ is an exception to the stable restriction property,
too. 

\vskip2mm 

Now, in order to show that the bundles listed in Thm.~\ref{T:c10c23} are the
only exceptions to the stable restriction property (in the case $c_1 = 0$,
$c_2 = 3$), we proceed as follows. Using the properties of the \emph{spectrum}
of a stable rank 3 vector bundle on $\piii$ (see Remark~\ref{R:spectrum}) we
describe the
Horrocks monad of any such bundle $E$ with $c_1 =0$, $c_2 = 3$, $c_3 \geq 0$
(hence $c_3 \in \{0\, ,\, 2\, ,\, 4\, ,\, 6\}$) that is not isomorphic to one
of the bundles from Thm.~\ref{T:c10c23}(a). Then we show that, for the general 
plane $H \subset \piii$, one has $\tH^0(E_H) = 0$. We use, for that, the map
$\mu \colon \tH^1(E(-1)) \otimes \sco_{\p^{3 \vee}}(-1) \ra \tH^1(E) \otimes
\sco_{\p^{3 \vee}}$ deduced from the multiplication map
$\tH^1(E(-1)) \otimes \tH^0(\sco_\piii(1)) \ra \tH^1(E)$. If one would have
$\tH^0(E_H) \neq 0$, for the general plane $H \subset \piii$, then $\mu$ would
have, generically, corank 1. Since $\tH^1(E)$ and $\tH^1(E(-1))$ are
$k$-vector spaces of small dimension (equal, for both, to $3 - c_3/2$) one
gets readily a \emph{contradiction}. Note that this settles, already, the case
$c_3 = 0$.

\vskip2mm 

Finally, if $E$ (as above) does not satisfy the stable restriction property
then, for the general plane $H \subset \piii$, one must have $\tH^0(E_H) = 0$
and $\h^0(E_H^\vee) = 1$. If $c_3 = 2$ and $E$ has an unstable plane or if
$c_3 = 4$ one gets a \emph{contradiction} by showing that the general plane
$H \subset \piii$ contains no line $L_0$ such that $\h^1(E_{L_0}^\vee) \geq 1$.
The argument for the case $c_3 = 4$ is a little bit lengthy because the best
we were able to do was to split it into several cases. The analysis of each 
case is, however, easy.

If $c_3 = 2$ and $E$ has no unstable plane (resp., $c_3 = 6$) we show that
$E$ is as in Thm.~\ref{T:c10c23}(b)(ii) (resp., Thm.~\ref{T:c10c23}(b)(i)).
We use, for that, the morphism
$\mu \colon \tH^1(E^\vee(-1)) \otimes \sco_{\p^{3 \vee}}(-1) \ra \tH^1(E^\vee)
\otimes \sco_{\p^{3 \vee}}$
deduced from the multiplication map
$\tH^1(E^\vee(-1)) \otimes \tH^0(\sco_\piii(1)) \ra \tH^1(E^\vee)$
and, in the case $c_3 = 6$, the main result of Vall\`{e}s \cite{val}.

As a byproduct of our description of the Horrocks monads of the stable rank 3
vector bundles on $\piii$ with $c_1 = 0$ and $c_2 = 3$, we also show that all
of their moduli spaces are nonsingular and connected, of dimension 28.

\vskip2mm 

\noindent 
{\bf Notation.}\quad  (i) We denote by $\p^n$ the projective $n$-space over 
an algebraically closed field $k$ of characteristic 0. We use the classical 
definition $\p^n = \p(V) := (V \setminus \{0\})/k^\ast$, where $V := k^{n+1}$. 
If $e_0 , \ldots , e_n$ is the canonical basis of $V$ and $X_0 , \ldots , X_n$ 
the dual basis of $V^\vee$ then the homogeneous coordinate ring of $\p^n$ is the 
symmetric algebra $S := S(V^\vee) \simeq k[X_0 , \ldots , X_n]$.

(ii) If $\scf$ is a coherent sheaf on $\p^n$ and $i \geq 0$ an integer, we
denote by $\tH^i_\ast(\scf)$ the graded $k$-vector space 
$\bigoplus_{l \in \z}\tH^i(\scf(l))$ endowed with its natural structure of graded 
$S$-module. We also denote by $\h^i(\scf)$ the dimension of $\tH^i(\scf)$ as a 
$k$-vector space.

(iii) If $Y \subset X$ are closed subschemes of $\p^n$, we denote by $\sci_X$
the ideal sheaf of $\sco_\piii$ defining $X$ and by $\sci_{Y , X}$ the ideal
sheaf of $\sco_X$ defining $Y$ as a closed subscheme of $X$, i.e.,
$\sci_{Y , X} = \sci_Y/\sci_X$. If $\scf$ is a coherent sheaf on $\p^n$, we
denote the tensor product $\scf \otimes_{\sco_\p} \sco_X$ by $\scf_X$ and
identify it with the restriction $\scf \, \vert \, X$ of $\scf$ to $X$. In
particular, if $x$ is a (closed) point of $\p^n$ then $\scf_{\{x\}}$ is just
the \emph{reduced stalk} $\scf(x) := \scf_x/\fm_x\scf_x$ of $\scf$ at $x$. 

(iv) A \emph{monad} with cohomology sheaf $\scf$ is a bounded 
complex $K^\bullet$ (usually, with only three non-zero terms) of vector bundles 
on $\p^n$ such that $\sch^0(K^\bullet) \simeq \scf$ and $\sch^i(K^\bullet) = 0$ 
for $i \neq 0$. For \emph{Horrocks monads}, see Barth and Hulek \cite{bh}.

\section{Preliminaries}\label{S:prelim}

This section is devoted to recalling some well known facts about stable rank 3 
vector bundles with $c_1 = 0$ on $\pii$ and $\piii$. In particular, we recall
the definition and properties of the \emph{spectrum} of a stable rank 3 vector
bundle on $\piii$ and Beilinson's theorem. Then we introduce the new
ingredient used in the proof of Thm.~\ref{T:c10c23}, namely the morphism
$\mu \colon \tH^1(E^\prim(-1)) \otimes \sco_{\p^{3 \vee}}(-1) \ra \tH^1(E^\prim)
\otimes \sco_{\p^{3 \vee}}$
deduced from the multiplication map
$\tH^1(E^\prim(-1)) \otimes S_1 \ra \tH^1(E^\prim)$, where $E^\prim$ is a stable
rank 3 vector bundle on $\piii$ with $c_1 = 0$, $c_2 = 3$.
We conclude the section by recalling some basic facts about moduli spaces. 

\begin{lemma}\label{L:h0fleq2} 
Let $F$ be a semistable rank $3$ vector bundle on $\pii$ with Chern classes 
$c_1 = 0$ and $c_2 \geq 1$. Then ${\fam0 h}^0(F) \leq 2$ and if 
${\fam0 h}^0(F) = 2$ then $F$ can be realized as an extension$\, :$ 
\[
0 \lra 2\sco_\pii \lra F \lra \sci_Z \lra 0\, , 
\]
for some $0$-dimensional subscheme $Z$ of $\pii$. 
\end{lemma}

\begin{proof} 
Since $F$ has rank 3 and $c_1 = 0$, it is semistable if and only if 
$\tH^0(F(-1)) = 0$ and $\tH^0(F^\vee(-1)) = 0$. It is well known that such a 
bundle has $c_2 \geq 0$ and if $c_2 = 0$ then $F \simeq 3\sco_\pii$. 

Now, under the hypothesis of the lemma, assume that $\h^0(F) \geq 2$ and 
consider two linearly independent global sections $s_1$ and $s_2$ of $F$. 
We assert that the global section $s_1 \wedge s_2$ of $\bigwedge^2F$ is 
non-zero. \emph{Indeed}, since $\tH^0(F(-1)) = 0$ the zero scheme $Z_1$ of 
$s_1$ has codimension at least 2 in $\pii$. If $s_1 \wedge s_2 = 0$ then there 
exists a regular function $f$ on $\pii \setminus Z_1$ such that $s_2 = fs_1$. 
But the only regular functions on $\pii \setminus Z_1$ are the constant ones 
hence $s_1$ and $s_2$ are linearly dependent, which \emph{contradicts} our 
assumption. 

We have $\bigwedge^2F \simeq F^\vee$. Since $\tH^0(F^\vee(-1)) = 0$, the zero 
scheme $Z$ of the global section $s_1 \wedge s_2$ of $\bigwedge^2F$ has 
codimension at least 2 in $\pii$. It follows that the \emph{Eagon-Northcott 
complex}$\, :$ 
\[
0 \lra 2\sco_\pii \xra{(s_1\, ,\, s_2)} F \xra{s_1 \wedge s_2 \wedge \ast} 
\sci_Z \lra 0 
\]
is exact. Since $\text{length}\, Z = c_2 \geq 1$ one deduces that $\h^0(F) = 
2$. 
\end{proof}

\begin{lemma}\label{L:gms}
Let $F$ be a rank $3$ vector bundle on $\pii$ with $c_1 = 0$ and such that
${\fam0 H}^0(F(-2)) = 0$ and ${\fam0 H}^0(F^\vee(-2)) = 0$. Then, for the
general line $L \subset \pii$, one has $F_L \simeq 3\sco_L$ or
$F_L \simeq \sco_L(1) \oplus \sco_L \oplus \sco_L(-1)$. 
\end{lemma}

\begin{proof}
If $\tH^0(F(-1)) = 0$ and $\tH^0(F^\vee(-1)) = 0$ then $F$ is semistable and
one can apply the theorem of Grauert-M\"{u}lich-Spindler \cite{sp}. If
$\tH^0(F^\vee(-1)) \neq 0$ or $\tH^0(F(-1)) \neq 0$ then one has an exact
sequence$\, :$
\[
0 \lra G(1) \lra F^\prim \lra \sci_Z(-1) \lra 0\, , 
\]
where $F^\prim$ is either $F$ or $F^\vee$, $Z$ is a 0-dimensional (or empty)
closed subscheme of $\pii$ and $G$ is a rank 2 vector bundle with
$c_1(G) = -1$. Since $\tH^0(F^\prim (-2)) = 0$ it follows that
$\tH^0(G(-1)) = 0$. If $\tH^0(G) = 0$ then $G$ is stable. If $\tH^0(G) \neq 0$
then $G$ can be realized as an extension
$0 \ra \sco_\pii \ra G \ra \sci_W(-1) \ra 0$, for some 0-dimensional subscheme
$W$ of $\pii$. Using the theorem of
Grauert-M\"{u}lich (for the former case) one gets that, for the general line
$L \subset \pii$ one has $G_L \simeq \sco_L \oplus \sco_L(-1)$. If, moreover,
$L \cap Z = \emptyset$ then
$F^\prim_L \simeq \sco_L(1) \oplus \sco_L \oplus \sco_L(-1)$. 
\end{proof}

\begin{remark}\label{R:c1scf}
We recall, here, a formula that we shall need a couple of times. If $\scf$ is 
a coherent torsion sheaf on $\p^n$ then$\, :$ 
\[
c_1(\scf) = {\textstyle \sum}(\text{length}\, \scf_\xi)\, \text{deg}\, X\, ,
\]
where the sum is indexed by the 1-codimensional irreducible components $X$ of 
$\text{Supp}\, \scf$ and $\xi$ is the generic point of $X$ (notice that 
$\scf_\xi$ is an Artinian $\sco_{\p^n , \xi}$-module).

\emph{Indeed}, let $\Sigma$ be the support of $\scf$ endowed with the reduced
subscheme structure. For any integer $i \geq 0$,
$\sci_\Sigma^i\scf/\sci_\Sigma^{i+1}\scf$ is an $\sco_\Sigma$-module hence there
exists a dense open subset $\Sigma_i$ of $\Sigma$ such that the restriction
of $\sci_\Sigma^i\scf/\sci_\Sigma^{i+1}\scf$ to $\Sigma_i$ is locally free.
$\text{length}\, (\sci_\Sigma^i\scf/\sci_\Sigma^{i+1}\scf)_\xi$ is the rank of this
locally free sheaf arround $\xi$. One can restrict, now, $\scf$ to a general
line $L \subset \p^n$. 
\end{remark}

\begin{remark}\label{R:spehv} 
Let $E$ be a stable rank 3 vector bundle on $\piii$ with $c_1 = 0$.
We recall that a plane $H_0 \subset \piii$ is an \emph{unstable plane} for $E$
if $\tH^0(E_{H_0}^\vee(-1)) \neq 0$. The largest integer $r \geq 1$ for which 
$\tH^0(E_{H_0}^\vee(-r)) \neq 0$ is the \emph{order} of $H_0$. 
Ein, Hartshorne and Vogelaar show, in \cite[Prop.~5.1]{ehv}, that the 
following conditions are equivalent$\, :$ 
\begin{enumerate} 
\item[(i)] $\tH^0(E_H) \neq 0$ for every plane $H$ and there is an unstable 
plane for $E\, ;$ 
\item[(ii)] There exists a plane $H_0 \subset \piii$ and an exact
sequence$\, :$ 
\[
0 \lra \Omega_\piii(1) \lra E \lra \sco_{H_0}(-c_2+1) \lra 0\, ; 
\]   
\item[(iii)] $E$ has an unstable plane of order $c_2 - 1$. 
\end{enumerate}
If the above conditions are satisfied then$\, :$ (iv) $c_3 = c_2^2 - c_2$. 
Moreover, if $c_2 \geq 4$ then (iv)\, $\Rightarrow$\, (iii) (hence all four 
conditions are equivalent). 
We assert that conditions (i)--(iii) above are also equivalent (for
$c_2 \geq 2$) to the condition$\, :$ 
\begin{enumerate} 
\item[(v)] There is a non-zero morphism $\phi \colon \Omega_\piii(1) \ra E$. 
\end{enumerate} 

\noindent 
\emph{Indeed}, let us show that (v)\, $\Rightarrow$\, (ii). Since 
$\Omega_\piii(1)$ and $E$ are stable vector bundles with
$c_1(\Omega_\piii(1)) = -1$ and $c_1(E) = 0$, $\phi$ must have, generically,
rank 3. It follows that 
$\bigwedge^3\phi \colon \sco_\piii(-1) \ra \sco_\piii$ is defined by a non-zero 
linear form $h_0$. Let $H_0 \subset \piii$ be the plane of equation $h_0 = 0$. 
$\Cok \phi$ is annihilated by $h_0$ hence it is an $\sco_{H_0}$-module. The 
Auslander-Buchsbaum relation shows that $\text{depth}(\Cok \phi)_x \geq 2$, 
$\forall \, x \in H_0$, hence $\Cok \phi$ is a locally free $\sco_{H_0}$-module. 
One deduces, from Remark~\ref{R:c1scf}, that it has rank 1, i.e., that 
$\Cok \phi \simeq \sco_{H_0}(a)$ for some $a \in \z$. One has
$1 = c_2(\Omega_\piii(1)) = c_2 + a$ hence $a = -c_2 + 1$. 

Notice that condition (v) above is equivalent to the existence of a non-zero 
element $\xi$ of $\tH^1(E(-1))$ such that $S_1\xi = 0$ in $\tH^1(E)$ (use the 
exact sequence $0 \ra \Omega_\piii(1) \ra S_1 \otimes \sco_\piii \ra 
\sco_\piii(1) \ra 0$). 
\end{remark}

\begin{lemma}\label{L:jumpinglines}
Let $E$ be a stable rank $3$ vector bundle on $\piii$ with $c_1 = 0$, $c_2 = 3$
and $c_3 \geq 0$. Assume that ${\fam0 H}^0(E_H) = 0$ for the general plane
$H \subset \piii$. Assume, also, that any line $L \subset \piii$ for which
${\fam0 h}^1(E_L^\vee) \geq 1$ either passes through one of finitely many
ponts, or is contained in one of finitely many planes, or belongs to a family
of dimension at most $1$. Then, for the general plane $H \subset \piii$, one
has ${\fam0 H}^0(E_H^\vee) = 0$. 
\end{lemma}

\begin{proof}
$E$ does not satisfy the equivalent conditions (i)--(iii) from
Remark~\ref{R:spehv} (because $\tH^0(E_H) = 0$ for the general plane
$H \subset \piii$) and nor does $E^\vee$ (because $c_3 \neq -6$). One
deduces that, for any plane $H \subset \piii$, one has $\tH^0(E_H^\vee(-2)) = 0$
and $\tH^0(E_H(-2)) = 0$. Lemma~\ref{L:gms} implies that, for any plane $H$,
the family of lines $L \subset H$ for which $\h^1(E_L^\vee) \geq 1$ has
dimension at most 1. It follows that the general plane $H \subset \piii$
contains no line $L_0$ for which $\h^1(E_{L_0}^\vee) \geq 1$. As we noticed in
the Introduction (right after the statement of Thm.~\ref{T:c10c23}), this
implies the conclusion of the lemma. 
\end{proof}

\begin{remark}\label{R:spectrum}
Let $E$ be a stable rank 3 vector bundle on $\piii$ with $c_1 = 0$. One of the 
most important application of Thm.~\ref{T:ehv} (and of the generalized
Grauert-M\"{u}lich theorem of Spindler \cite{sp}) is the existence of a
non-decreasing sequence of integers $k_E = (k_1 , \ldots , k_m)$, called the
\emph{spectrum} of $E$, such that, 
putting $K := \bigoplus_{i = 1}^m\sco_\pj(k_i)$, one has$\, :$ 
\begin{enumerate} 
\item[(I)] $\h^1(E(l)) = \h^0(K(l+1)$, for $l \leq -1\, ;$ 
\item[(II)] $\h^2(E(l)) = \h^1(K(l+1))$, for $l \geq -3$. 
\end{enumerate}
We shall need the following properties of the spectrum$\, :$ $m = c_2$, 
$-2\Sigma k_i = c_3$, the spectrum is connected, i.e., $k_{i+1} - k_i \leq 1$,
for $1 \leq i \leq m-1$, and if 0 does not occur in the spectrum then either 
$k_{m-2} = k_{m-1} = k_m = -1$ or $k_1 = k_2 = k_3 = 1$. Moreover, by Serre
duality, the spectrum of $E^\vee$ is $(-k_m , \ldots , -k_1)$. 
Details can be found in the
papers of Okonek and Spindler \cite{oks2}, \cite{oks3}, and in the papers
\cite{c1}, \cite{c2} of the author. These papers use the approach of
Hartshorne \cite{ha}, \cite{ha2} who treated the rank 2 case. Details can be
also found in Appendix~\ref{A:spectrum}. 

It is easy to show that conditions (i)--(iii) from Remark~\ref{R:spehv} are
also equivalent to the condition$\, :$
\begin{enumerate}
\item[(vi)] The spectrum of $E$ is $(-c_2+1, \ldots , -1 , 0)$. 
\end{enumerate}
\emph{Indeed}, (ii) $\Rightarrow$ (vi) by the above definition of the
spectrum. On the other hand, (vi) $\Rightarrow$ (iii) because the spectrum of
$E^\vee$ is $(0 , 1 , \ldots , c_2 - 1)$ hence $\h^1(E^\vee(-c_2)) = 1$ and
$\h^1(E^\vee(-c_2 + 1)) = 3$ hence there exists a non-zero linear form $h_0$
such that multiplication by
$h_0 \colon \tH^1(E^\vee(-c_2)) \ra \tH^1(E^\vee(-c_2+1))$ is the zero map which
implies that $\h^0(E_{H_0}^\vee(-c_2+1)) = 1$, $H_0$ being the plane of equation
$h_0 = 0$. 
\end{remark}

\begin{remark}\label{R:beilinson} 
Let $E$ be a stable rank 3 vector bundle on $\piii$ with $c_1 = 0$, $c_2 = 3$. 
According to the previous remark, if $c_3 = 6$ then the possible 
spectra of $E$ are $(-2, -1 , 0)$ and $(-1 , -1 , -1)$, if $c_3 = 4$ the
spectrum of $E$ is $(-1 , -1 , 0)$, if $c_3 = 2$ the spectrum of $E$ is
$(-1 , 0 , 0)$, and if $c_3 = 0$ the possible spectra of $E$ are $(0 , 0 , 0)$
and $(-1 , 0 , 1)$.
Assuming that neither $E$ nor $E^\vee$ satisfy the equivalent conditions
(i)--(iii) from Remark~\ref{R:spehv} (i.e., that the spectrum of $E$ 
is neither $(-2 , -1 , 0)$ nor $(0 , 1 , 2)$),
one has $\tH^1(E(l)) = 0$ for $l \leq -3$ and $\tH^2(E(l)) = 0$ 
for $l \geq -1$. In this case, \emph{Beilinson's theorem} \cite{bei}, 
with the improvements of Eisenbud, Fl\o ystad and Schreyer \cite[(6.1)]{efs} 
(these results are recalled in \cite[1.23--1.25]{acm1}), implies  
that $E$ is the cohomology sheaf of a monad that can be described as the 
total complex of a double complex with the following (possibly) non-zero 
terms$\, :$ 
\[
\SelectTips{cm}{12}\xymatrix{{\tH^1(E(-2)) \otimes \Omega_\piii^2(2)}\ar[r] & 
{\tH^1(E(-1)) \otimes \Omega_\piii^1(1)}\ar[r] & {\tH^1(E) \otimes \sco_\piii}\\ 
0\ar[r]\ar[u] & {\tH^2(E(-3)) \otimes \Omega_\piii^3(3)}\ar[r]\ar[u] & 
{\tH^2(E(-2)) \otimes \Omega_\piii^2(2)}\ar[u]} 
\]
such that the term $\tH^1(E(-1)) \otimes \Omega_\piii^1(1)$ has bidegree 
$(0,0)$. The horizontal differentials of this double complex are equal to 
$\sum_{i = 0}^3X_i \otimes e_i$, $X_i$ acting to the left on $\tH^p(E(-l))$ via 
the $S$-module structure of $\tH^p_\ast(E)$ and $e_i$ acting to the right on 
$\Omega_\piii^l(l)$ by contraction (recall that $\Omega_\piii^l(l)$ embeds 
canonically into $\sco_\piii \otimes \bigwedge^lV^\vee$). 
\end{remark}

\begin{lemma}\label{L:s1n-1}
Let $E$ be a stable rank $3$ vector bundle on $\piii$ with $c_1 = 0$,
$c_2 = 3$. Assume that neither $E$ nor $E^\vee$ satisfy the equivalent
conditions \emph{(i)--(iii)} from Remark~\emph{\ref{R:spehv}}. If
${\fam0 H}^2(E(-2)) = 0$ then, for any vector subspace $N_{-1}$ of codimension
$1$ of ${\fam0 H}^1(E(-1))$, one has $S_1N_{-1} = {\fam0 H}^1(E)$. 
\end{lemma}

\begin{proof}
Since $\tH^2(E(-2)) = 0$, the Beilinson monad of $E$ shows that the morphism
$\delta \colon \tH^1(E(-1)) \otimes \Omega_\piii(1) \ra \tH^1(E) \otimes
\sco_\piii$ deduced from the map
$\tH^1(E(-1)) \otimes S_1 \ra \tH^1(E)$ is an epimorphism. Assume, by
contradiction, that there exists a proper subspace $N_0$ of $\tH^1(E)$ such
that $S_1N_{-1} \subseteq N_0$. Then $\delta$ maps
$N_{-1} \otimes \Omega_\piii(1)$ into $N_0 \otimes \sco_\piii$. Since there is no
epimorphism $\Omega_\piii(1) \ra \sco_\piii$ one gets a \emph{contradiction}. 
\end{proof}

\begin{remark}\label{R:mu} 
Let $E$ be a stable rank 3 vector bundle on $\piii$ with $c_1 = 0$, $c_2 = 3$.
The new ingredient that we use in the proof of Thm.~\ref{T:c10c23} is the
analysis of the map$\, :$
\[
\mu \colon \tH^1(E^\prim(-1)) \otimes \sco_{\p^{3\vee}}(-1) \lra \tH^1(E^\prim)
\otimes \sco_{\p^{3\vee}} 
\]
deduced from the multiplication map
$\tH^1(E^\prim(-1)) \otimes S_1 \ra \tH^1(E^\prim)$, where $E^\prim$ is either
$E$ or $E^\vee$. Notice that the kernel 
of the reduced stalk of $\mu$ at a point $[h] \in \p^{3 \vee}$ corresponding to 
a plane $H \subset \piii$ of equation $h = 0$ is isomorphic to
$\tH^0(E_H^\prim)$. 

If the spectrum of $E^\prim$ is not $(-2 , -1 , 0)$ then
$\tH^2(E^\prim(l)) = 0$ for 
$l \geq -1$ hence, by Riemann-Roch, $\tH^1(E^\prim(-1))$ and $\tH^1(E^\prim)$
have the same dimension, namely $d := 3 - c^\prime_3/2$ (with
$c_3^\prime = \pm c_3$). If one fixes $k$-bases of 
$\tH^1(E^\prim(-1))$ and of $\tH^1(E^\prim)$ then
$\tH^0(\mu(1)) \colon \tH^1(E^\prim(-1)) \ra \tH^1(E^\prim) \otimes
\tH^0(\sco_{\p^{3\vee}}(1))$ 
is represented by a $d \times d$ matrix
$\mathcal{M}$ with entries in $\tH^0(\sco_{\p^{3\vee}}(1)) = V$. Then the 
$d \times d$ matrix $M_i$ with scalar entries associated to the multiplication 
map $X_i \colon \tH^1(E^\prim(-1)) \ra \tH^1(E^\prim)$ is obtained by evaluating
$X_i$ at the entries of $\mathcal{M}$. It follows that
$\mathcal{M} = \sum_{i = 0}^3M_ie_i$. 

Notice that the same matrix $\mathcal{M}$ defines the horizontal differential 
$\tH^1(E(-1)) \otimes \Omega_\piii^1(1) \ra \tH^1(E) \otimes \sco_\piii$ of the 
Beilinson monad of $E$ (recall that 
$\text{Hom}_{\sco_\piii}(\Omega_\piii^1(1) , \sco_\piii)$ can be identified with 
$V$).

Assume, now, that $\tH^0(E_H^\prim) \neq 0$, for every plane $H \subset \piii$.
Then $\mu$ has, generically, corank 1 (by the theorem of Spindler \cite{sp2}
recalled in the Introduction). Since $\Ker \mu$ is reflexive of rank 1 it must
be invertible, i.e., $\Ker \mu \simeq \sco_{\p^{3 \vee}}(a)$, for some integer
$a$. On the other hand, one must have an exact sequence$\, :$
\[
0 \lra (\Cok \mu)_{\text{tors}} \lra \Cok \mu \lra \sci_Y(b)
\lra 0\, , 
\]
for some integer $b$ and some closed subscheme $Y$ of $\p^{3 \vee}$, of
codimension at least 2. One has the relation$\, :$
\begin{equation}\label{E:ab}
a = -d + b + c_1\left((\Cok \mu)_{\text{tors}}\right)\, . 
\end{equation}
By Remark~\ref{R:c1scf}, $c_1\left((\Cok \mu)_{\text{tors}}\right) \geq 0$. We
conclude with three easy observations.

(i) One has $a \leq -2$. \emph{Indeed}, if $a = -1$ then there exists a
non-zero element $\xi$ of $\tH^1(E^\prim(-1))$ such that $S_1\xi = (0)$ in
$\tH^1(E^\prim)$ which \emph{contradicts} the last assertion in
Remark~\ref{R:spehv}.

(ii) If $S_1\tH^1(E^\prim(-1)) = \tH^1(E^\prim)$ then $b \geq 1$.
\emph{Indeed}, if $b = 0$ then one must have $Y = \emptyset$. The kernel of
the epimorphism $\tH^1(E^\prim) \otimes \sco_{\p^{3 \vee}} \ra \sco_{\p^{3 \vee}}$ has
the form $N_0 \otimes \sco_{\p^{3 \vee}}$, for some 1-codimensional subspace
$N_0$ of $\tH^1(E^\prim)$. Since the image of $\mu$ is contained in
$N_0 \otimes \sco_{\p^{3 \vee}}$ it follows that
$S_1\tH^1(E^\prim(-1)) \subseteq N_0$, which \emph{contradicts} our assumption.

(iii) Let $L$ be a line and $\piii$ and let $L^\vee$ be the line in
$\p^{3 \vee}$ whose points correspond to the planes containing $L$. The
restriction
$\mu \vb L^\vee \colon \tH^1(E^\prim(-1)) \otimes \sco_{L^\vee}(-1) \ra
\tH^1(E^\prim) \otimes \sco_{L^\vee}$
is defined by the multiplication map
$\tH^1(E^\prim(-1)) \otimes \tH^0(\sci_L(1)) \ra \tH^1(E^\prim)$. If
$\tH^2(E^\prim(-2)) = 0$ then, tensorizing by $E^\prim$ the exact sequence
$0 \ra \sco_\piii(-2) \ra 2\sco_\piii(-1) \ra \sco_\piii \ra \sco_L \ra 0$, one
gets an exact sequence$\, :$
\[
\tH^1(E^\prim(-1)) \otimes \tH^0(\sci_L(1)) \lra \tH^1(E^\prim) \lra
\tH^1(E^\prim_L) \lra 0\, . 
\]
If $\h^1(E^\prim_L) \geq 1$ and if $L$ is contained in a plane $H$ with
$\h^0(E^\prim_H) = 1$ (hence $\h^1(E^\prim_H) = 1$, by Riemann-Roch) one deduces
that $\Cok (\mu \vb L^\vee) \simeq \sco_{L^\vee} \oplus \sct$, where $\sct$ is a
torsion $\sco_{L^\vee}$-module. This implies that if $b \geq 1$ then $L$ must
intersect $Y$. If, moreover, the general plane $H \subset \piii$ contains
a line $L_0$ with $\h^1(E_{L_0}^\prim) \geq 1$ then $\dim Y = 1$. 
\end{remark}

\begin{remark}\label{R:moduli}  
We denote, for $n \in \{0,\, 2,\, 4,\, 6\}$, by $\text{M}(n)$ the moduli space 
of stable rank 3 vector bundles on $\piii$ with Chern classes $c_1 = 0$, 
$c_2 = 3$, $c_3 = n$. Let $E$ be such a bundle and let $[E]$ be the 
corresponding point of $\text{M}(n)$. It is well known that the tangent space 
$\text{T}_{[E]}\text{M}(n)$ of $\text{M}(n)$ at $[E]$ is canonically isomorphic 
to $\tH^1(E^\vee \otimes E)$ and that every irreducible component of 
$\text{M}(n)$ containing $[E]$ has dimension
$\geq \h^1(E^\vee\otimes E) - \h^2(E^\vee \otimes E)$.
In particular, if $\tH^2(E^\vee \otimes E) = 0$ then 
$\text{M}(n)$ is nonsingular at $[E]$, of (local) dimension
$\h^1(E^\vee \otimes E)$.
These results can be deduced from the work of Wehler \cite{w} 
(see Huybrechts and Lehn \cite[Cor.~4.5.2]{hl}). 

Since $E$ is stable, one has $\h^0(E^\vee \otimes E) = 1$ and $\h^3(E^\vee 
\otimes E) = \h^0((E \otimes E^\vee)(-4)) = 0$ hence$\, :$ 
\[
\h^1(E^\vee\otimes E) - \h^2(E^\vee \otimes E) = 1 - \chi(E^\vee \otimes E)\, . 
\]
Since $E^\vee \otimes E$ is selfdual, one has $c_i(E^\vee \otimes E) = 0$ for 
$i$ odd. Moreover, $c_2(E^\vee \otimes E)$ depends only on $c_1(E)$ and 
$c_2(E)$ (restrict to a plane). If $E_0$ is the bundle from 
Theorem~\ref{T:c10c23}(b)(i) then $\chi(E_0^\vee \otimes E_0) = -27$ because 
$E_0^\vee \otimes E_0$ is the cohomology sheaf of a monad of the form$\, :$ 
\[
0 \lra 18\sco_\piii(-1) \lra 45\sco_\piii \lra 18\sco_\piii(1) \lra 0\, . 
\]
Consequently, if $E$ is an arbitrary rank 3 vector bundle with $c_1 = 0$, 
$c_2 = 3$ and any $c_3$ then $\chi(E^\vee \otimes E) = -27$. One deduces that 
the \emph{expected dimension} of $\text{M}(n)$ is 28.  
\end{remark}

\section{The case $c_3 = 0$}\label{S:c30} 

\begin{lemma}\label{L:c30000monads} 
Let $E$ be a stable rank $3$ vector bundle on $\piii$ with $c_1 = 0$, $c_2 = 
3$, $c_3 = 0$ and spectrum $(0 , 0 , 0)$. Then $E$ is the cohomology sheaf of a 
Beilinson monad of the form$\, :$ 
\[
0 \lra 3\Omega_\piii^3(3) \overset{\gamma}{\lra} 3\Omega_\piii^1(1) 
\overset{\delta}{\lra} 3\sco_\piii \lra 0\, , 
\]
and of a Horrocks monad of the form$\, :$ 
\[
0 \lra 3\sco_\piii(-1) \overset{\alpha}{\lra} 9\sco_\piii 
\overset{\beta}{\lra} 3\sco_\piii(1) \lra 0\, . 
\]
\end{lemma}

\begin{proof} 
One has, by Riemann-Roch, $\h^1(E) = 3$. For the first monad see 
Remark~\ref{R:beilinson} while the second monad can be deduced from the first 
one and the exact sequence $0 \ra \Omega_\piii^1(1) \ra 4\sco_\piii \ra 
\sco_\piii(1) \ra 0$. 
\end{proof}

\begin{proposition}\label{P:c30000}  
Let $E$ be a stable rank $3$ vector bundle on $\piii$ with $c_1 = 0$,
$c_2 = 3$, $c_3 = 0$ and spectrum $(0 , 0 , 0)$. Then the restriction of $E$
to a general plane is stable. 
\end{proposition}

\begin{proof} 
Since $E^\vee$ has the same Chern classes and spectrum as $E$, it suffices to 
show that, for the general plane $H \subset \piii$, one has $\tH^0(E_H) = 0$. 
Assume, by contradiction, that $\tH^0(E_H) \neq 0$, for every plane $H$. 
Then the morphism $\mu \colon \tH^1(E(-1)) \otimes \sco_{\p^{3\vee}}(-1) \ra
\tH^1(E) \otimes \sco_{\p^{3\vee}}$
from Remark~\ref{R:mu} has, generically, corank 1. Using the notation from
that remark, one has $a \leq -2$ (by observation (i)) and $b \geq 1$ (by
observation (ii) because, using the Horrocks monad of $E$, one sees that
$S_1\tH^1(E(-1)) = \tH^1(E)$).
It follows, from relation \eqref{E:ab} (in which one has
$d = 3$), that $a = -2$ and $b = 1$. Since $\sci_Y(1)$ is globally generated
by at most 3 linear forms, $Y$ must be a point or a line. $Y$ cannot be a
point because, by Lemma~\ref{L:h0fleq2}, $\mu$ has corank $\leq 2$ at every
point of $\p^{3 \vee}$ ($E_H$ is semistable for every plane $H$ because
$\tH^1(E(-2)) = 0$ and $\tH^1(E^\vee(-2)) = 0$).

$Y$ cannot be a line, either. \emph{Indeed}, if $Y$ is a line then the kernel
of the epimorphism $\tH^1(E) \otimes \sco_{\p^{3 \vee}} \ra \sci_Y(1)$ is
isomorphic to $\sco_{\p^{3 \vee}} \oplus \sco_{\p^{3 \vee}}(-1)$, hence $\mu$
factorizes as$\, :$
\[
\tH^1(E(-1)) \otimes \sco_{\p^{3 \vee}}(-1) \overset{\overline{\mu}}{\lra}
\sco_{\p^{3 \vee}} \oplus \sco_{\p^{3 \vee}}(-1) \lra \tH^1(E) \otimes
\sco_{\p^{3 \vee}}\, . 
\]
The kernel of the component
$\tH^1(E(-1)) \otimes \sco_{\p^{3 \vee}}(-1) \ra \sco_{\p^{3 \vee}}(-1)$ of
$\overline{\mu}$ has the form $N_{-1} \otimes \sco_{\p^{3 \vee}}(-1)$, for some
1-codimensional subspace $N_{-1}$ of $\tH^1(E(-1))$. The direct summand
$\sco_{\p^{3 \vee}}$ of $\sco_{\p^{3 \vee}} \oplus \sco_{\p^{3 \vee}}(-1)$ corresponds
to a 1-dimensional subspace $N_0$ of $\tH^1(E)$. Since $\mu$ maps
$N_{-1} \otimes \sco_{\p^{3 \vee}}(-1)$ into $N_0 \otimes \sco_{\p^{3 \vee}}$, it
follows that $S_1N_{-1} \subseteq N_0$, which \emph{contradicts}
Lemma~\ref{L:s1n-1}. 
\end{proof} 

\begin{lemma}\label{L:c30h0eh2}  
Let $E$ be a stable rank $3$ vector bundle on $\piii$ with $c_1 = 0$, $c_2 = 
3$, $c_3 = 0$ and spectrum $(0 , 0 , 0)$. Then the planes $H \subset \piii$ 
for which ${\fam0 h}^0(E_H) = 2$ form a closed subset of $\p^{3 \vee}$ of 
dimension $\leq 1$. 
\end{lemma}

\begin{proof} 
We showed, in Prop.~\ref{P:c30000}, that $\h^0(E_H) = 0$, for the general plane 
$H \subset \piii$. Moreover, for any plane $H$, $E_H$ is semistable (because 
$\tH^1(E(-2)) = 0$ and $\tH^1(E^\vee(-2)) = 0$) hence $\h^0(E_H) \leq 2$, 
by Lemma~\ref{L:h0fleq2}. Assume, by contradiction, that there exists a closed 
(reduced and) irreducible surface $\Sigma \subset \p^{3 \vee}$ such that, for 
every $[h] \in \Sigma$, if $H \subset \piii$ is the plane of equation $h = 0$ 
then $\h^0(E_H) = 2$. Consider, as in the proof of Prop.~\ref{P:c30000}, the 
morphism $\mu \colon \tH^1(E(-1)) \otimes \sco_{\p^{3\vee}}(-1) \ra \tH^1(E) 
\otimes \sco_{\p^{3\vee}}$ from Remark~\ref{R:mu}. Then $\mu$ has rank 1 at every 
point of $\Sigma$ hence the image of $\mu_\Sigma \colon \tH^1(E(-1)) \otimes 
\sco_\Sigma(-1) \ra \tH^1(E) \otimes \sco_\Sigma$ 
is a subbundle $\scl$ of rank 1 of the trivial bundle 
$\tH^1(E) \otimes \sco_\Sigma$. 

Now, all the $2 \times 2$ minors of the $3 \times 3$ matrix defining $\mu$ 
vanish on $\Sigma$ hence $\Sigma$ is either a nonsingular quadric surface, or 
a quadric cone, or a plane. It follows that either $\scl \simeq 
\sco_\Sigma(-1)$, or $\scl \simeq \sco_\Sigma$, or $\Sigma$ is a nonsingular 
quadric surface and $\scl \simeq \sco_\Sigma(-1 , 0)$ or $\scl \simeq 
\sco_\Sigma(0 , -1)$. 

\vskip2mm 

\noindent 
$\bullet$\quad If $\text{deg}\, \Sigma = 2$ and $\scl \simeq \sco_\Sigma(-1)$ 
or $\Sigma$ is a nonsigular quadric surface and $\scl \simeq 
\sco_\Sigma(-1 , 0)$ or $\scl \simeq \sco_\Sigma(0 , -1)$ then there exists a 
direct summand $\sco_\Sigma(-1)$ of $\tH^1(E(-1)) \otimes \sco_\Sigma(-1)$ such 
that $\mu_\Sigma$ vanishes on it. This direct summand corresponds to a non-zero 
element $\xi$ of $\tH^1(E(-1))$ such that $h\xi = 0$ in $\tH^1(E)$, $\forall 
\, [h] \in \Sigma$. It follows that $S_1\xi = (0)$ in $\tH^1(E)$, and this 
\emph{contradicts} the last part of Remark~\ref{R:spehv}. 

\noindent
$\bullet$\quad If $\text{deg}\, \Sigma = 2$ and $\scl \simeq \sco_\Sigma$ then 
there exists a 1-dimensional subspace $N_0$ of $\tH^1(E)$ such that the image 
of $\mu_\Sigma$ is $N_0 \otimes \sco_\Sigma$. It follows that $h\tH^1(E(-1)) 
\subseteq N_0$, $\forall \, [h] \in \Sigma$, hence $S_1\tH^1(E(-1)) \subseteq 
N_0$, which \emph{contradicts} the fact that $S_1\tH^1(E(-1)) = \tH^1(E)$ (use 
the Horrocks monad of $E$ from Lemma~\ref{L:c30000monads}). 

\noindent 
$\bullet$\quad It thus remains to consider the following two cases$\, :$ 

\vskip2mm 

\noindent 
{\bf Case 1.}\quad $\Sigma$ \emph{is a plane and} $\scl \simeq 
\sco_\Sigma(-1)$. 

\vskip2mm 

\noindent 
In this case, $\Sigma = \p((V/kv_3)^\vee) \subset \p(V^\vee) = \p^{3 \vee}$, for 
some non-zero vector $v_3$ of $V$. Choosing convenient bases of $\tH^1(E)$ and 
$\tH^1(E(-1))$, $\mu_\Sigma$ is represented by a matrix of the form$\, :$ 
\[
\begin{pmatrix} 
{\widehat v}_0 & 0 & 0\\
{\widehat v}_1 & 0 & 0\\
{\widehat v}_2 & 0 & 0 
\end{pmatrix}
\]
where $v_0,\, v_1,\, v_2$ are elements of $V$ such that their 
classes$\pmod{kv_3}$ form a $k$-basis of $V/kv_3$, i.e., such that $v_0,\, 
v_1,\, v_2,\, v_3$ is a $k$-basis of $V$. It follows that $\mu$ itself is 
represented by a matrix of the form$\, :$ 
\[
\begin{pmatrix} 
v_0 + c_{00}v_3 & c_{01}v_3 & c_{02}v_3\\ 
v_1 + c_{10}v_3 & c_{11}v_3 & c_{12}v_3\\ 
v_2 + c_{20}v_3 & c_{21}v_3 & c_{22}v_3
\end{pmatrix}
\]
with $c_{ij} \in k$. 
Now, according to Remark~\ref{R:beilinson}, $E$ is the cohomology bundle of a 
Beilinson monad of the form$\, :$ 
\[
0 \lra \tH^2(E(-3)) \otimes \Omega_\piii^3(3) \overset{\gamma}{\lra} 
\tH^1(E(-1)) \otimes \Omega_\piii^1(1) \overset{\delta}{\lra} 
\tH^1(E) \otimes \sco_\piii \lra 0\, . 
\]
Moreover, as we noticed in Remark~\ref{R:mu}, $\delta$ and $\mu$ are defined 
by the same matrix with entries in $V$. 
Choose $(a_0 , a_1 , a_2) \in k^3 \setminus \{0\}$ such that 
$(a_0 , a_1 , a_2)(c_{0i} , c_{1i} , c_{2i})^{\text{t}} = 0$, $i = 1,\, 2$. 
The composition of the map
$(a_0 , a_1 , a_2) \colon 3\sco_\piii \ra \sco_\piii$ and of $\delta$ is not an
epimorphism (because there is no epimorphism $\Omega_\piii(1) \ra \sco_\piii$)
and this \emph{contradicts} the fact that $\delta$ is an epimorphism. 

\vskip2mm 

\noindent 
{\bf Case 2.}\quad $\Sigma$ \emph{is a plane and} $\scl \simeq \sco_\Sigma$. 

\vskip2mm 

\noindent 
Analogously, the differential $\delta$ of the Beilinson monad of $E$ is 
defined by a matrix of the form$\, :$ 
\[
\begin{pmatrix} 
v_0 + c_{00}v_3 & v_1 + c_{01}v_3 & v_2 + c_{02}v_3\\ 
c_{10}v_3 & c_{11}v_3 & c_{12}v_3\\ 
c_{20}v_3 & c_{21}v_3 & c_{22}v_3
\end{pmatrix}\, ,
\] 
for some basis $v_0 , \ldots , v_3$ of $V$. One gets a \emph{contradiction} as
in Case~1 by composing the map $(0 , 1 , 0) \colon 3\sco_\piii \ra \sco_\piii$
with $\delta$. 
\end{proof} 

\begin{lemma}\label{L:s1xi} 
Let $E$ be a stable rank $3$ vector bundle on $\piii$ with $c_1 = 0$,
$c_2 = 3$, $c_3 = 0$ and spectrum $(0 , 0 , 0)$. If $\xi$ is a general element
of ${\fam0 H}^1(E(-1))$ then $S_1\xi = {\fam0 H}^1(E)$. 
\end{lemma}

\begin{proof} 
Consider the following closed subset of $\p^{3 \vee} \times \p(\tH^1(E(-1))\, :$ 
\[
Z := \{([h] , [\xi]) \vb h\xi = 0 \text{ in } \tH^1(E)\}\, . 
\]
Prop.~\ref{P:c30000} and Lemma~\ref{L:c30h0eh2} give us information about 
the dimension of the fibres of the first projection $p \colon Z \ra \p^{3 \vee}$.
One deduces that $\dim Z \leq 2$. Considering, now, the second 
projection $q \colon Z \ra \p(\tH^1(E(-1))$, one sees that if $\xi$ is a 
general element of $\tH^1(E(-1))$ then $\dim q^{-1}([\xi]) \leq 0$. This means 
that $\dim_k\{h \in S_1 \vb h\xi = 0 \text{ in } \tH^1(E)\} \leq 1$ which 
implies that $S_1\xi = \tH^1(E)$. 
\end{proof} 

\begin{lemma}\label{L:beta} 
Let $E$ be a stable rank $3$ vector bundle on $\piii$ with $c_1 = 0$, $c_2 = 
3$, $c_3 = 0$ and spectrum $(0 , 0 , 0)$. Then the differential
$\beta \colon 9\sco_\piii \ra 3\sco_\piii(1)$
of the Horrocks monad of $E$ from 
Lemma~\emph{\ref{L:c30000monads}} is defined, up to automorphisms of 
$9\sco_\piii$ and $3\sco_\piii(1)$, by a matrix of the form$\, :$ 
\[
\begin{pmatrix} 
h_0 & h_1 & h_2^\prime & h_3^\prime & h_4^\prime & h_5^\prime & h_6^\prime & 
h_7^\prime & h_8^\prime\\ 
0 & h_0 & h_1 & h_2 & h_3 & 0 & 0 & 0 & 0\\
0 & 0 & 0 & 0 & 0 & h_0 & h_1 & h_2 & h_3
\end{pmatrix}\, , 
\]
with $h_0, \ldots , h_3$ a $k$-basis of $S_1$, and such that $h_2^\prime$ and 
$h_6^\prime$ belong to $kh_2 + kh_3$ and $h_i^\prime \in kh_1 + kh_2 + kh_3$ for 
$i \in \{3 , \ldots , 8\} \setminus \{6\}$.  
\end{lemma}

\begin{proof}
Put $K := \Ker \beta$. One has an exact sequence $0 \ra 
3\sco_\piii(-1) \overset{\overline \alpha}{\ra} K \ra E \ra 0$,
with 
$\overline \alpha$ induced by $\alpha \colon 3\sco_\piii(-1) \ra 9\sco_\piii$. 
One gets isomorphisms$\, :$ 
\[
\tH^0(3\sco_\piii) \Izo \tH^1(K(-1)) \Izo \tH^1(E(-1))\, ,\  
\Cok \tH^0(\beta) \Izo \tH^1(K) \Izo \tH^1(E)\, . 
\]
Choose a decomposition $3\sco_\piii \simeq \sco_\piii \oplus 2\sco_\piii$ such 
that $1 \in \tH^0(\sco_\piii)$ is mapped, by the first of the above 
isomorphisms, into an element $\xi$ of $\tH^1(E(-1))$ such that
$S_1\xi = \tH^1(E)$ (see Lemma~\ref{L:s1xi}). If
$\text{pr}_{23} \colon 3\sco_\piii(1) \ra 2\sco_\piii(1)$
is the projection onto the last two factors and 
$\beta_{23} := \text{pr}_{23} \circ \beta$, one sees easily that 
$\tH^0(\beta_{23}) \colon \tH^0(9\sco_\piii) \ra \tH^0(2\sco_\piii(1))$ is 
surjective. 

Consider, now, a decomposition
$9\sco_\piii \simeq \sco_\piii \oplus 
8\sco_\piii$ such that $\Ker \tH^0(\beta_{23}) = \tH^0(\sco_\piii)$.
Then, up to an automorphism of $8\sco_\piii$, $\beta$ is represented by a
matrix of the form$\, :$
\[
\begin{pmatrix} 
h_0^\prime & h_1^\prime & h_2^\prime & h_3^\prime & h_4^\prime & h_5^\prime & 
h_6^\prime & h_7^\prime & h_8^\prime\\ 
0 & h_0 & h_1 & h_2 & h_3 & 0 & 0 & 0 & 0\\
0 & 0 & 0 & 0 & 0 & h_0 & h_1 & h_2 & h_3
\end{pmatrix}\, , 
\]
with $h_0 , \ldots , h_3$ an arbitrary $k$-basis of $S_1$ (which will be, 
subsequently, modified). Since $\tH^0(E) = 0$, $\tH^0(\beta)$ is injective 
hence $h_0^\prime \neq 0$. Up to automorphisms of the two direct summands 
$4\sco_\piii$ of $8\sco_\piii$, one can assume that $h_0 = h_0^\prime$. 

Let $H_0 \subset \piii$ be the plane of equation $h_0 = 0$. One has, according 
to Lemma~\ref{L:h0fleq2}, $\h^0(E_{H_0}) \leq 2$. It follows that
$h_1^\prime \notin kh_0$ or
$h_5^\prime \notin kh_0$. We can assume that $h_1^\prime \notin kh_0$
(if $h_1^\prime \in kh_0$ and $h_5^\prime \notin kh_0$ then we transpose 
the second and the third rows of the matrix and then we transpose the columns 
$i$ and $i + 4$, $i = 2 , \ldots , 5$). Then again, by automorphisms of the 
two direct summands $4\sco_\piii$ of $8\sco_\piii$ fixing their first direct 
summand $\sco_\piii$, we can assume that, moreover, $h_1 = h_1^\prime$. 

Finally, substracting from the first row a linear combination of the second 
and the third rows, we can asssume that
$h_2^\prime ,\, h_6^\prime \in kh_0 + kh_2 + kh_3$
(this operation also modifies the second entry $h_1$ on the first 
row but we shall fix this immediately). Then substracting from the $i\, $th 
column, $2 \leq i \leq 9$, a convenient multiple of the first column, we can 
arrange that $h_1$ reappears as the second entry of the first row, that 
$h_2^\prime ,\, h_6^\prime \in kh_2 + kh_3$, and that $h_i^\prime \in kh_1 + kh_2 
+ kh_3$, for $i \in \{3 , \ldots , 8\} \setminus \{6\}$. 
\end{proof}

\begin{lemma}\label{L:h1erah1el}  
Let $E$ be a vector bundle on $\piii$ and $L \subset \piii$ a line. If 
${\fam0 H}^2(E(-1)) = 0$ and ${\fam0 H}^3(E(-2)) = 0$ then the restriction map 
${\fam0 H}^1(E) \ra {\fam0 H}^1(E_L)$ is surjective and one has an exact 
sequence$\, :$ 
\[
2{\fam0 H}^1(E) \xra{(h_0\, ,\, h_1)} {\fam0 H}^1(E(1)) \lra 
{\fam0 H}^1(E_L(1)) \lra 0\, . 
\]
\end{lemma}

\begin{proof}
The Castelnuovo-Mumford Lemma (in its slightly more general form stated in 
\cite[Lemma~1.21]{acm1}) implies that $\tH^2(E(l)) = 0$ for $l \geq -1$ and 
$\tH^3(E(l)) = 0$ for $l \geq -2$. One tensorizes, now, by $E$ and by $E(1)$ 
the Koszul complex $0 \ra \sco_\piii(-2) \ra 2\sco_\piii(-1) \ra \sco_\piii \ra 
\sco_L \ra 0$. 
\end{proof}

\begin{lemma}\label{L:l0}  
Let $E$ be a stable rank $3$ vector bundle on $\piii$ with $c_1 = 0$, $c_2 = 
3$, $c_3 = 0$ and spectrum $(0 , 0 , 0)$. Assume that the diferential $\beta$ 
of the Horrocks monad of $E$ is defined by the matrix from the conclusion of 
Lemma~\emph{\ref{L:beta}}. Let $L_0 \subset \piii$ be the line of equations 
$h_0 = h_1 = 0$. Then ${\fam0 H}^1(E(1)) \izo {\fam0 H}^1(E_{L_0}(1))$. 
Moreover, ${\fam0 h}^1(E_{L_0}(1)) \leq 1$ and ${\fam0 h}^1(E_{L_0}(1)) = 1$ if 
and only if $h_2^\prime = h_6^\prime = 0$ and $h_5^\prime \in kh_1$. 
\end{lemma}

\begin{proof} 
We assert that $h_i\tH^1(E) = (0)$ in $\tH^1(E(1))$, $i = 0,\, 1$. 
\emph{Indeed}, one has, as at the beginning of the proof of 
Lemma~\ref{L:beta}, $\tH^1(E(1)) \simeq \Cok \tH^0(\beta(1))$. One thus has to 
show that the elements of $3\tH^0(\sco_\piii(2))$ of the form 
$(h_ih_j , 0 , 0)^{\text{t}}$, $(0 , h_ih_j , 0)^{\text{t}}$, 
$(0 , 0 , h_ih_j)^{\text{t}}$, $i = 0,\, 1$, $j = 0 , \ldots , 3$, can be written 
as combinations of the columns of the matrix defining $\beta$ with 
coefficients linear forms. This is quite easy. For example$\, :$ 
\begin{gather*}
(0 , h_0h_j , 0)^{\text{t}} = 
h_0(h_{j+1}^\prime , h_j , 0)^{\text{t}} - 
h_{j+1}^\prime(h_0 , 0 , 0)^{\text{t}}\, ,\\    
(h_1h_j , 0 , 0)^{\text{t}} = 
h_j(h_1 , h_0 , 0)^{\text{t}} - (0 , h_0h_j , 0)^{\text{t}}\, ,\\ 
(0 , h_1h_j , 0)^{\text{t}} = 
h_1(h_{j+1}^\prime , h_j , 0)^{\text{t}} - (h_1h_{j+1}^\prime , 0 , 0)^{\text{t}}\, . 
\end{gather*}
Lemma~\ref{L:h1erah1el} implies, now, that $\tH^1(E(1)) \izo 
\tH^1(E_{L_0}(1))$. Restricting to $L_0$ the Horrocks monad of $E$ one sees 
that one has an exact sequence$\, :$ 
\[
0 \lra 3\sco_{L_0}(-1) \lra 2\sco_{L_0} \oplus N \lra E_{L_0} \lra 0\, , 
\]
where $N$ is the kernel of the epimorphism $7\sco_{L_0} \ra 3\sco_{L_0}(1)$ 
defined by the matrix$\, :$ 
\[
\begin{pmatrix} 
h_2^\prime & h_3^\prime \vb L_0 & h_4^\prime \vb L_0 & h_5^\prime \vb L_0 & 
h_6^\prime & h_7^\prime \vb L_0 & h_8^\prime \vb L_0\\ 
0 & h_2 & h_3 & 0 & 0 & 0 & 0\\ 
0 & 0 & 0 & 0 & 0 & h_2 & h_3
\end{pmatrix}\, . 
\]
Now, $\tH^1(N(l)) \izo \tH^1(E_{L_0}(l))$ for $l \geq 0$. One has
$\h^0(N) \leq 3$
(because there is no epimorphism $3\sco_{L_0} \ra 3\sco_{L_0}(1)$). 
Since $\h^0(N) - \h^1(N) = 1$ it follows that $\h^1(N) \leq 2$ hence 
$\h^1(N(1)) \leq 1$ (since $N$ is a bundle on $L_0 \simeq \pj$). 

If $\h^1(N(1)) = 1$ then $\h^1(N) = 2$ hence $\h^0(N) = 3$. If a linear 
combination of the columns of the above matrix, with coefficients
$c_1 , \ldots , c_7$,
is 0 then one must have $c_2 = c_3 = c_6 = c_7 = 0$. One deduces 
that if $\h^0(N) = 3$ then $h_2^\prime = h_6^\prime = 0$ and
$h_5^\prime \vb L_0 = 0$. 
Conversely, if $h_2^\prime = h_6^\prime = 0$ and $h_5^\prime \vb L_0 = 0$ then 
$N \simeq 3\sco_{L_0} \oplus \sco_{L_0}(-3)$ hence $\h^1(N(1)) = 1$.  
\end{proof} 

\begin{proposition}\label{P:c30000moduli} 
Let ${\fam0 M}(0)$ be the moduli space of stable rank $3$ vector bundles on 
$\piii$ with $c_1 = 0$, $c_2 = 3$, $c_3 = 0$. Then the open subscheme 
${\fam0 M}_{\fam0 min}(0)$ of ${\fam0 M}(0)$ corresponding to the bundles with 
minimal spectrum $(0 , 0 , 0)$ is nonsingular and irreducible, of dimension 
$28$. 
\end{proposition} 

\begin{proof} 
We show, firstly, that if $E$ is a stable rank 3 vector bundle on $\piii$ with 
the Chern classes and spectrum from the statement then $\tH^2(E^\vee \otimes E) 
= 0$. \emph{Indeed}, consider the Horrocks monad of $E$ from 
Lemma~\ref{L:c30000monads}. Let $Q$ be the cokernel of $\alpha$. Tensorizing 
by $E$ the exact sequence$\, :$ 
\[
0 \lra 3\sco_\piii(-1) \overset{{\overline \beta}^\vee}{\lra} Q^\vee \lra 
E^\vee \lra 0\, , 
\]
and using the fact that $\tH^i(E(-1)) = 0$, $i = 2,\, 3$, one gets that 
$\tH^2(Q^\vee \otimes E) \izo \tH^2(E^\vee \otimes E)$. Then, tensorizing by $E$ 
the exact sequence$\, :$ 
\[
0 \lra Q^\vee \lra 9\sco_\piii \overset{\alpha^\vee}{\lra} 3\sco_\piii(1) \lra 
0\, ,
\]
and using the fact that $\tH^2(E) = 0$, one gets an exact sequence$\, :$
\[
9\tH^1(E) \xra{\tH^1(\alpha^\vee \otimes \, {\text{id}}_E)} 3\tH^1(E(1)) 
\lra \tH^2(Q^\vee \otimes E) \lra 0\, . 
\] 
Now, Lemma~\ref{L:l0} implies that there is a line $L_0 \subset \piii$ such 
that $\tH^1(E(1)) \izo \tH^1(E_{L_0}(1))$. Moreover, by Lemma~\ref{L:h1erah1el}, 
the restriction map $\tH^1(E) \ra \tH^1(E_{L_0})$ is surjective. Since 
$\alpha_{L_0}^\vee \colon 9\sco_{L_0} \ra 3\sco_{L_0}(1)$ is an epimorphism, 
the map 
\[
\tH^1(\alpha_{L_0}^\vee \otimes \, {\text{id}}_{E \vb L_0}) \colon 9\tH^1(E_{L_0}) 
\lra 3\tH^1(E_{L_0}(1)) 
\]
is surjective (because $L_0$ has dimension 1). This implies that 
$\tH^1(\alpha^\vee \otimes \, {\text{id}}_E)$ is surjective hence $\tH^2(Q^\vee 
\otimes E) = 0$ hence $\tH^2(E^\vee \otimes E) = 0$. It follows that 
$\text{M}_{\text{min}}(0)$ is nonsingular, of pure dimension 28 (see 
Remark~\ref{R:moduli}). 

Next, let $\mathcal{U}$ be the open subset of $\text{M}_{\text{min}}(0)$ 
corresponding to the bundles with $\tH^1(E(1)) = 0$. There is a surjective 
morphism $\mathcal{S} \ra \mathcal {U}$, where $\mathcal{S}$ is the (reduced) 
space of monads of the form
$0 \ra 3\sco_\piii(-1) \overset{\alpha}{\ra} 
9\sco_\piii \overset{\beta}{\ra} 3\sco_\piii(1) \ra 0$,
with $\beta$ defined by 
a matrix of the form considered in the conclusion of Lemma~\ref{L:beta} such 
that at least one of the following holds$\, :$ $h_2^\prime \neq 0$ or 
$h_6^\prime \neq 0$ or $h_5^\prime \notin kh_1$ (see Lemma~\ref{L:l0}). For any 
such $\beta$, $\tH^0(\beta(1)) \colon \tH^0(9\sco_\piii(1)) \ra 
\tH^0(3\sco_\piii(2))$ is surjective hence the vector space$\, :$ 
\begin{equation}\label{E:spaceofalphas}
\{ \alpha \in \text{Hom}_{\sco_\piii}(3\sco_\piii(-1) , 9\sco_\piii) \vb 
\beta \circ \alpha = 0\} 
\end{equation}
has dimension $3 \times (\h^0(9\sco_\piii(1)) - \h^0(3\sco_\piii(2))) = 18$. 
One deduces that $\mathcal{S}$ is irreducible hence $\mathcal{U}$ is 
irreducible. 

According to Lemma~\ref{L:l0}, the complement $\text{M}_{\text{min}}(0) 
\setminus \mathcal{U}$ parametrizes the bundles with $\h^1(E(1)) = 1$. One has 
a surjective morphism $\mathcal{T} \ra \text{M}_{\text{min}}(0) \setminus 
\mathcal{U}$, where $\mathcal{T}$ is the space of monads of the above form 
but, this time, with $h_2^\prime = 0$, $h_6^\prime = 0$ and $h_5^\prime \in kh_1$. 
In this case $\tH^0(\beta(1))$ has corank 1 hence the space 
\eqref{E:spaceofalphas} has dimension $3 \times 7 = 21$. Anyway, $\mathcal{T}$ 
is irreducible hence $\text{M}_{\text{min}}(0) \setminus \mathcal{U}$ is 
irreducible. 

Since $\text{M}_{\text{min}}(0)$ is nonsingular, of pure dimension 28, it 
suffices, in order to check that it is irreducible, to show that 
${\overline {\mathcal{U}}} \cap (\text{M}_{\text{min}}(0) \setminus \mathcal{U}) 
\neq \emptyset$. This can be shown using the family of monads 
$0 \ra 3\sco_\piii(-1) \overset{\alpha_t}{\lra} 9\sco_\piii 
\overset{\beta_t}{\lra} 3\sco_\piii(1) \ra 0$, with$\, :$ 
\[
\beta_t := 
\begin{pmatrix} 
X_0 & X_1 & 0 & 0 & X_2 & tX_3 & tX_2 & X_3 & 0\\
0 & X_0 & X_1 & X_2 & X_3 & 0 & 0 & 0 & 0\\
0 & 0 & 0 & 0 & 0 & X_0 & X_1 & X_2 & X_3
\end{pmatrix}\, , 
\]
\[
\alpha_t^\vee := 
\begin{pmatrix} 
X_2 & X_3 & 0 & tX_3 & -X_0 - tX_2 & 0 & X_2 & -X_1 & 0\\ 
0 & X_2 & X_3 & -X_0 & -X_1 & 0 & 0 & 0 & 0\\
-X_3 & 0 & 0 & 0 & 0 & -X_2 & X_3 & X_0 & -X_1
\end{pmatrix}\, . 
\qedhere
\]
\end{proof}

\begin{proposition}\label{P:c30-101} 
Let $E$ be a stable rank $3$ vector bundle on $\piii$ with $c_1 = 0$,
$c_2 = 3$, $c_3 = 0$ and spectrum $(-1 , 0 , 1)$. Then$\, :$ 

\emph{(a)} $E$ has an unstable plane of order $1$. 

\emph{(b)} The restriction of $E$ to a general plane is stable. 

\emph{(c)} $E$ is the cohomology sheaf of a Horrocks monad of the form$\, :$ 
\[
0 \lra \sco_\piii(-2) \overset{\alpha}{\lra} \sco_\piii(1) \oplus 3\sco_\piii 
\oplus \sco_\piii(-1) \overset{\beta}{\lra} \sco_\piii(2) \lra 0\, . 
\]
\end{proposition}

\begin{proof} 
(a) The spectrum of $E^\vee$ is $(-1 , 0 , 1)$, too. It follows that 
$\h^1(E^\vee(-2)) = 1$ and $\h^1(E^\vee(-1)) = 3$. One deduces that there exists 
a non-zero linear form $h_0$ such that the multiplication by $h_0 \colon 
\tH^1(E^\vee(-2)) \ra \tH^1(E^\vee(-1))$ is the zero map. If $H_0$ is the plane 
of equation $h_0 = 0$ then $\h^0(E_{H_0}^\vee(-1)) = 1$. 

(b) It follows from (a) and from \cite[Prop.~5.1]{ehv} (recalled in 
Remark~\ref{R:spehv}) that $\tH^0(E_H) = 0$, for the general plane
$H \subset \piii$.
Since $E^\vee$ has the same Chern classes and spectrum as $E$, 
one deduces that, for the general plane $H \subset \piii$, one has 
$\tH^0(E_H^\vee) = 0$, too. 

(c) We assert that the graded $S$-module $\tH^1_\ast(E)$ is generated by 
$\tH^1(E(-2))$. \emph{Indeed}, using the spectrum one sees that $\tH^1(E(l)) 
= 0$ for $l \leq -3$, $\h^1(E(-2)) = 1$, $\h^1(E(-1)) = 3$ and that 
$\tH^2(E(l)) = 0$ for $l \geq -1$. Moreover, $\h^1(E) = 3$ by Riemann-Roch. 
Since $\tH^2(E(-1)) = 0$ and $\tH^3(E(-2)) = 0$, the slightly more general 
variant of the Castelnuovo-Mumford Lemma recalled in \cite[Lemma~1.21]{acm1} 
implies that $\tH^1_\ast(E)$ is generated in degrees $\leq 0$. 

Since $\tH^0(E) = 0$, $\tH^1(E(-2))$ cannot be annihilated by two linearly 
independent linear forms (because if it would be, denoting by $L$ the line 
defined by these forms and using the exact sequence $0 \ra \sco_\piii(-2) \ra 
2\sco_\piii(-1) \ra \sci_L \ra 0$, one would have $\tH^0(\sci_L \otimes E) \neq 
0$). It follows that $S_1\tH^1(E(-2)) = \tH^1(E(-1))$. On the other hand, 
by (b), if $h$ is a general linear form then, denoting by $H$ the plane of 
equation $h = 0$, one has $\tH^0(E_H) = 0$ hence multiplication by
$h \colon \tH^1(E(-1)) \ra \tH^1(E)$ is injective hence bijective. Our
assertion is proven. 

Since $E^\vee$ has the same Chern classes and spectrum as $E$ it follows that 
$\tH^1_\ast(E^\vee)$ is generated by $\tH^1(E^\vee(-2))$. One deduces (see Barth 
and Hulek \cite{bh}) that $E$ is the cohomology sheaf of a Horrocks monad of 
the form $0 \ra \sco_\piii(-2) \ra B \ra \sco_\piii(2) \ra 0$, where $B$ is a 
direct sum of line bundles. $B$ has rank 5, $\tH^0(B(-2)) = 0$ and 
$\h^0(B(-1)) = \h^0(\sco_\piii(1)) - \h^1(E(-1)) = 1$. Analogously, 
$\tH^0(B^\vee(-2)) = 0$ and $\h^0(B^\vee(-1)) = 1$. It follows that
$B \simeq \sco_\piii(1) \oplus 3\sco_\piii \oplus \sco_\piii(-1)$. 
\end{proof}

\begin{lemma}\label{L:h0beta2}  
If $\beta \colon \sco_\piii(1) \oplus 3\sco_\piii \oplus \sco_\piii(-1) \ra 
\sco_\piii(2)$ is an epimorphism with ${\fam0 H}^0(\beta)$ injective then 
${\fam0 H}^0(\beta(2))$ is surjective. 
\end{lemma}

\begin{proof} 
The hypothesis $\tH^0(\beta)$ injective is equivalent to
$\tH^0(\Ker \beta) = 0$
and the conclusion to $\tH^1(\Ker \beta(2)) = 0$. The component 
$\sco_\piii(1) \ra \sco_\piii(2)$ is defined by a linear form $h_0$ which must 
be non-zero. We deduce an exact sequence$\, :$ 
\[
0 \lra \Ker \beta \lra 3\sco_\piii \oplus \sco_\piii(-1) 
\overset{\overline{\beta}}{\lra} \sco_{H_0}(2) \lra 0\, . 
\]
Let $\beta_0 \colon 3\sco_{H_0} \oplus \sco_{H_0}(-1) \ra \sco_{H_0}(2)$ denote 
the morphism ${\overline{\beta}} \otimes \sco_{H_0}$. Then one has an exact 
sequence$\, :$ 
\[
0 \lra 3\sco_\piii(-1) \oplus \sco_\piii(-2) \lra \Ker \beta \lra \Ker \beta_0 
\lra 0\, . 
\]
Let $\beta_0^\prime \colon 3\sco_{H_0} \ra \sco_{H_0}(2)$ and $\beta_0^\secund 
\colon \sco_{H_0}(-1) \ra \sco_{H_0}(2)$ be the components of $\beta_0$. 
$\beta_0^\secund$ is defined by a cubic form $f \in \tH^0(\sco_{H_0}(3))$. One 
deduces that $\text{Im}\, \beta_0^\prime = \sci_{Z , H_0}(2)$, where $Z$ is a 
closed subscheme of $H_0$ of dimension $\leq 0$. Let $K$ denote the kernel of 
$\beta_0^\prime$. One has exact sequences$\, :$ 
\begin{gather*}
0 \lra K \lra 3\sco_{H_0} \lra \sci_{Z , H_0}(2) \lra 0\, ,\\
0 \lra K \lra \Ker \beta_0 \lra \sci_{Z , H_0}(-1) \lra 0\, . 
\end{gather*}
$K$ is a rank 2 vector bundle on $H_0$ with $c_1(K) = -2$. It follows that 
$K^\vee \simeq K(2)$. Dualizing the first exact sequence, one get an exact 
sequence$\, :$ 
\[
0 \lra \sco_{H_0}(-2) \lra 3\sco_{H_0} \lra K^\vee \lra \omega_Z(1) \lra 0 
\] 
from which one deduces that $\tH^1(K^\vee) = 0$ hence $\tH^1(K(2)) = 0$. One 
also deduces, from the first exact sequence, that $\tH^1(\sci_{Z , H_0}(1)) 
\simeq \tH^2(K(-1)) \simeq \tH^0(K^\vee(-2))^\vee \simeq \tH^0(K)^\vee = 0$. 
Using, now, the second exact sequence one gets that $\tH^1(\Ker \beta_0(2)) 
= 0$. 
\end{proof}

\begin{proposition}\label{P:c30moduli}  
The moduli space ${\fam0 M}(0)$ of stable rank $3$ vector bundles on $\piii$ 
with Chern classes $c_1 = 0$, $c_2 = 3$, $c_3 = 0$ is nonsingular and 
irreducible, of dimension $28$. 
\end{proposition}

\begin{proof} 
The closed subset $\text{M}_{\text{max}}(0) := \text{M}(0) \setminus 
\text{M}_{\text{min}}(0)$ of $\text{M}(0)$ corresponds to the bundles with 
maximal spectrum $(-1 , 0 , 1)$. We show, firstly, that if $E$ is such a 
bundle then $\tH^2(E^\vee \otimes E) = 0$. \emph{Indeed}, according to 
Prop.~\ref{P:c30-101}(c), $E$ is the cohomology sheaf of a monad of 
the form
$0 \ra \sco_\piii(-2) \overset{\alpha}{\ra} B \overset{\beta}{\ra} \sco_\piii(2)
\ra 0$,
where $B = \sco_\piii(1) \oplus 3\sco_\piii \oplus \sco_\piii(-1)$.
If $Q$ is the cokernel of $\alpha$ then one has an exact 
sequence $0 \ra \sco_\piii(-2) \ra Q^\vee \ra E^\vee \ra 0$. Since
$\tH^3(E(-2)) = 0$, the map
$\tH^2(Q^\vee \otimes E) \ra \tH^2(E^\vee \otimes E)$
is surjective. Tensorizing by $E$ the exact sequence$\, :$ 
\[
0 \lra Q^\vee \lra \sco_\piii(-1) \oplus 3\sco_\piii \oplus \sco_\piii(1) 
\overset{\alpha^\vee}{\lra} \sco_\piii(2) \ra 0\, ,
\] 
and using the fact that $\tH^1(E(2)) = 0$ (by Lemma~\ref{L:h0beta2}) and 
that $\tH^2(E(l)) = 0$, for $l \geq -1$, according to the spectrum, one 
deduces that $\tH^2(Q^\vee \otimes E) = 0$ hence $\tH^2(E^\vee \otimes E) = 0$. 

It follows that $\text{M}(0)$ is nonsingular, of local dimension 28, at 
every point of $\text{M}_{\text{max}}(0)$ (hence it is nonsingular of dimension 
28 everywhere, due to Prop.~\ref{P:c30000moduli}). We will show, next, that 
$\text{M}_{\text{max}}(0)$ is irreducible, of dimension 27. This will imply, of 
course, that $\text{M}(0)$ is irreducible (by Prop.~\ref{P:c30000moduli}, 
again). 

Let $\mathcal{N}$ be the (reduced) space of the monads of the above form, 
with $\tH^0(\beta)$ and $\tH^0(\alpha^\vee)$ injective. Applying the functor 
$\sch^0$ to these monads, one gets a map $\mathcal{N} \ra \text{M}(0)$ whose 
image is $\text{M}_{\text{max}}(0)$. The vector space
$\text{Hom}_{\sco_\piii}(B , \sco_\piii(2))$ has dimension 54 and if
$\beta \colon B \ra \sco_\piii(2)$ is 
an epimorphism with $\tH^0(\beta)$ injective then, by Lemma~\ref{L:h0beta2}, 
the space of morphisms $\alpha \colon \sco_\piii(-2) \ra B$ satisfying 
$\beta \circ \alpha = 0$ has dimension $54 - 35 = 19$. One deduces that 
$\mathcal{N}$ is irreducible of dimension $54 + 19 = 73$. 

Now, two bundles $E$ and $E^\prime$ with spectrum $(-1 , 0 , 1)$ are isomorphic 
if and only if their Horrocks monads are isomorphic (see \cite[Prop.~4]{bh}). 
Let $G$ be the group of automorphisms of $B$. The group $\Gamma := 
\text{GL}(1) \times G \times \text{GL}(1)$ acts on $\mathcal{N}$ by$\, :$ 
\[
(c , \phi , c^\prime) \cdot (\alpha , \beta) := (\phi \alpha c^{\prime -1} , 
c\beta \phi^{-1})\, . 
\]
The orbits of this action are exactly the fibres of the morphism
$\mathcal{N} \ra \text{M}(0)$. $\text{GL}(1)$
embeds diagonally into $\Gamma$. Since 
$\text{Hom}_{\sco_\piii}(E , E) \simeq k$, for any stable bundle on $\piii$, 
the stabilizer of any element of $\mathcal{N}$ under the action of $\Gamma$ 
is the image of that embedding. Since $\Gamma/\text{GL}(1)$ has dimension 
$47 - 1 = 46$ it follows that $\text{M}_{\text{max}}(0)$ has dimension
$73 - 46 = 27$. 
\end{proof}

\section{The case $c_3 = 2$}\label{S:c32}

\begin{lemma}\label{L:c32monad}
Let $E$ be a stable rank $3$ vector bundle on $\piii$ with $c_1 = 0$,
$c_2 = 3$, $c_3 = 2$. Then$\, :$ 

\emph{(a)} $E$ is the cohomology sheaf of a monad of the form$\, :$ 
\[
0 \lra \sco_\piii(-1) \oplus \sco_\piii(-2) \overset{\alpha}{\lra} 
6\sco_\piii \oplus \sco_\piii(-1) \overset{\beta}{\lra} 2\sco_\piii(1) 
\lra 0\, . 
\]

\emph{(b)} If $E$ has no unstable plane then it is the cohomology sheaf of a 
monad of the form$\, :$ 
\[
0 \lra \sco_\piii(-2) \overset{\alpha}{\lra} 6\sco_\piii 
\overset{\beta}{\lra} 2\sco_\piii(1) \lra 0\, .
\]
\end{lemma}

\begin{proof} 
(a) The spectrum of $E$ must be $(-1 , 0 , 0)$. One deduces that $\tH^1(E(l)) 
= 0$ for $l \leq -2$, $\h^1(E(-1)) = 2$, $\h^2(E(-3)) = 4$, $\h^2(E(-2)) = 1$, 
and $\tH^2(E(l)) = 0$ for $l \geq -1$. By Riemann-Roch, $\h^1(E) = 2$. 

\vskip2mm 

\noindent 
{\bf Claim 1.}\quad $\tH^1_\ast(E)$ \emph{is generated by} $\tH^1(E(-1))$. 

\vskip2mm 

\noindent 
\emph{Indeed}, since $\tH^2(E(-1)) = 0$ and $\tH^3(E(-2)) = 0$, the 
Castelnuovo-Mumford Lemma (in its slightly more general form recalled in 
\cite[Lemma~1.21]{acm1}) implies that $\tH^1_\ast(E)$ is generated in degrees 
$\leq 0$. It remains to show that the multiplication map $S_1 \otimes 
\tH^1(E(-1)) \ra \tH^1(E)$ is surjective. Assume, by contradiction, that it 
is not. Then its image is contained in a 1-dimensional subspace $A$ of 
$\tH^1(E)$. 
Consider the Beilinson monad of $E$ (see Remark~\ref{R:beilinson})$\, :$ 
\[
0 \lra \tH^2(E(-3)) \otimes \Omega_\piii^3(3) \overset{\gamma}{\lra} 
\begin{matrix} 
\tH^2(E(-2)) \otimes \Omega_\piii^2(2)\\ 
\oplus\\ 
\tH^1(E(-1)) \otimes \Omega_\piii^1(1) 
\end{matrix}
\overset{\delta}{\lra} \tH^1(E) \otimes \sco_\piii \lra 0\, . 
\]
By our assumption, the image of the restriction $\delta_2$ of $\delta$ to 
$\tH^1(E(-1)) \otimes \Omega_\piii^1(1)$ is contained in $A \otimes \sco_\piii$. 
Let $A^\prime$ denote the quotient $\tH^1(E)/A$. Denoting by $\gamma_1$ the 
component
$\tH^2(E(-3)) \otimes \Omega_\piii^3(3) \ra \tH^2(E(-2)) \otimes 
\Omega_\piii^2(2)$
of $\gamma$, one deduces that one has an epimorphism 
$\Cok \gamma_1 \ra A^\prime \otimes \sco_\piii$. But the multiplication map 
$S_1 \otimes \tH^2(E(-3)) \ra \tH^2(E(-2))$ is surjective (because 
$\tH^3(E(-4)) = 0$) hence the morphism $\gamma_1$ is non-zero. Since there is 
no complex $\sco_\piii(-1) \overset{\phi}{\ra} \Omega_\piii^2(2) 
\overset{\pi}{\ra} \sco_\piii$ 
with $\phi \neq 0$ and $\pi$ an epimorphism, we have got the desired
\emph{contradiction}. 

\vskip2mm 

\noindent 
{\bf Claim 2.}\quad $\tH^1_\ast(E^\vee)$ \emph{has one minimal generator of 
degree} $-2$ \emph{and at most one of degree} $-1$. 

\vskip2mm 

\noindent 
\emph{Indeed}, since the spectrum of $E^\vee$ is $(0 , 0 , 1)$, it follows that 
$\tH^1(E^\vee(l)) = 0$ for $l \leq -3$, $\h^1(E^\vee(-2)) = 1$, 
$\h^1(E^\vee(-1)) = 4$, and $\tH^2(E^\vee(l)) = 0$ for $l \geq -2$. One deduces 
that $\tH^1_\ast(E^\vee)$ is generated in degrees $\leq -1$ (because 
$\tH^2(E^\vee(-2)) = 0$ and $\tH^3(E^\vee(-3)) = 0$). Moreover, the 
multiplication map $S_1 \otimes \tH^1(E^\vee(-2)) \ra \tH^1(E^\vee(-1))$ cannot 
have rank $\leq 2$ because in that case it would exist a line
$L \subset \piii$ such that $\tH^0(\sci_L \otimes E^\vee) = 0$,
\emph{which is not true}. 

\vskip2mm 

\noindent 
The two claims above imply that $E$ is the cohomology sheaf of a (not 
necessarily minimal) Horrocks monad of the form $0 \ra \sco_\piii(-1) \oplus 
\sco_\piii(-2) \ra B \ra 2\sco_\piii(1) \ra 0$, with $B$ a direct sum of line 
bundles. $B$ has rank $7$, $\tH^0(B(-1)) = 0$, $c_1(B) = -1$ hence $B \simeq 
6\sco_\piii \oplus \sco_\piii(-1)$. 

(b) If $E$ has no unstable plane then the multipliction map $S_1 \otimes 
\tH^1(E^\vee(-2)) \ra \tH^1(E^\vee(-1))$ is injective hence bijective hence 
$\tH^1_\ast(E^\vee)$ has no minimal generator of degree $-1$. 
\end{proof}

\begin{lemma}\label{L:c32h0eh0}  
Let $E$ be a stable rank $3$ vector bundle on $\piii$ with $c_1 = 0$, $c_2 = 
3$, $c_3 = 2$. Then, for the general plane $H \subset \piii$, one has 
${\fam0 H}^0(E_H) = 0$. 
\end{lemma}

\begin{proof} 
Assume, by contradiction, that this is not the case. Consider the morphism
$\mu \colon \tH^1(E(-1)) \otimes \sco_{\p^{3 \vee}}(-1) \ra \tH^1(E) \otimes
\sco_{\p^{3 \vee}}$ from Remark~\ref{R:mu}. Using the notation from that remark,
one has $a \leq -2$ (by observation (i)) and $b \geq 1$ (by observation (ii)
because, by Claim 1 in the proof of Prop.~\ref{P:c30000}, one has
$S_1\tH^1(E(-1)) = \tH^1(E)$). But this \emph{contradicts} relation
\eqref{E:ab} (in which one has $d = 2$). 
\end{proof}

\begin{proposition}\label{P:c32unstableplane} 
Let $E$ be a stable rank $3$ vector bundle on $\piii$ with $c_1 = 0$, $c_2 = 
3$, $c_3 = 2$. Assume that $E$ has an unstable plane. Then, for the general 
plane $H \subset \piii$, one has ${\fam0 H}^0(E_H^\vee) = 0$.
\end{proposition}

\begin{proof} 
We intend to apply Lemma~\ref{L:jumpinglines}. 
According to Lemma~\ref{L:c32monad}, $E$ is the cohomology sheaf of a minimal 
Horrocks monad of the form$\, :$ 
\[
0 \lra \sco_\piii(-1) \oplus \sco_\piii(-2) \overset{\alpha}{\lra} 
6\sco_\piii \oplus \sco_\piii(-1) \overset{\beta}{\lra} 2\sco_\piii(1) 
\lra 0\, . 
\] 
The component $\sco_\piii(-1) \ra \sco_\piii(-1)$ of $\alpha$ is zero and the 
component $\sco_\piii(-2) \ra \sco_\piii(-1)$ is defined by a linear form $h_0$. 
Since $\tH^0(\alpha^\vee)$ is injective, one has $h_0 \neq 0$. Let
$H_0 \subset \piii$ be the plane of equation $h_0 = 0$.

Dualizing the above monad, one deduces that $E^\vee$ is the middle cohomology
sheaf of a complex$\, :$
\[
0 \lra 2\sco_\piii(-1) \lra 2\sco_\piii \oplus \Omega_\piii(1) \overset{\e}{\lra}
\sco_{H_0}(2) \lra 0\, . 
\]
$\e$ is defined by two elements $f_0$, $f_1$ of $\tH^0(\sco_{H_0}(2))$ and by
a morphism $\Omega_\piii(1) \ra \sco_{H_0}(2)$ which can be written as a
composite map
$\Omega_\piii(1) \ra \sco_\piii \otimes V^\vee \overset{\phi}{\ra} \sco_{H_0}(2)$.
Since $\tH^0(E^\vee) = 0$ it follows that $f_0$ and $f_1$ are linearly
independent. Put $Z := \{x \in H_0 \vb f_0(x) = f_1(x) = 0\}$.

If $L$ is a line not contained in $H_0$ then $\tH^1(E_L^\vee)$ is isomorphic
to the cokernel of$\, :$
\[
\tH^0(\e_L) \colon \tH^0\left(2\sco_L \oplus (\Omega_\piii(1) \vb L)\right)
\lra \tH^0(\sco_{L \cap H_0}(2))\, . 
\]
If $L \cap Z = \emptyset$ then $\tH^0(\e_L)$ is surjective hence
$\tH^1(E_L^\vee) = 0$. Assume, now, that $L \cap H_0$ consists of a point $x$
belonging to $Z$. Since $\e$ is an epimorphism, the map
$(\Omega_\piii(1))(x) \ra (\sco_{H_0}(2))(x)$ is surjective. But
$(\Omega_\piii(1))(x) = \tH^0(\sci_{\{x\}}(1)) \subset V^\vee$ and
$\tH^0(\Omega_\piii(1) \vb L) = \tH^0(\sci_L(1)) \subset \tH^0(\sci_{\{x\}}(1))$.
One deduces that for all the lines $L$ passing through $x$ and not contained
in $H_0$, except at most one, the map
$\tH^0(\Omega_\piii(1) \vb L) = \tH^0(\sci_L(1)) \ra \sco_{H_0}(2)(x) =
\tH^0(\sco_{L \cap H_0}(2))$
induced by $\phi$ is surjective. Consequently, any line $L \subset \piii$
for which $\h^1(E_L^\vee) \geq 1$ is either contained in $H_0$ or belongs to
a family of dimension at most 1 (because $\dim Z \leq 1$). The conclusion of
the proposition follows, now, from Lemma~\ref{L:jumpinglines}. 
\end{proof} 

\begin{lemma}\label{L:wsubwedge2v}  
Let $V$ be a $4$-dimensional $k$-vector space and $W$ a $4$-dimensional 
subspace of $\bigwedge^2V$. Then there exists a $k$-basis $v_0 , \ldots , v_3$ 
of $V$ such that $W$ admits one of the following bases$\, :$ 
\begin{enumerate} 
\item[(i)] $v_0 \wedge v_1$, $v_1 \wedge v_2$, $v_2 \wedge v_3$,
$v_0 \wedge v_3\, ;$ 
\item[(ii)] $v_0 \wedge v_1$, $v_1 \wedge v_2$, $v_2 \wedge v_3$,
$v_0 \wedge v_2 + v_1 \wedge v_3\, ;$ 
\item[(iii)] $v_0 \wedge v_1$, $v_1 \wedge v_2$, $v_2 \wedge v_3$,
$v_0 \wedge v_2$. 
\end{enumerate}
\end{lemma} 

\begin{proof} 
Consider the canonical pairing $\langle \ast , \ast \rangle \colon 
\bigwedge^2V^\vee \times \bigwedge^2 V \ra k$
and let
$W^\perp \subset \bigwedge^2V^\vee$
consist of the elements $\alpha$ with
$\langle \alpha , \omega \rangle = 0$, $\forall \, \omega \in W$.
Since $W^\perp$ has dimension 
$2$, a well known result (see, for example, \cite[Lemma~G.4]{acm2}) says that 
there exists a basis $h_0 , \ldots , h_3$ of $V^\vee$ such that $W^\perp$ admits 
one of the following bases$\, :$ 
\begin{enumerate} 
\item[(1)] $h_0 \wedge h_2$, $h_1 \wedge h_3$$\, ;$ 
\item[(2)] $h_0 \wedge h_3$, $h_0 \wedge h_2 - h_1 \wedge h_3$$\, ;$ 
\item[(3)] $h_0 \wedge h_3$, $h_1 \wedge h_3$. 
\end{enumerate}
Let $v_0 , \ldots , v_3$ be the dual basis of $V$. If $W^\perp$ admits the 
basis (1) (resp., (2), resp., (3)) then $V$ admits the basis (i) (resp., (ii), 
resp., (iii)). 
\end{proof} 

\begin{proposition}\label{P:c32nounstableplane} 
Let $E$ be a stable rank $3$ vector bundle on $\piii$ with $c_1 = 0$,
$c_2 = 3$, $c_3 = 2$. Assume that $E$ has no unstable plane. If
${\fam0 H}^0(E_H^\vee) \neq 0$, for every plane $H \subset \piii$, then $E$ is
as in Theorem~\emph{\ref{T:c10c23}(b)(ii)}. 
\end{proposition}

\begin{proof} 
According to Lemma~\ref{L:c32monad}(b), $E$ is the cohomology sheaf of a monad 
of the form$\, :$ 
\[
0 \lra \sco_\piii(-2) \overset{\alpha}{\lra} 6\sco_\piii 
\overset{\beta}{\lra} 2\sco_\piii(1) \lra 0\, .
\] 
The spectrum of the dual vector bundle $E^\vee$ is $(0 , 0 , 1)$. It follows 
that $\tH^1(E^\vee(l)) = 0$ for $l \leq -3$, $\h^1(E^\vee(-2)) = 1$, 
$\h^1(E^\vee(-1)) = 4$, $\h^2(E^\vee(-3)) = 2$, $\tH^2(E^\vee(l)) = 0$ for 
$l \geq -2$. Moreover, by Riemann-Roch, $\h^1(E^\vee) = 4$. By 
Remark~\ref{R:beilinson}, $E^\vee$ is the cohomology sheaf of a Beilinson 
monad of the form$\, :$ 
\[
0 \lra \begin{matrix} \tH^1(E^\vee(-2)) \otimes \Omega_\piii^2(2)\\ \oplus\\ 
\tH^2(E^\vee(-3)) \otimes \Omega_\piii^3(3) \end{matrix} 
\overset{\gamma}{\lra} \tH^1(E^\vee(-1)) \otimes \Omega_\piii^1(1) 
\overset{\delta}{\lra} \tH^1(E^\vee) \otimes \sco_\piii \lra 0\, .  
\]

Consider, now, the morphism
$\mu \colon \tH^1(E^\vee(-1)) \otimes \sco_{\p^{3\vee}}(-1) \ra \tH^1(E^\vee)
\otimes \sco_{\p^{3\vee}}$
from Remark~\ref{R:mu}. Using the notation from that remark, one has
$a \leq -2$ (by observation (i)) and $b \geq 1$ (by observation (ii) because
one sees easily, dualizing the above Horrocks monad, that
$S_1\tH^1(E^\vee(-1)) = \tH^1(E^\vee)$).

\vskip2mm

\noindent
{\bf Claim 1.}\quad $b \geq 2$.

\vskip2mm 

\noindent
\emph{Indeed}, assume, by contradiction, that $b = 1$. As we noticed in the
Introduction (after the statement of Thm.~\ref{T:c10c23}) the general plane
$H \subset \piii$ contains a line $L_0$ with $\h^1(E^\vee_{L_0}) \geq 1$.
Observation (iii) in Remark~\ref{R:mu} implies that $Y$ has dimension 1.
Since $\sci_Y(1)$ is globally generated, $Y$ must be a line. In this case,
$\mu$ factorizes as$\, :$
\[
\tH^1(E^\vee(-1)) \otimes \sco_{\p^{3 \vee}}(-1) \overset{\overline{\mu}}{\lra}
(N_0 \otimes \sco_{\p^{3 \vee}}) \oplus \sco_{\p^{3 \vee}}(-1) \lra
\tH^1(E^\vee) \otimes \sco_{\p^{3 \vee}}\, , 
\]
for some subspace $N_0$ of $\tH^1(E^\vee)$ of dimension 2. The kernel of the
component
$\tH^1(E^\vee(-1)) \otimes \sco_{\p^{3 \vee}}(-1) \ra \sco_{\p^{3 \vee}}(-1)$
of $\overline{\mu}$ has the form $N_{-1} \otimes \sco_{\p^{3 \vee}}(-1)$ for some
1-codimensional subspace $N_{-1}$ of $\tH^1(E^\vee(-1))$. Since $\mu$ maps
$N_{-1} \otimes \sco_{\p^{3 \vee}}(-1)$ into $N_0 \otimes \sco_{\p^{3 \vee}}$ it
follows that $S_1N_{-1} \subseteq N_0$ which \emph{contradicts}
Lemma~\ref{L:s1n-1}. 

\vskip2mm 

It follows, now, from Claim 1 and from relation \eqref{E:ab} (with $d = 4$), 
that $b = 2$ and $a = -2$. 
In this case, we have an exact sequence$\, :$  
\[
0 \lra \sco_{\p^{3 \vee}}(-2) \overset{\kappa}{\lra} \tH^1(E^\vee(-1)) \otimes 
\sco_{\p^{3 \vee}}(-1) \overset{\mu}{\lra} \tH^1(E^\vee) \otimes 
\sco_{\p^{3 \vee}}\, .
\] 
Choosing a basis of $\tH^1(E^\vee(-1))$, $\kappa$ is defined by four vectors 
$u_0 , \ldots , u_3 \in V = \tH^0(\sco_{\p^{3 \vee}}(1))$. We assert that 
$u_0 , \ldots , u_3$ \emph{are linearly independent}. 

\emph{Indeed}, if $ku_0 + \ldots + ku_3$ has dimension $c < 4$ then there is a 
decomposition $\tH^1(E^\vee(-1)) = N_{-1} \oplus N_{-1}^\prime$, with $N_{-1}$ 
of dimension $c$, such that
$\text{Im}\, \kappa \subset N_{-1} \otimes \sco_{\p^{3 \vee}}(-1)$. 
It follows that
$\Cok \kappa \simeq \scf \oplus (N_{-1}^\prime \otimes \sco_{\p^{3 \vee}}(-1))$,
where $\scf$ is a sheaf defined by an exact 
sequence$\, :$
\[
0 \lra \sco_{\p^{3 \vee}}(-2) \lra N_{-1} \otimes \sco_{\p^{3 \vee}}(-1) \lra 
\scf \lra 0\,. 
\]
One cannot have $c = 1$ because, in this case, $\scf$ is a torsion sheaf and 
this \emph{contradicts} the fact that $\Cok \kappa \simeq \text{Im}\, \mu$. 
If $c \in \{2,\, 3\}$ then, dualizing the above sequence, one sees that 
$\text{Hom}_{\sco_{\p^{3 \vee}}}(\scf , \sco_{\p^{3 \vee}})$ has dimension 1 for 
$c = 2$ and dimension 3 for $c = 3$. One deduces that there exists a subspace 
$N_0$ of $\tH^1(E^\vee)$, of dimension 1 if $c = 2$ and of dimension 3 if
$c = 3$, such that $\mu$ maps $N_{-1} \otimes \sco_{\p^{3 \vee}}(-1)$ into
$N_0 \otimes \sco_{\p^{3 \vee}}$. This means that $S_1N_{-1} \subseteq N_0$ and this 
\emph{contradicts} the following assertion$\, :$

\vskip2mm 

\noindent 
{\bf Claim 2.}\quad \emph{If} $N_{-1}$ \emph{is a subspace of} 
${\fam0 H}^1(E^\vee(-1))$, \emph{of dimension} $2$ \emph{or} $3$, 
\emph{then the dimension of the subspace} $S_1N_{-1}$ 
\emph{of} ${\fam0 H}^1(E^\vee)$ \emph{is} $> \dim_kN_{-1}$. 

\vskip2mm 

\noindent 
\emph{Indeed}, the image of
$\tH^0(\alpha^\vee) \colon \tH^0(6\sco_\piii) \ra \tH^0(\sco_\piii(2))$
is a 6-dimensional, base point free subspace $U$ of 
$\tH^0(\sco_\piii(2))$. One has
$\tH^1(E^\vee(-1)) \simeq \tH^0(\sco_\piii(1))$ 
and $\tH^1(E^\vee) \simeq \tH^0(\sco_\piii(2))/U$. By the first isomorphism, 
$N_{-1}$ is identified with $\tH^0(\sci_\Lambda(1))$, for some linear subspace 
$\Lambda$ of $\piii$ with $\text{codim}(\Lambda , \piii) = \dim_kN_{-1}$. Then 
$S_1N_{-1} \simeq (\tH^0(\sci_\Lambda(2)) + U)/U$ hence
$\tH^1(E^\vee)/S_1N_{-1}$ 
is isomorphic to the cokernel of the restriction map
$U \ra \tH^0(\sco_\Lambda(2))$.
Since $U$ is base point free, the dimension of this 
cokernel is $0$ if $\dim \Lambda = 0$ and $\leq 1$ if $\dim \Lambda = 1$. The  
claim is proven. 

\vskip2mm 

It remains that $u_0, \ldots , u_3$ are linearly independent. This 
implies that $\Cok \kappa \simeq \text{T}_{\p^{3 \vee}}(-2)$. $\mu$ induces a 
morphism $\overline{\mu} \colon \text{T}_{\p^{3 \vee}}(-2) \ra \tH^1(E^\vee) 
\otimes \sco_{\p^{3 \vee}}$.
Since one has  $S_1\tH^1(E^\vee(-1)) = \tH^1(E^\vee)$, the map
$\tH^0(\mu^\vee) \colon \tH^1(E^\vee)^\vee \ra \tH^1(E^\vee(-1))^\vee \otimes
\tH^0(\sco_{\p^{3 \vee}}(1))$
is injective. It follows that the map 
$\tH^0({\overline{\mu}}^\vee) \colon \tH^1(E^\vee)^\vee \ra 
\tH^0(\Omega_{\p^{3 \vee}}(2))$
is injective, too. Its image is a 4-dimensional 
subspace $W$ of $\tH^0(\Omega_{\p^{3 \vee}}(2)) \simeq \bigwedge^2V$. Using
Lemma~\ref{L:wsubwedge2v} one sees, now, that, choosing convenient $k$-bases
of $\tH^1(E^\vee(-1))$ and $\tH^1(E^\vee)$, $\mu$ is represented by some
concrete $4 \times 4$ matrix $\mathcal{M}$ with entries in $V$.

We make, at this point, the following observation$\, :$ the same matrix
$\mathcal{M}$ defines the differential $\delta$ 
of the above Beilinson monad of $E^\vee$ (see Remark~\ref{R:mu}).
The component 
$\gamma_1 \colon \tH^1(E^\vee(-2)) \otimes \Omega_\piii^2(2) \ra 
\tH^1(E^\vee(-1)) \otimes \Omega_\piii^1(1)$ of the differential $\gamma$ of the 
monad is defined by a $4 \times 1$ matrix $(w_0 , \ldots , w_3)^{\text{t}}$ with 
entries in $V$. Since $E$ has no unstable plane, the multiplication map 
$\tH^1(E^\vee(-2)) \otimes S_1 \ra \tH^1(E^\vee(-1))$ is an isomorphism. It 
follows that $w_0 , \ldots , w_3$ \emph{are linearly independent}. Moreover, 
the fact that $\delta \circ \gamma_1 = 0$ is equivalent to the following 
relation (for matrices with entries in the exterior algebra $\bigwedge 
V$)$\, :$ 
\begin{equation}\label{E:mwedgew}
\mathcal{M} \wedge (w_0\, ,\, w_1\, ,\, w_2\, ,\, w_3)^{\text{t}} 
= 0\, . 
\end{equation}

The argument splits, now, according to Lemma~\ref{L:wsubwedge2v} into three
cases. 

\vskip2mm 

\noindent 
{\bf Case 1.}\quad $W$ \emph{is as in Lemma}~\ref{L:wsubwedge2v}(i). 

\vskip2mm 

\noindent 
In this case, choosing convenient bases of $\tH^1(E^\vee(-1))$ and 
$\tH^1(E^\vee)$, $\mu$ is defined by the transpose of the matrix
\[
\begin{pmatrix} 
-v_1 & 0 & 0 & -v_3\\ 
v_0 & -v_2 & 0 & 0\\
0 & v_1 & -v_3 & 0\\
0 & 0 & v_2 & v_0
\end{pmatrix}\, , 
\text{ i.e., by the matrix }
\mathcal{M} := 
\begin{pmatrix} 
-v_1 & v_0 & 0 & 0\\
0 & -v_2 & v_1 & 0\\
0 & 0 & -v_3 & v_2\\
-v_3 & 0 & 0 & v_0
\end{pmatrix}\, . 
\]
Recall that the same matrix defines the differential
$\delta \colon \tH^1(E^\vee(-1)) \otimes \Omega_\piii^1(1) \ra \tH^1(E^\vee)
\otimes \sco_\piii$  
of the Beilinson monad of $E^\vee$.  

It is an elementary fact that if $u_1 , \ldots , u_p$ 
are linearly independent vectors and if $u_1^\prime , \ldots , u_p^\prime$ are 
some other vectors satisfying $\sum_{i = 1}^pu_i \wedge u_i^\prime = 0$ then 
there exists a $p \times p$ symmetric matrix $A$ such that$\, :$ 
\[
(u_1^\prime , \ldots , u_p^\prime) = (u_1 , \ldots , u_p)A\, . 
\]
In particular, $u_i^\prime \in ku_1 + \ldots + ku_p$, $i = 1 , \ldots , p$.

One sees, now, easily 
that relation \eqref{E:mwedgew} implies that
$w_i \in kv_i$, $i = 0, \ldots , 3$,
i.e., that $w_i = a_iv_i$, $i = 0 , \ldots , 3$. One deduces, from the 
same relation, that $a_i = (-1)^ia_0$, $i = 1, \, 2,\, 3$. Consequently, we 
can assume that $w_i = (-1)^iv_i$, $i = 0, \ldots , 3$. Moreover, after a 
linear change of coordinates in $\piii$, we can assume that
$v_i = (-1)^ie_i$, $i = 0 , \ldots , 3$, where $e_0 , \ldots , e_3$ is the
canonical basis of $V = k^4$. 

Now, the Beilinson monad of $E^\vee$ shows that one has an exact sequence$\, :$ 
\[
0 \lra 2\sco_\piii(-1) \lra K \lra E^\vee \lra 0\, ,
\]
where $K$ is the cohomology sheaf of the monad$\, :$ 
\[
0 \lra \Omega_\piii^2(2) \overset{\gamma_1}{\lra} 4\Omega_\piii^1(1) 
\overset{\delta}{\lra} 4\sco_\piii \lra 0\, , 
\]
with $\delta$ and $\gamma_1$ defined by the matrices$\, :$ 
\[
\delta = 
\begin{pmatrix} 
e_1 & e_0 & 0 & 0\\
0 & -e_2 & -e_1 & 0\\
0 & 0 & e_3 & e_2\\
e_3 & 0 & 0 & e_0
\end{pmatrix}\, ,\  
\gamma_1 = 
\begin{pmatrix} e_0\\ e_1\\ e_2\\ e_3 \end{pmatrix}\, . 
\]

We assert that $K$ is isomorphic to the kernel of the epimorphism $\pi \colon 
6\sco_\piii \ra \sco_\piii(2)$ defined by 
$(X_2^2\, ,\, X_3^2\, ,\, -X_0X_2\, ,\, -X_1X_3\, ,\, X_0^2\, ,\, X_1^2)$. 
\emph{Indeed}, let $K^\prime$ be the kernel of $\pi$. The only non-zero 
cohomology groups $\tH^p(K^\prime(l))$ in the range $-3 \leq l \leq 0$ are 
$\tH^1(K^\prime(-2)) \simeq S_0$, $\tH^1(K^\prime(-1)) \simeq S_1$ and 
$\tH^1(K^\prime) \simeq S_2/I_2$, where $I_2$ is the subspace of $S_2$ generated 
by the monomials defining $\pi$. Choosing the canonical bases of $S_0$ and 
$S_1$ and the basis of $S_2/I_2$ consisting of the classes of the monomials 
$X_0X_1$, $-X_1X_2$, $X_2X_3$, $X_0X_3$ one sees that the Beilinson monad of 
$K^\prime$ is precisely the above monad (the linear part 
$\tH^1(K^\prime(-l)) \otimes \Omega_\piii^l(l) \ra \tH^1(K^\prime(-l+1)) \otimes 
\Omega_\piii^{l-1}(l-1)$ 
of a differential of the Beilinson monad is defined 
by $\sum_{i = 0}^3X_i \otimes e_i$). It follows that $K^\prime \simeq K$. 

Consequently, $E^\vee$ is the cohomology sheaf of a monad of the form$\, :$ 
\[
0 \lra 2\sco_\piii(-1) \overset{\rho}{\lra} 6\sco_\piii \overset{\pi}{\lra} 
\sco_\piii(2) \lra 0\, , 
\]
with $\pi$ the morphism considered above. $\tH^0(\pi(1)) \colon 
\tH^0(6\sco_\piii(1)) \ra \tH^0(\sco_\piii(3))$ is obviously surjective hence its 
kernel has dimension 4. It is, therefore, generated by the elements$\, :$ 
\begin{gather*} 
(X_0\, ,\, 0\, ,\, X_2\, ,\, 0\, ,\, 0\, ,\, 0)^{\text{t}}\, ,\ 
(0\, ,\, X_1\, ,\, 0\, ,\, X_3\, ,\, 0\, ,\, 0)^{\text{t}}\, ,\\
(0\, ,\, 0\, ,\, X_0\, ,\, 0\, ,\, X_2\, ,\, 0)^{\text{t}}\, ,\  
(0\, ,\, 0\, ,\, 0\, ,\, X_1\, ,\, 0\, ,\, X_3)^{\text{t}}\, . 
\end{gather*}
One deduces that $\rho$ must be defined by the transpose of a matrix of the 
form$\, :$ 
\[
\begin{pmatrix} 
a_0X_0 & a_1X_1 & a_0X_2 + a_2X_0 & a_1X_3 + a_3X_1 & a_2X_2 & a_3X_3\\
b_0X_0 & b_1X_1 & b_0X_2 + b_2X_0 & b_1X_3 + b_3X_1 & b_2X_2 & b_3X_3
\end{pmatrix}\, . 
\]
Since $\rho^\vee$ is surjective at the point $[1 : 0 : 0 : 0]$, one has 
$\begin{vmatrix} a_0 & a_2\\ b_0 & b_2 \end{vmatrix} \neq 0$. Permuting, if 
necessary, the rows of the above matrix, one can assume that $a_0 \neq 0$. 
Substracting from the second row the first row multiplied by $b_0a_0^{-1}$, 
one gets the matrix$\, :$ 
\[
\begin{pmatrix} 
a_0X_0 & a_1X_1 & a_0X_2 + a_2X_0 & a_1X_3 + a_3X_1 & a_2X_2 & a_3X_3\\ 
0 & b_1^\prime X_1 & b_2^\prime X_0 & b_1^\prime X_3 + b_3^\prime X_1 & b_2^\prime X_2 
& b_3^\prime X_3
\end{pmatrix}\, , 
\]
where $b_i^\prime = a_0^{-1}\begin{vmatrix} a_0 & a_i\\ b_0 & b_i \end{vmatrix}$, 
$i = 1,\, 2,\, 3$. Notice that $b_2^\prime \neq 0$. Substracting from the 
first row of the new matrix the second row multiplied by $a_2b_2^{\prime -1}$, 
one gets the matrix$\, :$ 
\[
\begin{pmatrix} 
a_0X_0 & a_1^\prime X_1 & a_0X_2 & a_1^\prime X_3 + a_3^\prime X_1 & 0 & 
a_3^\prime X_3\\ 
0 & b_1^\prime X_1 & b_2^\prime X_0 & b_1^\prime X_3 + b_3^\prime X_1 & b_2^\prime X_2 
& b_3^\prime X_3 
\end{pmatrix}\, , 
\]
where $a_i^\prime = b_2^{\prime -1}\begin{vmatrix} a_i & a_2\\ b_i^\prime & 
b_2^\prime \end{vmatrix}$. Finally, multiplying the first (resp., second) row 
of the last matrix by $a_0^{-1}$ (resp., $b_2^{\prime -1}$), one gets the 
matrix$\, :$ 
\[
\begin{pmatrix} 
X_0 & a_1^\secund X_1 & X_2 & a_1^\secund X_3 + a_3^\secund X_1 & 0 & 
a_3^\secund X_3\\ 
0 & b_1^\secund X_1 & X_0 & b_1^\secund X_3 + b_3^\secund X_1 & X_2 
& b_3^\secund X_3 
\end{pmatrix}\, . 
\]
Since $\rho^\vee$ is surjective at the point $[0 : 1 : 0 : 0]$, it follows that 
$\begin{vmatrix} a_1^\secund & a_3^\secund \\ b_1^\secund & b_3^\secund 
\end{vmatrix} \neq 0$. Conversely, if this determinant is non-zero then$\, :$ 
\[
X_0^2\, ,\  X_2^2\, ,\  {\textstyle \begin{vmatrix} a_1^\secund & a_3^\secund \\ 
b_1^\secund & b_3^\secund \end{vmatrix}}X_1^2\, ,\ 
{\textstyle \begin{vmatrix} a_1^\secund & a_3^\secund \\ 
b_1^\secund & b_3^\secund \end{vmatrix}}X_3^2
\]
are among the $2 \times 2$ minors of the above matrix hence this matrix 
defines an epimorphism $6\sco_\piii \ra 2\sco_\piii(1)$. 

\vskip2mm 

\noindent 
{\bf Case 2.}\quad $W$ \emph{is as in Lemma}~\ref{L:wsubwedge2v}(ii). 

\vskip2mm 

\noindent
In this case, $\mu$ is defined by the matrix$\, :$ 
\[
\mathcal{M} := 
\begin{pmatrix} 
-v_1 & v_0 & 0 & 0\\
0 & -v_2 & v_1 & 0\\
0 & 0 & -v_3 & v_2\\
-v_2 & -v_3 & v_0 & v_1
\end{pmatrix}\, . 
\]
Recall relation \eqref{E:mwedgew}. One deduces, from this relation, using the
elementary fact recalled at the beginning of Case 1,
that $w_0 \in kv_0 + kv_1$, $w_1 \in (kv_1 + kv_0) \cap (kv_2 + kv_1) = kv_1$,
$w_2 \in kv_2$, and $w_3 \in kv_2 + kv_3$.
Moreover, the relation $-v_2 \wedge w_1 + v_1 \wedge w_2 = 0$
implies that $w_1 = -av_1$ and $w_2 = av_2$, for some $a \in k$ and 
the relation $-v_3 \wedge w_2 + v_2 \wedge w_3 = 0$ implies that
$w_3 = -av_3 + bv_2$,
for some $b \in k$. The coefficient of $v_1 \wedge v_3$ in the left 
hand side of the relation$\, :$ 
\[
-v_2 \wedge w_0 - v_3 \wedge w_1 + v_0 \wedge w_2 + v_1 \wedge w_3 = 0 
\]
is $-2a$ hence $a = 0$ and this \emph{contradicts} the fact that $w_0 , 
\ldots , w_3$ are linearly independent. Consequently, this case 
\emph{cannot occur}. 

\vskip2mm 

\noindent 
{\bf Case 3.}\quad $W$ \emph{is as in Lemma}~\ref{L:wsubwedge2v}(iii). 

\vskip2mm 

\noindent
In this case, $\mu$ is defined by the matrix$\, :$
\[
\mathcal{M} := 
\begin{pmatrix} 
-v_1 & v_0 & 0 & 0\\ 
0 & -v_2 & v_1 & 0\\
0 & 0 & -v_3 & v_2\\
-v_2 & 0 & v_0 & 0
\end{pmatrix}\, .  
\]
Let $h$ be a non-zero element of $V^\vee$ vanishing in $v_0, \, v_1 ,\, v_2$ 
and let $H \subset \piii$ be the plane of equation $h = 0$.  
The matrix of the multiplication by
$h \colon \tH^1(E^\vee(-1)) \ra \tH^1(E^\vee)$
is obtained by applying $h$ to the entries of $\mathcal{M}$. 
This matrix has rank 1 hence $\h^0(E_H^\vee) = 3$. But this \emph{contradicts} 
Lemma~\ref{L:h0fleq2}. Consequently, this case \emph{cannot occur}. 
\end{proof} 

\begin{lemma}\label{L:betac32} 
Let $E$ be a stable rank $3$ vector bundle on $\piii$ with $c_1 = 0$,
$c_2 = 3$, $c_3 = 2$. Then the differential
$\beta \colon 6\sco_\piii \oplus \sco_\piii(-1) \ra 2\sco_\piii(1)$
of the Horrocks monad of 
$E$ from Lemma~\emph{\ref{L:c32monad}(a)} is defined, up to automorphisms of 
$6\sco_\piii \oplus \sco_\piii(-1)$ and $2\sco_\piii(1)$, by a matrix of the 
form$\, :$ 
\[
\begin{pmatrix} 
h_0 & h_1 & h_2^\prime & h_3^\prime & h_4^\prime & h_5^\prime & q\\ 
0 & 0 & h_0 & h_1 & h_2 & h_3 & 0 
\end{pmatrix}\, , 
\]
with $h_0, \ldots , h_3$ a $k$-basis of $S_1$, and such that
$h_2^\prime , \ldots , h_5^\prime$ belong to $kh_2 + kh_3$ and
$q \in kh_2^2 + kh_2h_3 + kh_3^2$.   
\end{lemma} 

\begin{proof} 
Let $\beta_1 \colon 6\sco_\piii \ra 2\sco_\piii(1)$ be the restriction of 
$\beta$. The other component $\beta_2 \colon \sco_\piii(-1) \ra 2\sco_\piii(1)$ 
of $\beta$ induces an epimorphism $\sco_\piii(-1) \ra \Cok \beta_1$ hence 
$\Cok \beta_1 \simeq \sco_Z(-1)$, for some closed subscheme $Z$ of $\piii$. 
Since one has an epimorphism $2\sco_\piii(1) \ra \sco_Z(-1)$, it follows that 
$\dim Z \leq 0$. 

\vskip2mm 

\noindent
{\bf Claim.}\quad \emph{For a general surjection}
$\pi \colon 2\sco_\piii(1) \ra \sco_\piii(1)$,
$\pi \circ \beta_1 \colon 6\sco_\piii \ra \sco_\piii(1)$ 
\emph{is an epimorphism}. 

\vskip2mm 

\noindent 
\emph{Indeed}, since $Z$ has dimension $\leq 0$ one has $\Cok \beta_1 \simeq 
\sco_Z(1)$. The epimorphism $2\sco_\piii(1) \ra \sco_Z(1)$ is defined by two 
global sections $f_0,\, f_1$ of $\sco_Z$, vanishing simultaneously at no point 
of $Z$. It follows that a general linear combination $a_0f_0 + a_1f_1$ 
vanishes at no point of $Z$. If $\pi \colon 2\sco_\piii(1) \ra \sco_\piii(1)$ is 
defined by $(-a_1 , a_0)$ then $\pi \circ \beta_1$ is an epimorphism. 

\vskip2mm 

It follows from the claim that, up to an automorphism of $2\sco_\piii(1)$, 
one can assume that $\text{pr}_2 \circ \beta_1$ is an epimorphism, where 
$\text{pr}_2 \colon 2\sco_\piii(1) \ra \sco_\piii(1)$ is the projection on the 
second factor. Now, up to an automorphism of $6\sco_\piii \oplus 
\sco_\piii(-1)$, one can assume that the matrix of $\beta$ has the form$\, :$ 
\[
\begin{pmatrix} 
h_0^\prime & h_1^\prime & h_2^\prime & h_3^\prime & h_4^\prime & h_5^\prime & q\\ 
0 & 0 & h_0 & h_1 & h_2 & h_3 & 0 
\end{pmatrix}\, , 
\]
where $h_0 , \ldots , h_3$ is an arbitrary basis of $S_1$. Since $\tH^0(E) = 
0$, $\tH^0(\beta)$ is injective hence $h_0^\prime$ and $h_1^\prime$ are linearly 
independent. Writting $6\sco_\piii$ as $2\sco_\piii \oplus 4\sco_\piii$, up to 
an automorphism of $4\sco_\piii$ one can assume that $h_i = h_i^\prime$, $i = 
0,\, 1$. Finally, substracting from each of the columns 3--7 convenient 
combinations of the first two columns, one can assume that
$h_i^\prime \in kh_2 + kh_3$, $i = 2, \ldots, 5$, and
$q \in kh_2^2 + kh_2h_3 + kh_3^2$. 
\end{proof}

\begin{lemma}\label{L:l0c32}  
Under the hypothesis and with the notation from Lemma~\emph{\ref{L:betac32}}, 
let $L_0 \subset \piii$ be the line of equations 
$h_0 = h_1 = 0$. Then ${\fam0 H}^1(E(1)) \izo {\fam0 H}^1(E_{L_0}(1))$. 
Moreover, ${\fam0 h}^1(E_{L_0}(1)) \leq 1$ and ${\fam0 h}^1(E_{L_0}(1)) = 1$ if 
and only if $h_2^\prime = h_3^\prime = 0$. 
\end{lemma}

\begin{proof}
One has $\tH^1(E(l)) \simeq \Cok \tH^0(\beta(l))$, $l = 0,\, 1$. Using the 
relations$\, :$ 
\[
\begin{pmatrix}0\\ h_ih_j \end{pmatrix} = 
h_i\begin{pmatrix} h_{j+2}^\prime\\ h_j \end{pmatrix} - 
h_{j+2}^\prime\begin{pmatrix} h_i\\ 0 \end{pmatrix}\, ,\ i = 0,\, 1\, ,\  
j = 0, \ldots , 3\, , 
\]
one sees that the multiplication maps $h_i \colon \tH^1(E) \ra \tH^1(E(1))$, 
$i = 0,\, 1$, are both the zero map. Using Lemma~\ref{L:h1erah1el}, one 
deduces that $\tH^1(E(1)) \izo \tH^1(E_{L_0}(1))$. 

Now, one has an exact sequence $0 \ra \sco_{L_0}(-1) \oplus \sco_{L_0}(-2) \ra 
2\sco_{L_0} \oplus N \ra E_{L_0} \ra 0$, where $N$ is the kernel of the 
epimorphism $\phi \colon 4\sco_{L_0} \oplus \sco_{L_0}(-1) \ra 2\sco_{L_0}(1)$ 
defined by the matrix$\, :$ 
\[
\begin{pmatrix} 
h_2^\prime & h_3^\prime & h_4^\prime & h_5^\prime & q\\ 
0 & 0 & h_2 & h_3 & 0 
\end{pmatrix}\, , 
\] 
It follows that $\tH^1(N(1)) \izo \tH^1(E_{L_0}(1))$. One has $\h^0(N) \leq 2$ 
because there is no epimorphism $\sco_{L_0} \oplus \sco_{L_0}(-1) \ra 
2\sco_{L_0}(1)$. Since $\h^0(N) - \h^1(N) = 0$, one deduces that
$\h^1(N) \leq 2$ hence $\h^1(N(1)) \leq 1$ (since $N$ is a vector bundle on
$L_0 \simeq \pj$). Moreover, if $\h^1(N(1)) = 1$ then $\h^1(N) = 2$ hence
$\h^0(N) = 2$ and this happens if and only if $h_2^\prime = h_3^\prime = 0$. 
\end{proof}

\begin{proposition}\label{P:c32moduli} 
The moduli space ${\fam0 M}(2)$ of stable rank $3$ vector bundles $E$ on 
$\piii$ with $c_1 = 0$, $c_2 = 3$, $c_3 = 2$ is nonsingular and irreducible, 
of dimension $28$. 
\end{proposition} 

\begin{proof} 
One uses the same kind of argument as in the proof of 
Prop.~\ref{P:c30000moduli}. Firstly, using Lemma~\ref{L:c32monad}(a) and 
Lemma~\ref{L:l0c32}, one shows that $\tH^2(E^\vee \otimes E) = 0$ for every 
bundle $E$ as in the statement. It follows that $\text{M}(2)$ is nonsingular, 
of pure dimension $28$. Then the open subset $\mathcal{U}$ of $\text{M}(2)$ 
corresponding to the bundles $E$ with $\tH^1(E(1)) = 0$ is irreducible, and 
so is its complement $\text{M}(2) \setminus \mathcal{U}$. 
In order to check that $\text{M}(2)$ is irreducible, it suffices, now, to show 
that 
${\overline {\mathcal{U}}} \cap (\text{M}(2) \setminus \mathcal{U}) 
\neq \emptyset$. This can be shown using the family of monads 
$0 \ra \sco_\piii(-1) \oplus \sco_\piii(-2) \overset{\alpha_t}{\lra} 6\sco_\piii 
\oplus \sco_\piii(-1) \overset{\beta_t}{\lra} 2\sco_\piii(1) \ra 0$, with$\, :$ 
\begin{gather*}
\beta_t := 
\begin{pmatrix} 
X_0 & X_1 & tX_2 & tX_3 & 0 & X_2 & X_3^2\\
0 & 0 & X_0 & X_1 & X_2 & X_3 & 0
\end{pmatrix}\, ,\\ 
\alpha_t := 
\begin{pmatrix} 
X_2 & 0 & X_3 & -X_2 & X_1 & -X_0 & 0\\
X_3^2 + tX_1X_3 & -X_2^2 - tX_1X_2 & X_1^2 & -X_0X_1 & -X_1X_3 & X_1X_2 & -X_0 
\end{pmatrix}^{\text{t}}\, . 
\end{gather*}
\end{proof}

\section{The case $c_3 = 4$}\label{S:c34}

\begin{lemma}\label{L:c34monad}  
Let $E$ be a stable rank $3$ vector bundle on $\piii$ with $c_1 = 0$, 
$c_2 = 3$, $c_3 = 4$. Then $E$ is the cohomology sheaf of a Horrocks monad of 
the form$\, :$ 
\[
0 \lra 2\sco_\piii(-2) \overset{\alpha}{\lra} 3\sco_\piii \oplus 
3\sco_\piii(-1) \overset{\beta}{\lra} \sco_\piii(1) \lra 0\, . 
\]
\end{lemma}

\begin{proof} 
The spectrum of $E$ must be $(-1 , -1 , 0)$ hence $\tH^1(E(l)) = 0$ for 
$l \leq -2$, $\h^1(E(-1)) = 1$, $\tH^2(E(l)) = 0$ for $l \geq -1$. Moreover, 
by Riemann-Roch, $\h^1(E) = 1$. Since $\tH^2(E(-1)) = 0$ and
$\tH^3(E(-2)) = 0$,
it follows that $\tH^1_\ast(E)$ is generated in degrees $\leq 0$ (by the 
Castelnuovo-Mumford Lemma, as formulated in \cite[Lemma~1.21]{acm1}). 

On the other hand, the spectrum of $E^\vee$ is $(0 , 1 , 1)$ hence 
$\tH^1(E^\vee(l)) = 0$ for $l \leq -3$, $\h^1(E^\vee(-2)) = 2$, $\h^1(E^\vee(-1)) 
= 5$, $\tH^2(E^\vee(l)) = 0$ for $l \geq -2$. Moreover, by Riemann-Roch, 
$\h^1(E^\vee) = 5$. Since $\tH^2(E^\vee(-2)) = 0$ and $\tH^3(E^\vee(-3)) = 0$, 
$\tH^1_\ast(E^\vee)$ is generated in degrees $\leq -1$. 

\vskip2mm 

\noindent
{\bf Case 1.}\quad $E$ \emph{has no unstable plane}. 

\vskip2mm 

\noindent 
In this case, applying the Bilinear Map Lemma \cite[Lemma~5.1]{ha} to the 
multiplication map $S_1 \otimes \tH^1(E^\vee(-2)) \ra \tH^1(E^\vee(-1))$ one 
gets that this map is surjective hence $\tH^1_\ast(E^\vee)$ is generated by 
$\tH^1(E^\vee(-2))$. It follows that $E$ is the cohomology sheaf of a (not 
necessarily minimal) Horrocks monad of the form$\, :$ 
\[
0 \lra 2\sco_\piii(-2) \lra B \lra \sco_\piii(1) \oplus \sco_\piii \lra 0\, ,
\]
where $B$ is a direct sum of line bundles. $B$ has rank 7, $c_1(B) = -3$, 
$\tH^0(B(-1)) = 0$ and $\tH^0(B^\vee(-2)) = 0$ hence $B \simeq 4\sco_\piii 
\oplus 3\sco_\piii(-1)$. Since there is no epimorphism $3\sco_\piii(-1) \ra 
\sco_\piii$, it follows that the component $4\sco_\piii \ra \sco_\piii$ of the 
differential $B \ra \sco_\piii(1) \oplus \sco_\piii$ of the monad is non-zero 
hence one can cancel a direct summand $\sco_\piii$ of $B$ and the direct 
summand $\sco_\piii$ of $\sco_\piii(1) \oplus \sco_\piii$. One thus gets a monad 
as in the statement. 

\vskip2mm 

\noindent 
{\bf Case 2.}\quad $E$ \emph{has an unstable plane}. 

\vskip2mm 

\noindent 
In this case, \cite[Prop.~5.1]{ehv} (recalled in Remark~\ref{R:spehv}) 
implies that $\tH^0(E_H) = 0$, for the general plane $H \subset \piii$. One  
deduces that if $h \in S_1$ is a general linear form then multiplication by 
$h \colon \tH^1(E(-1)) \ra \tH^1(E)$ is injective hence bijective. It follows 
that $\tH^1_\ast(E)$ is generated by $\tH^1(E(-1))$. Assume that 
$\tH^1_\ast(E^\vee)$ has $m$ minimal generators of degree $-1$, for some $m \geq 
0$. Then $E$ is the cohomology sheaf of a monad of the form$\, :$ 
\[
0 \lra m\sco_\piii(-1) \oplus 2\sco_\piii(-2) \lra B^\prime \lra \sco_\piii(1) 
\lra 0\, ,
\]
where $B^\prime$ is a direct sum of line bundles. $B^\prime$ has rank $m + 6$, 
$c_1(B^\prime) = -m - 3$, $\tH^0(B^\prime(-1)) = 0$ and 
$\tH^0(B^{\prime \vee}(-2)) = 0$. It follows that
$B^\prime \simeq 3\sco_\piii \oplus (m + 3)\sco_\piii(-1)$.
The component $m\sco_\piii(-1) \ra (m + 3)\sco_\piii(-1)$
of the left differential of the monad is zero, by the 
minimality of $m$. Since there is no locally split monomorphism
$\sco_\piii(-1) \ra 3\sco_\piii$ it follows that $m = 0$. 
\end{proof} 

\begin{lemma}\label{L:c34h0eh0} 
Let $E$ be a stable rank $3$ vector bundle on $\piii$ with $c_1 = 0$, 
$c_2 = 3$, $c_3 = 4$. Then there is a point $x \in \piii$ such that, for every 
plane $H \subset \piii$, ${\fam0 H}^0(E_H) = 0$ if $x \notin H$ and 
${\fam0 h}^0(E_H) = 1$ if $x \in H$. 
\end{lemma}

\begin{proof}
The component $\beta_1 \colon 3\sco_\piii \ra \sco_\piii(1)$ of the differential 
$\beta$ of the monad of $E$ from Lemma~\ref{L:c34monad} is defined by three 
linearly independent linear forms (because $\tH^0(E) = 0$). $x$ is the point 
where these three forms vanish simultaneously. 
\end{proof}

\begin{lemma}\label{L:alpha2} 
Let $E$ be a stable rank $3$ vector bundle on $\piii$ with $c_1 = 0$, 
$c_2 = 3$, $c_3 = 4$ and let $\alpha_2$ be the component $2\sco_\piii(-2) \ra 
3\sco_\piii(-1)$ of the differential $\alpha$ of the monad of $E$ from 
Lemma~\emph{\ref{L:c34monad}}. Then, up to automorphisms of $\piii$, 
$3\sco_\piii$ and $2\sco_\piii(1)$, $\alpha_2^\vee(-1)$ is defined by one of the 
following matrices$\, :$ 
\begin{gather*} 
(\fam0 1) \begin{pmatrix} X_0 & X_1 & X_2\\ 0 & X_0 & X_1 \end{pmatrix};\,  
(\fam0 2) \begin{pmatrix} X_0 & X_1 & 0\\ 0 & X_0 & X_2 \end{pmatrix};\,  
(\fam0 3) \begin{pmatrix} X_0 & 0 & X_2\\ 0 & X_1 & X_2 \end{pmatrix};\,  
(\fam0 4) \begin{pmatrix} X_0 & X_1 & X_2\\ 0 & X_0 & X_3 \end{pmatrix};\\ 
(\fam0 5) \begin{pmatrix} X_0 & 0 & X_2\\ 0 & X_1 & X_3 \end{pmatrix};\, 
(\fam0 6) \begin{pmatrix} X_0 & X_1 & X_2\\ 0 & X_2 & X_3 \end{pmatrix};\,  
(\fam0 7) \begin{pmatrix} X_0 & X_1 & X_2\\ X_1 & X_2 & X_3 \end{pmatrix}. 
\end{gather*} 
\end{lemma} 

\begin{proof} 
The argument is standard. Let us denote
$\alpha_2^\vee(-1) \colon 3\sco_\piii \ra 2\sco_\piii(1)$ by $\phi$.
The morphism $\phi$ is uniquely determined by 
$\tH^0(\phi) \colon \tH^0(3\sco_\piii) \ra \tH^0(2\sco_\piii(1))$ which can be 
viewed as a linear map $\rho \colon k^3 \ra (k^2)^\vee \otimes V^\vee$. Let 
$\psi \colon 3\sco_\pj \ra \sco_\pj(1) \otimes V^\vee$ be the unique morphism 
for which $\tH^0(\psi) = \rho$. Let $u_0,\, u_1$ be the canonical basis of 
$k^2$ and $T_0,\, T_1$ the dual basis of $(k^2)^\vee$. $\phi$ is defined by a 
$2 \times 3$ matrix $\Phi$ with entries in $V^\vee$ and $\psi$ is defined by a 
$4 \times 3$ matrix $\Psi$ with entries in $(k^2)^\vee$. Since both of these 
matrices are derived from $\rho$, they are related as follows$\, :$  
for $i = 0,\, 1$, the $i\, $th row of $\Phi$ is 
$(X_0 , \ldots , X_3)\Psi(u_i)$, where $\Psi(u_i)$ is the $4 \times 3$ matrix 
with entries in $k$ obtained by evaluating the entries of $\Psi$ at $u_i$. 
Notice that $\Psi(u_i)$ defines the reduced stalk of the morphism $\psi$ at 
the point $[u_i]$ of $\pj$. 

\vskip2mm 

\noindent 
{\bf Claim 1.}\quad $\psi \colon 3\sco_\pj \ra \sco_\pj(1) \otimes V^\vee$ 
\emph{has rank} $\geq 2$ \emph{at every point of} $\pj$. 

\vskip2mm 

\noindent 
\emph{Indeed}, the fact that $\tH^0(E^\vee) = 0$ implies that 
$\tH^0(\alpha^\vee)$ is injective. In particular, $\tH^0(\phi(1))$ is injective. 
Assume, by contradiction, that there is a point of $\pj$ where $\psi$ has 
rank $\leq 1$. Up to an automorphism of $\pj$, one can assume that this point 
is $[0:1]$. This means that, up to an automorphism of $2\sco_\piii(1)$, $\phi$ 
is represented by a matrix of the form$\, :$ 
\[
\begin{pmatrix} 
h_{00} & h_{01} & h_{02}\\ 
h_{10} & h_{11} & h_{12} 
\end{pmatrix}, 
\]
with $\dim_k(kh_{10} + kh_{11} + kh_{12}) \leq 1$. Up to an automorphism of 
$3\sco_\piii$ one can assume that $h_{11} = h_{12} = 0$ and this 
\emph{contradicts} the fact that $\tH^0(\phi(1))$ is injective. 

\vskip2mm 

Consider, now, the morphism
$\psi^\vee \colon \sco_\pj(-1) \otimes V \ra 3\sco_\pj$.
Since $\pj$ has dimension 1 and $\tH^1(\sco_\pj(-1)) = 0$ it follows 
that the map $\tH^0(3\sco_\pj) \ra \tH^0(\Cok \psi^\vee)$ is surjective hence 
$\h^0(\Cok \psi^\vee) \leq 3$. One deduces that if $\psi^\vee$ has, generically, 
rank 3 then $\Cok \psi^\vee$ is a torsion sheaf of length $\leq 3$ generated, 
locally, by one element, and if it has rank 2 everywhere then $\Cok \psi^\vee$ 
is a line bundle, which must be $\sco_\pj(2)$ or $\sco_\pj(1)$ (it cannot be 
$\sco_\pj$ because $\tH^0(\phi)$ is injective hence so is $\tH^0(\psi)$). 

Consequently, up to an automorphism of $\pj$, one can assume that 
$\Cok \psi^\vee$ is one of the following sheaves$\, :$ 
\begin{gather*} 
\text{(i)}\, \sco_{\pj , P_0}/\fm_{P_0}^3;\, 
\text{(ii)}\, \sco_{\pj , P_0}/\fm_{P_0}^2 \oplus 
\sco_{\pj , P_1}/\fm_{P_1};\, 
\text{(iii)} {\textstyle \bigoplus_{i = 0}^2}\sco_{\pj , P_i}/\fm_{P_i};\, 
\text{(iv)}\, \sco_{\pj , P_0}/\fm_{P_0}^2;\, \\
\text{(v)}\, \sco_{\pj , P_0}/\fm_{P_0} \oplus 
\sco_{\pj , P_1}/\fm_{P_1};\, 
\text{(vi)}\, \sco_{\pj , P_0}/\fm_{P_0};\, 
\text{(vii)}\, 0;\, 
\text{(viii)}\, \sco_\pj(2);\, 
\text{(ix)}\, \sco_\pj(1),\, 
\end{gather*}
where $P_0 = [0:1]$, $P_1 = [1:0]$ and $P_2 = [1:-1]$. 

In case (i), choosing the $k$-basis of $\sco_{\pj , P_0}/\fm_{P_0}^3$ consisting 
of the classes of the regular functions $1,\, -T_0/T_1,\, T_0^2/T_1^2$, 
$\psi^\vee$ is defined, up to automorphisms of $3\sco_\pj$ and $V$, by the 
following matrix$\, :$ 
\[
\begin{pmatrix} 
T_0 & 0 & 0 & 0\\
T_1 & T_0 & 0 & 0\\
0 & T_1 & T_0 & 0
\end{pmatrix}
\]
hence the matrix $\Psi$ defining $\psi$ is the dual of this matrix. One deduces 
that the matrix $\Phi$ defining $\phi$ is as in item (1) from the statement. 

Analogously, in the cases (ii)--(vii), $\Phi$ is as in the items (2)--(7) from 
the statement, respectively. We show, now, that the cases (viii) and (ix) 
\emph{cannot occur} in our context. 

In case (viii), choosing the $k$-basis $T_1^2,\, -T_0T_1,\, T_0^2$ of 
$\tH^0(\sco_\pj(2))$, $\psi^\vee$ is defined by the matrix$\, :$ 
\[
\begin{pmatrix}
T_0 & 0 & 0 & 0\\
T_1 & T_0 & 0 & 0\\
0 & T_1 & 0 & 0
\end{pmatrix}\text{ hence }\Phi = 
\begin{pmatrix} 
X_0 & X_1 & 0\\
0 & X_0 & X_1
\end{pmatrix}. 
\]
Consider the line  $L \subset \piii$ of equations $X_0 = X_1 = 0$ and restrict 
to $L$ the dual of the monad from Lemma~\ref{L:c34monad}$\, :$ 
\[
0 \lra \sco_L(-1) \overset{\beta_L^\vee}{\lra} 3\sco_L \oplus 3\sco_L(1) 
\overset{\alpha_L^\vee}{\lra} 2\sco_L(2) \lra 0\, . 
\]
Let $\alpha_1 \colon 2\sco_\piii(-2) \ra 3\sco_\piii$ be 
the other component of $\alpha$ and let
$\beta_1 \colon 3\sco_\piii \ra 
\sco_\piii(1)$ and $\beta_2 \colon 3\sco_\piii(-1) \ra \sco_\piii(1)$
be the 
components of $\beta$. Since $\alpha_2^\vee \vb L = 0$, $(\alpha_1^\vee \vb L) 
\colon 3\sco_L \ra 2\sco_L(2)$ is an epimorphism hence its kernel is isomorphic 
to $\sco_L(-4)$. Moreover, $(\alpha_2^\vee \vb L) \circ (\beta_2^\vee \vb L) = 0$ 
hence $(\alpha_1^\vee \vb L) \circ (\beta_1^\vee \vb L) = 0$. It follows that 
$(\beta_1^\vee \vb L) = 0$ and this \emph{contradicts} the fact that $\beta_1$ 
is defined by three linearly independent linear forms (because $\tH^0(E) = 0$). 

Finally, in case (ix), choosing the $k$-basis of $\tH^0(\sco_\pj(1))$ 
consisting of $T_1 ,\, -T_0$, $\psi^\vee$ is defined by the matrix$\, :$ 
\[
\begin{pmatrix} 
T_0 & 0 & 0 & 0\\
T_1 & 0 & 0 & 0\\
0 & T_0 & T_1 & 0
\end{pmatrix}\text{ hence } 
\Phi = 
\begin{pmatrix} 
X_0 & 0 & X_1\\
0 & X_0 & X_2
\end{pmatrix}. 
\]
But $(X_1\, ,\, X_2\, ,\, -X_0)^{\text{t}}$ belongs to the kernel of the map 
$3S_1 \ra 2S_2$ defined by $\Phi$ and this \emph{contradicts} the fact 
that $\tH^0(\phi(1))$ is injective. 
\end{proof} 

\begin{proposition}\label{P:c34} 
Let $E$ be a stable rank $3$ vector bundle on $\piii$ with $c_1 = 0$,
$c_2 = 3$, $c_3 = 4$. Then, for the general plane $H \subset \piii$, one has 
${\fam0 H}^0(E_H^\vee) = 0$. 
\end{proposition}

\begin{proof} 
We intend to apply Lemma~\ref{L:jumpinglines}. 
Consider the monad of $E$ from Lemma~\ref{L:c34monad} and let
$\alpha_1 \colon 2\sco_\piii(-2) \ra 3\sco_\piii$
and $\alpha_2 \colon 2\sco_\piii(-2) \ra 3\sco_\piii(-1)$
be the components of $\alpha$. It follows, from Lemma~\ref{L:alpha2},
that the degeneracy scheme of $\alpha_2$ is a locally 
Cohen-Macaulay subscheme $Y \subset \piii$ of pure dimension 1, which is 
locally complete intersection except at finitely many points and has degree 
3. More precisely, denoting by $L_{ij}$ the line of equations $X_i = X_j = 0$,
$0 \leq i < j \leq 3$, one can assume that one of the following holds$\, :$ 
\begin{enumerate}
\item[(1)] $Y$ is the Weil divisor $3L_{01}$ on the cone $\Sigma$ of
equation $X_1^2 - X_0X_2 = 0\, ;$
\item[(2)] $Y = X \cup L_{01}$, where $X$ is the divisor $2L_{02}$ on the
plane $H_0 \colon X_0 = 0\, ;$ 
\item[(3)] $Y = L_{01} \cup L_{02} \cup L_{12}\, ;$
\item[(4)] $Y = X \cup L_{01}$, where $X$ is the divisor $2L_{03}$ on the
surface $\Sigma \colon X_1X_3 - X_0X_2 = 0\, ;$ 
\item[(5)] $Y = L_{01} \cup L_{02} \cup L_{13}\, ;$
\item[(6)] $Y = C \cup L_{23}$, where $C$ is the conic of equations
$X_0 = X_1X_3 - X_2^2 = 0\, ;$
\item[(7)] $Y$ is a twisted cubic curve.   
\end{enumerate}
One deduces, using the Eagon-Northcott complex, an exact sequence$\, :$ 
\[
0 \lra 2\sco_\piii(-2) \overset{\alpha_2}{\lra} 3\sco_\piii(-1) \lra 
\sci_Y(1) \lra 0\, , 
\]
which, by dualization, produces an exact sequence$\, :$ 
\[
0 \lra \sco_\piii(-1) \lra 3\sco_\piii(1) \overset{\alpha_2^\vee}{\lra} 
2\sco_\piii(2) \overset{\pi}{\lra} \omega_Y(3) \lra 0\, . 
\]
The image of 
$\tH^0(\pi \circ \alpha_1^\vee) \colon \tH^0(3\sco_\piii) \ra \tH^0(\omega_Y(3))$ 
is a subspace $W$ of $\tH^0(\omega_Y(3))$, which 
has dimension 3 (because $\tH^0(E^\vee) = 0$) and generates $\omega_Y(3)$ 
globally. Denoting by $Q$ the cokernel of $\alpha$, one has exact
sequences$\, :$
\begin{gather*}
0 \lra \sco_\piii(-1) \overset{\overline{\beta}^\vee}{\lra} Q^\vee \lra
E^\vee \lra 0\, ,\\
0 \lra \sco_\piii(-1) \lra Q^\vee \lra W \otimes_k \sco_\piii
\overset{\e}{\lra} \omega_Y(3) \lra 0\, , 
\end{gather*} 
where $\e$ is the evaluation morphism. Let us denote $\omega_Y(3)$ by
$\scl$. Consider a line $L$ that is not a component of $Y$ and passes through
none of the points where $Y$ is not locally complete intersection. One has
an exact sequence$\, :$
\[
0 \lra \sco_L(a-1) \lra Q_L^\vee \lra W \otimes_k \sco_L
\overset{\e_L}{\lra} \scl_L \lra 0\, , 
\]
where $a = \text{length}\, (L \cap Y)$ (because $\scl_L \simeq \sco_{L \cap Y}$).
Since $\tH^1(Q_L^\vee) \izo \tH^1(E_L^\vee)$ it follows that
$\h^1(E_L^\vee) \geq 1$ if and only if
$\tH^0(\e_L) \colon W \ra \tH^0(\scl_L)$ is not surjective. 
In that case one must have $a \geq 2$, i.e., $L$ must be a \emph{secant} of
$Y$. We split, now, the proof into several cases according to the various
possibilities for $Y$ listed above. 

\vskip2mm

\noindent
{\bf Case 1.}\quad $Y$ \emph{as in item} (1) \emph{above}.
    
\vskip2mm

\noindent
If $L$ is a secant of $Y$ not passing through the vertex $P_3 := [0:0:0:1]$ of
the cone $\Sigma$ then $L$ is tangent to $\Sigma$ at a point of
$L_{01} \setminus \{P_3\}$ hence $L$ is contained in the plane $H_0$ of
equation $X_0 = 0$. One applies, now, Lemma~\ref{L:jumpinglines}.

\vskip2mm

\noindent
The cases 2, 3 and 5 where $Y$ is as in one of the items (2), (3), or (5)
above can be treated similarly.

\vskip2mm

\noindent
{\bf Case 4.}\quad $Y$ \emph{as in item} (4) \emph{above}.

\vskip2mm

\noindent 
A secant $L$ of $Y$ not passing through the point $P_2 := [0:0:1:0]$ where
$L_{03}$ and $L_{01}$ intersect is either contained in the plane $H_0$
spanned by $L_{03}$ and $L_{01}$ or is tangent to the nonsingular quadric
surface $\Sigma$ at a point of $L_{03} \setminus \{P_2\}$. Taking into account
Lemma~\ref{L:jumpinglines}, it suffices to show that if $L$ is a general
tangent to $\Sigma$ at a point of $L_{03} \setminus \{P_2\}$ then the map
$\tH^0(\e_L)$ is surjective.

Our argument uses two observations. Firstly, let $s_0$, $s_1$, $s_2$ be a
$k$-basis of $W$. The morphism $(0,1) \colon 2\sco_\piii \ra \sco_{L_{03}}$
induces an epimorphism $\omega_Y(1) \ra \sco_{L_{03}}$. Let $f_i$ be the image
of $s_i$ into $\tH^0(\sco_{L_{03}}(2))$. Since $f_0$, $f_1$, $f_2$ generate
$\sco_{L_{03}}(2)$ globally, one can assume that $f_0$ and $f_1$ are linearly
independent. Interpreting the restrictions of $f_0$ and $f_1$ to
$L_{03} \setminus \{P_2\}$ as functions of one variable, one has, for a general
point $x$ of $L_{03} \setminus \{P_2\}$,
$f_0^\prime(x)f_1(x) - f_0(x)f_1^\prime(x) \neq 0$. This means that the images of
$f_0$ and $f_1$ in $\sco_{L_{03}\, ,\, x}/\fm_{L_{03}\, ,\, x}^2$ are linearly
independent.

Secondly, let $\Gamma$ be the ``fat point'' on $\Sigma$ at $x$ defined by
the ideal sheaf $\sci_\Gamma := \sci_{\{x\}}^2 + \sci_\Sigma$.
$\sci_{\{x\}}/\sci_\Gamma$ is a 2-dimensional vector space. If $L$ is a line
tangent to $\Sigma$ at $x$ then $(\sci_L + \sci_\Gamma)/\sci_\Gamma$ is a
1-dimensional subspace of $\sci_{\{x\}}/\sci_\Gamma$ and in this way one gets
all the 1-dimensional subspaces of $\sci_{\{x\}}/\sci_\Gamma$. If
$L \neq L_{03}$ then $\sci_L + \sci_\Gamma$ is the ideal sheaf of the scheme
$L \cap X$, while $\sci_{L_{03}} + \sci_\Gamma$ defines the divisor $2x$ on
$L_{03}$. Let us denote the scheme associated to this divisor by $D$.

Now, according to the first observation, if $x$ is a general point of
$L_{03} \setminus \{P_2\}$ then one can assume that the images of $s_0$ and
$s_1$ in $\scl_D$ are linearly independent. It follows, from the second
observation, that if $L$ is a general tangent line to $\Sigma$ at $x$ then
the images of $s_0$ and $s_1$ in $\scl_{L \cap X}$ are linearly independent.
This implies that the map $\tH^0(\e_L)$ is surjective.

\vskip2mm

\noindent
{\bf Case 6.}\quad $Y$ \emph{as in item} (6) \emph{above}.

\vskip2mm    

\noindent
If $L$ is a secant of $Y$ then either $L$ passes through the point
$P_1 := [0:1:0:0]$ where $C$ and $L_{23}$ intersect, or it is contained in
the plane $H_0$ of $C$, or joins a point $x$ of $C \setminus \{P_1\}$ and a
point $y$ of $L_{23} \setminus \{P_1\}$. $\omega_Y$ is an invertible
$\sco_Y$-module. Since $\omega_Y(1)$ is globally generated,
$\scl = \omega_Y(3)$ is very ample. Since $W$ generates $\scl$ globally,
the restriction maps $W \ra \tH^0(\scl \vb C)$ and
$W \ra \tH^0(\scl \vb L_{23})$ must have rank at least 2.

Let $x$ be a point of $C \setminus \{P_1\}$. The space
$W^\prime := \{\sigma \in W \vb \sigma(x) = 0\}$ has dimension 2. By what have
been said above, the restriction map $W^\prime \ra \tH^0(\scl \vb L_{23})$ is
non-zero hence, for a general point $y \in L_{23} \setminus \{P_1\}$,
the space $W^\secund := \{\sigma \in W^\prime \vb \sigma(y) = 0\}$ has
dimension 1. If $L$ is the line joining $x$ and $y$ then the map
$\tH^0(\e_L)$ is surjective.

\vskip2mm

\noindent
Finally, the case 7 where $Y$ is as in item (7) above is quite easy. 
\end{proof} 

\begin{proposition}\label{P:c34moduli} 
The moduli space ${\fam0 M}(4)$ of stable rank $3$ vector bundles on $\piii$ 
with $c_1 = 0$, $c_2 = 3$, $c_3 = 4$ is nonsingular and connected, of 
dimension $28$. 
\end{proposition} 

\begin{proof} 
Recall, from Lemma~\ref{L:c34monad}, that if $E$ is a stable rank 3 vector 
bundle on $\piii$ with the Chern classes from the statement then $E$ is the 
cohomology sheaf of a monad of the form$\, :$ 
\[
0 \lra 2\sco_\piii(-2) \overset{\alpha}{\lra} 3\sco_\piii \oplus 3\sco_\piii(-1) 
\overset{\beta}{\lra} \sco_\piii(1) \lra 0\, . 
\]
Using the argument from the beginning of the proof of 
Prop.~\ref{P:c30000moduli} and taking into account that $\tH^2(E(l)) = 0$, 
for $l \geq -1$, $\tH^3(E(-1)) = 0$, and $\tH^1(E(2)) = 0$ (actually, 
$\tH^1(E(l)) = 0$ for $l \geq 1$ because $\tH^0(\beta(1))$ must be obviously 
surjective), one deduces that $\tH^2(E^\vee \otimes E) = 0$. Moreover, the kind 
of argument used in the proof of Prop.~\ref{P:c30000moduli} to show the 
irreducibility of $\mathcal{U}$ can be used to prove the irreducibility of 
$\text{M}(4)$. 
\end{proof}

\section{The case $c_3 = 6$}\label{S:c36} 

\begin{lemma}\label{L:c36-1-1-1monad}  
Let $E$ be a stable rank $3$ vector bundle on $\piii$ with $c_1 = 0$, $c_2 = 
3$, $c_3 = 6$ and spectrum $(-1 , -1 , -1)$. Then one has an exact 
sequence$\, :$ 
\[
0 \lra 3\sco_\piii(-2) \overset{\alpha}{\lra} 6\sco_\piii(-1) \lra E 
\lra 0\, . 
\]
\end{lemma}

\begin{proof} 
The result is due to Spindler \cite{sp1}. We include, for completeness, a 
short argument. One has $\tH^2(E(l)) = 0$ for $l \geq -1$ (by the spectrum) 
and $\tH^3(E(l)) = 0$ for $l \geq -4$ (by Serre duality). Moreover, from 
Riemann-Roch, $\h^1(E) = 0$. It follows that $E$ is 1-regular. Using the 
spectrum one deduces that $\tH^1_\ast(E) = 0$. Since, by Riemann-Roch too, 
$\h^0(E(1)) = 6$, one has an epimorphism $6\sco_\piii \ra E(1)$. The kernel $K$ 
of this epimorphism has $\tH^i_\ast(K) = 0$, $i = 1,\, 2$, hence it is a direct 
sum of line bundles. Since $K$ has rank 3, $c_1(K) = -3$ and $\tH^0(K) = 0$ it 
follows that $K \simeq 3\sco_\piii(-1)$.  
\end{proof}

\begin{corollary}\label{C:c36-1-1-1h0eh0} 
Under the hypothesis of Lemma~\emph{\ref{L:c36-1-1-1monad}}, ${\fam0 H}^0(E_H) 
= 0$, for every plane $H \subset \piii$. 
\qed
\end{corollary}

\begin{proposition}\label{P:c36-1-1-1}  
Let $E$ be a stable rank $3$ vector bundle on $\piii$ with $c_1 = 0$,
$c_2 = 3$, $c_3 = 6$ and spectrum $(-1 , -1 , -1)$. If
${\fam0 H}^0(E_H^\vee) \neq 0$, for every plane $H \subset \piii$, then $E$ is
as in Theorem~\emph{\ref{T:c10c23}(b)(i)}. 
\end{proposition}

\begin{proof}
We will show that $E$ has infinitely many unstable planes. Then the main 
result of Vall\`{e}s \cite[Thm.~3.1]{val} (see, also, the proof of
\cite[Prop.~2.2]{val}) will imply the conclusion of the proposition. 

Assume, by contradiction, that $E$ has only finitely many unstable planes. Let 
$\Pi \subset \p^{3 \vee}$ be a plane containing none of the points of 
$\p^{3 \vee}$ corresponding to the unstable planes of $E$. Let $H \subset 
\piii$ be a plane of equation $h = 0$ such that $[h] \in \Pi$. One has 
$\tH^0(E_H^\vee(-1)) = 0$ and $\tH^0(E_H) = 0$ by Cor.~\ref{C:c36-1-1-1h0eh0}. 
Applying Lemma~\ref{L:h0fleq2} to $F := E_H^\vee$ (on $H \simeq \pii$) one gets 
that $\h^0(E_H^\vee) \leq 1$ hence $\h^0(E_H^\vee) = 1$, due to our hypothesis. 
One deduces that the kernel $\scm$ and the cokernel $\scl$ of the restriction
to $\Pi$ of the morphism
$\mu \colon \tH^1(E^\vee(-1)) \otimes \sco_{\p^{3 \vee}}(-1) \ra
\tH^1(E^\vee) \otimes \sco_{\p^{3 \vee}}$
from Remark~\ref{R:mu} are line bundles on $\Pi$ ($\tH^1(E^\vee(-1))$ and
$\tH^1(E^\vee)$ have both dimension 6) hence 
$\scl \simeq \sco_\Pi(a)$ and $\scm \simeq \sco_\Pi(b)$,
for some integers $a,\, b$. One thus has an exact sequence$\, :$ 
\[ 
0 \ra \sco_\Pi(b) \lra \tH^1(E^\vee(-1)) \otimes \sco_\Pi(-1)
\overset{\mu \vert \Pi}{\lra}  
\tH^1(E^\vee) \otimes \sco_\Pi \lra \sco_\Pi(a) \ra 0\, . 
\]
It follows that $a = b + 6$ and $\chi(\sco_\Pi(a-1)) = \chi(\sco_\Pi(b-1))$, 
i.e., $a(a+1) = b(b+1)$. Since the equation $(b+6)(b+7) = b(b+1)$ has no
integer solution, we have got a \emph{contradiction}. 
\end{proof} 

We recall, finally, that according to Spindler \cite[Satz~6,~Satz~7]{sp1}, 
the moduli space $\text{M}(6)$ of stable rank 3 vector bundles on $\piii$ with 
$c_1 = 0$, $c_2 = 3$, $c_3 = 6$ is nonsingular and connected, of dimension 28, 
and that the points of this moduli space corresponding to the bundles with 
spectrum $(-2 , -1 , 0)$ form an irreducible hypersurface. Moreover, by the 
proof of the Proposition on page 72 of \cite{c2}, if $E$ has this spectrum then 
$\tH^0(E_H^\vee) = 0$ for a general plane $H \subset \piii$.

\appendix 
\section{The spectrum of a stable rank 3 reflexive 
sheaf}\label{A:spectrum} 

We recall here the definition and the properties of 
the spectrum of a stable rank 3 reflexive sheaf on $\piii$ with $c_1 = 0$. 
In the case of a torsion free sheaf of arbitrary rank, the results are due to 
Okonek and Spindler \cite{oks2}, \cite{oks3}. In the particular case under 
consideration, their results have been refined by the author in \cite{c1}, 
\cite{c2}. All three authors follow, however, closely the approach of 
Hartshorne \cite{ha}, \cite{ha2} who treated the case of stable rank 2 
reflexive sheaves. This approach uses some technical results on $\pii$ and then 
a restriction theorem   
such as Spindler's generalization \cite{sp} of the theorem of
Grauert-M\"{u}lich or the restriction theorems of Schneider \cite{sch},  
Spindler \cite{sp2}, and of Ein, Hartshorne and Vogelaar \cite{ehv}. 

We begin by recalling the technical result on $\pii$ alluded above. 

\begin{theorem}\label{T:nsubh1f} 
Let $F$ be a rank $3$ vector bundle on $\pii$ with $c_1 = 0$, let 
$R = k[X_0 , X_1 , X_2]$ be the homogeneous coordinate ring of $\pii$ and let 
$N$ be a graded submodule of the graded $R$-module ${\fam0 H}^1_\ast(F)$. 
Put $n_i := \dim_kN_i$, for $i \in \z$. 

\emph{(a)} If $F$ is semistable then$\, :$ 
\begin{enumerate} 
\item[(i)] $n_{-1} \geq n_{-2}$$\, ;$ 
\item[(ii)] $n_{-i} > n_{-i-1}$ if $N_{-i-1} \neq 0$, $\forall \, i \geq 
2$$\, ;$ 
\item[(iii)] If $n_{-i} - n_{-i-1} = 1$ for some $i \geq 3$ then there exists a 
non-zero linear form $\ell \in R_1$ such that $\ell N_{-j} = (0)$ in 
${\fam0 H}^1(F(-j+1))$, $\forall \, j \geq i$. 
\end{enumerate}

\emph{(b)} If $F$ is stable then$\, :$ 
\begin{enumerate} 
\item[(1)] $n_{-1} > n_{-2}$ unless $N_{-1} = (0)$ and $N_{-2} = (0)$ or 
$N_{-1} = {\fam0 H}^1(F(-1))$ and $N_{-2} = {\fam0 H}^1(F(-2))$ $($by 
Riemann-Roch, ${\fam0 h}^1(F(-1)) = {\fam0 h}^1(F(-2)) = c_2$$)$$\, ;$ 
\item[(2)] \emph{(iii)} holds also for $i = 2$. 
\end{enumerate}
\end{theorem}

\begin{proof} 
The difficult assertions are (a)(i) and (b)(1). 
Let us prove (a)(i). Consider the universal extension$\, :$ 
\[
0 \lra F \lra G \lra \tH^1(F) \otimes \sco_\pii \lra 0\, . 
\]
$G$ is semistable with $c_1(G) = c_1(F) = 0$. One has $\tH^0(F(l)) \izo 
\tH^0(G(l))$ for $l \leq 0$, $\tH^1(G) = 0$, $\tH^1(F(l)) \izo \tH^1(G(l))$ for 
$l < 0$, and $\tH^2(F(l)) \izo \tH^2(G(l))$ for $l \geq -2$. By the theorem 
of Beilinson \cite{bei}, one has an exact sequence$\, :$ 
\[
0 \lra \tH^1(F(-2)) \otimes \Omega^2_\pii(2) \overset{d^{-1}}{\lra}  
\begin{matrix} \tH^1(F(-1)) \otimes \Omega^1_\pii(1)\\ \oplus\\ 
\tH^0(F) \otimes \sco_\pii \end{matrix} \lra G \lra 0\, . 
\]
Moreover, by a result of Eisenbud, Fl\o ystad and Schreyer \cite[(6.1)]{efs} 
(both results are recalled in \cite[1.23--1.25]{acm1}), the component 
$d_1^{-1} \colon \tH^1(F(-2)) \otimes \Omega^2_\pii(2) \ra \tH^1(F(-1)) \otimes 
\Omega^1_\pii(1)$
of $d^{-1}$ is defined by the operator $\sum_{i=0}^2X_i \otimes e_i$,
where $e_0,\, e_1,\, e_2$ is the canonical basis of $k^3$
($X_0,\, X_1,\, X_2$ being the dual basis of $R_1 = (k^3)^\vee$). 

We use, now, a trick that appears in the proof of a result of Drezet and 
Le Potier \cite[Prop.~(2.3)]{dlp}. Let $W$ be a non-zero vector subspace of 
$\tH^1(F(-2))$ and put $\Delta(W) := \dim_k(R_1W) - \dim_kW$ (where $R_1W 
\subseteq \tH^1(F(-1))$). Let $\Delta_{\text{min}}$ be the minimal value of 
$\Delta(W)$ (for all $W$ as above). We have to show that
$\Delta_{\text{min}} \geq 0$. Choose $W$ such that
$\Delta(W) = \Delta_{\text{min}}$ and $W$ is maximal among the subspaces having
this property. 

\vskip2mm 

\noindent 
{\bf Claim.}\quad \emph{The morphism}
$(\tH^1(F(-2))/W) \otimes \Omega^2_\pii(2) \ra (\tH^1(F(-1))/R_1W) \otimes
\Omega^1_\pii(1)$
\emph{induced by} $d_1^{-1}$ \emph{is a locally split monomorphism}. 

\vskip2mm 

\noindent 
\emph{Indeed}, assume, by contradiction, that there exists a point $x \in 
\pii$ such that the reduced stalk at $x$ of the morphism from the Claim is 
not injective. One has $x =[v_2]$ for some non-zero vector $v_2$ in $k^3$. 
Complete $v_2$ to a basis $v_0,\, v_1,\, v_2$ of $k^3$ and let
$\ell_0,\, \ell_1,\, \ell_2$ be the dual basis of $(k^3)^\vee$. Since
$\sum X_i \otimes e_i = \sum \ell_i \otimes v_i$ it follows that there exists
a non-zero element $\overline{\xi}$ of $\tH^1(F(-2))/W$ such that
$\ell_i \overline{\xi} = 0$ in $\tH^1(F(-1))/R_1W$, $i = 0,\, 1$. Lift
$\overline{\xi}$ to an element $\xi$ of $\tH^1(F(-2)) \setminus W$. One has
$\ell_i \xi \in R_1W$, $i = 0,\, 1$. Put $\widetilde{W} := W + k\xi$. One has
$\Delta(\widetilde{W}) \leq \Delta(W) = \Delta_{\text{min}}$ and this
\emph{contradicts} the maximality of $W$. 

\vskip2mm 

One deduces, from the Claim, that the cokernel $\scg$ of the morphism 
$W \otimes \Omega_\pii^2(2) \ra (R_1W) \otimes \Omega_\pii^1(1) \oplus 
\tH^0(F) \otimes \sco_\pii$ induced by $d^{-1}$ embeds into $G$. Since $G$ is 
semistable it follows that $-\Delta(W) = c_1(\scg) \leq 0$ and (a)(i) is 
proven. 

For (b)(1) one uses the same argument noticing that $G$ is stable in the sense 
of Gieseker and Maruyama, that is, taking into account that $c_1(G) = 0$ and 
$\chi(G) = 0$, for any coherent subsheaf $\scg$ of $G$ one has
$c_1(\scg) \leq 0$ and if $c_1(\scg) = 0$ then $\chi(\scg) < 0$ unless
$\scg = (0)$ or $\scg = G$. 

The rest of the assertions can be easily deduced from (a)(i) and (b)(1), 
respectively.  
For example, (a)(ii) can be proven by induction on $i$. Let us check the case 
$i = 2$. If, for any non-zero linear form $\ell \in R_1$, multiplication by 
$\ell \colon N_{-3} \ra N_{-2}$ is injective then $n_{-2} \geq n_{-3} + 2$, by 
the Bilinear Map Lemma \cite[Lemma~5.1]{ha}. Assume, now, that there exists 
a form $\ell$ such that the above multiplication map is not injective. Let 
$L \subset \pii$ be the line of equation $\ell = 0$ and let $N^\prime$ be the 
kernel of the composite map $\tH^0_\ast(F_L) \overset{\partial}{\lra} 
\tH^1_\ast(F)(-1) \ra \tH^1_\ast(F)(-1)/N(-1)$. Since $\tH^0(F(-1)) = 0$ one 
has, for $i \geq 1$, an exact sequence$\, :$ 
\[
0 \ra N^\prime_{-i} \lra N_{-i-1} \overset{\ell}{\lra} (\ell N)_{-i} 
\ra 0\, . 
\] 
Since $N^\prime_{-2} \neq 0$ it follows that $\dim_kN^\prime_{-2} < 
\dim_kN^\prime_{-1}$. Moreover, applying (a)(i) to $\ell N$, one gets that 
$\dim_k(\ell N)_{-2} \leq \dim_k(\ell N)_{-1}$ hence $n_{-3} < n_{-2}$. 
\end{proof} 

\begin{remark}\label{R:quotient} 
(a) Applying the above theorem to $F^\vee$ and using Serre duality, 
one gets similar information about the graded quotient modules $Q$ of 
$\tH^1_\ast(F)$ (in degrees $\geq -2$). 

(b) Let $F^{\, \prime}$ be a stable rank 3 vector bundle on $\pii$ with 
$c_1(F^{\, \prime}) = -1$ or $-2$. If $G^\prime$ is defined by the universal 
extension 
$0 \ra F^{\, \prime} \ra G^\prime \ra \tH^1(F^{\, \prime}) \otimes \sco_\pii \ra 0$ 
then one can show that every non-zero coherent subsheaf $\scg^\prime$ of 
$G^\prime$ satisfies $c_1(\scg^\prime) < 0$. One deduces, as in the proof of 
Thm.~\ref{T:nsubh1f}, that if $N^\prime$ is a graded $R$-submodule of 
$\tH^1_\ast(F^{\, \prime})$ and $n_i^\prime := \dim_kN^\prime_i$, then$\, :$ 
\begin{enumerate} 
\item[(1)] $n^\prime_{-i} > n^\prime_{-i-1}$ if $N^\prime_{-i-1} \neq 0$,
$\forall \, i \geq 1\, ;$ 
\item[(2)] If $n^\prime_{-i} - n^\prime_{-i-1} = 1$ for some $i \geq 2$ then there 
exists a non-zero linear form $\ell \in R_1$ such that
$\ell N^\prime_{-j} = (0)$ in $\tH^1(F^{\, \prime}(-j+1))$, $\forall \, j \geq i$.
\end{enumerate} 
See \cite[Prop.~B.4]{acm2} for details. 

(c) Applying the results in (b) to $F^{\, \prime \vee}(-1)$ (notice that 
$c_1(F^{\, \prime \vee}(-1)) = -2$ if $c_1(F^{\, \prime}) = -1$ and 
$c_1(F^{\, \prime \vee}(-1)) = -1$ if $c_1(F^{\, \prime}) = -2$) and using Serre 
duality one gets similar results about the graded quotient modules $Q^\prime$ 
of $\tH^1_\ast(F^{\, \prime})$ (in degrees $\geq -1$). 
\end{remark}  

\begin{definition}\label{D:spectrum}  
Let $\sce$ be a stable rank 3 reflexive sheaf on $\piii$ with $c_1 = 0$. It is 
known that the Chern classes of such a sheaf satisfy the relations $c_2 \geq 2$ 
and $-c_2^2 + c_2 \leq c_3 \leq c_2^2 - c_2$ (see \cite[Thm.~4.2]{ehv}). 
Choose a general plane $H \subset \piii$ of equation 
$h = 0$ such that, at least, $H$ does not contain any singular point of $\sce$ 
and $\sce_H$ is semistable and put$\, :$
\begin{gather*}
N := \text{Im}\, (\tH^1_\ast(\sce) \ra \tH^1_\ast(\sce_H)) \simeq 
\Cok (\tH^1_\ast(\sce(-1)) \overset{h}{\lra} \tH^1_\ast(\sce))\, ,\\
Q := \Cok (\tH^1_\ast(\sce) \ra \tH^1_\ast(\sce_H)) \simeq 
\Ker (\tH^2_\ast(\sce(-1)) \overset{h}{\lra} \tH^2_\ast(\sce))\, . 
\end{gather*}
Notice that, by the semistability of $\sce_H$, the multiplication map 
$h \colon \tH^1(\sce(i-1)) \ra \tH^1(\sce(i))$ (resp.,
$h \colon \tH^2(\sce(i-1)) \ra \tH^2(\sce(i))$ is injective (resp.,
surjective) for $i \leq -1$ (resp., $i \geq -2$). Put $n_i := \dim_kN_i$ and
$q_i := \dim_kQ_i$ and consider the following vector bundle on $\pj\, :$ 
\[
K := {\textstyle \bigoplus_{i \geq 1}}(n_{-i} - n_{-i-1})\sco_\pj(i-1) \oplus 
{\textstyle \bigoplus_{i \geq -1}}(q_i - q_{i+1})\sco_\pj(-i-2)\, . 
\]
One can write $K$ as $\bigoplus_{i=1}^m\sco_\pj(k_i)$, with
$k_1 \leq \cdots \leq k_m$. $k_\sce := (k_1 , \ldots , k_m)$ is called the
\emph{spectrum} of $\sce$. One checks easily, from definitions, that 
\begin{enumerate} 
\item[(i)] $\h^1(\sce(l)) = \h^0(K(l + 1))$, for $l \leq -1\, ;$ 
\item[(ii)] $\h^2(\sce(l)) = \h^1(K(l + 1))$, for $l \geq -3\, ;$ 
\item[(iii)] $m = c_2$ and $-2\sum k_i = c_3\, ;$ 
\item[(iv)] If, for some $i$, $k_i < 0$ (resp., $k_i > 0$) then
$k_i, k_i + 1, \ldots , -1$ (resp., $1 , 2, \ldots , k_i$) occur in the
spectrum (possibly several times).  
\end{enumerate}
\emph{Indeed}, one has, by the defintion of $K\, :$ 
\begin{gather*} 
\h^1(\sce(-l)) - \h^1(\sce(-l-1)) = n_{-l} = \h^0(K(-l+1)) - \h^0(K(-l))\, , 
\text{ for }l \geq 1\, ,\\
\h^2(\sce(l-1)) - \h^2(\sce(l)) = q_l = \h^1(K(l)) - \h^1(K(l+1))\, , 
\text{ for }l \geq -1\, . 
\end{gather*}
Actually, the second relation is valid also for $l = -2$ because$\, :$ 
\[
q_{-2} + n_{-2} = \h^1(\sce_H(-2)) = c_2 = \h^1(\sce_H(-1)) = q_{-1} + n_{-1}\, . 
\]
One gets, in particular, that$\, :$ 
\[K = 
{\textstyle \bigoplus_{i \geq 2}}(n_{-i} - n_{-i-1})\sco_\pj(i-1) \oplus 
{\textstyle \bigoplus_{i \geq -2}}(q_i - q_{i+1})\sco_\pj(-i-2)\, .  
\]
The relations (i) and (ii) follow immediately. 
For the relations (iii) one uses the fact that $n_{-1} + q_{-1} = c_2$ (which
implies that $K$ has rank $c_2$) and the Riemann-Roch formula
$\chi(\sce(-1)) = -c_2 + \frac{1}{2}c_3$, while assertion (iv) follows from 
Thm.~\ref{T:nsubh1f}(a)(ii) and its analogue for quotient modules.  
\end{definition} 

\begin{definition}\label{D:unstableplane} 
Let $\sce$ be a stable rank 3 reflexive sheaf on $\piii$ with $c_1 = 0$ and 
let $r > 0$ be an integer. A plane $H_0 \subset \piii$ is an \emph{unstable 
plane} for $\sce$ of order $r$ if $\text{Hom}(\sce_{H_0} , \sco_{H_0}(-r)) 
\neq 0$ and $\text{Hom}(\sce_{H_0} , \sco_{H_0}(-r-1)) = 0$ (note that $H_0$ 
can contain singular points of $\sce$). By Serre duality, this is equivalent to 
$\tH^2(\sce_{H_0}(r-3)) \neq 0$ and $\tH^2(\sce_{H_0}(r-2)) = 0$. 
\end{definition}

\begin{lemma}\label{L:criterionforunstplane}  
Let $\sce$ be a stable rank $3$ reflexive sheaf on $\piii$ with $c_1 = 0$ and 
let $r > 0$ be an integer such that ${\fam0 H}^2(\sce(r-2)) = 0$. Let $H 
\subset \piii$ be a plane avoiding the singular points of $\sce$ and such 
that $\sce_H$ is semistable and let $Q$ be as in 
Definition~\emph{\ref{D:spectrum}}. Then $\sce$ has an unstable plane of order 
$r$ under any of the hypotheses below$\, :$ 

\emph{(I)} There exists a non-zero linear form $\ell \in 
{\fam0 H}^0(\sco_H(1))$ such that multiplication by $\ell \colon Q_{r-3} \ra 
Q_{r-2}$ is the zero map$\, ;$ 

\emph{(II)} ${\fam0 h}^2(\sce(r-4)) \leq {\fam0 h}^2(\sce(r-3)) + 2$. 
\end{lemma}

\begin{proof} 
Since $\tH^2(\sce(r-2)) = 0$ and $\tH^3(\sce(r-3)) = 0$ one has 
$\tH^2(\sce_{H^\prime}(r-2)) = 0$, for every plane $H^\prime \subset \piii$. 
Consequently, it suffices to show that there is a non-zero linear form 
$h_0 \in S_1$ such that multiplication by
$h_0 \colon \tH^2(\sce(r-4)) \ra \tH^2(\sce(r-3))$ is not surjective. 

Assuming (I), let $h = 0$ be an equation of $H$ and let
$\lambda \in S_1 \setminus kh$ be a linear form lifting $\ell$. One has
$Q_{r-2} \izo \tH^2(\sce(r-3))$ and an exact sequence$\, :$ 
\[
0 \ra Q_{r-3} \lra \tH^2(\sce(r-4)) \overset{h}{\lra} \tH^2(\sce(r-3)) 
\ra 0\, . 
\]
Our hypothesis implies that multiplication by $\lambda \colon 
\tH^2(\sce(r-4)) \ra \tH^2(\sce(r-3))$ maps $Q_{r-3}$ into $(0)$ hence induces 
a map $\overline{\lambda} \colon \tH^2(\sce(r-4))/Q_{r-3} \ra \tH^2(\sce(r-3))$.
On the other hand, multiplication by $h$ induces an isomorphism $\overline{h} 
\colon \tH^2(\sce(r-4))/Q_{r-3} \izo \tH^2(\sce(r-3))$. Then there exists 
$c \in k$ such that $c\overline{h} - \overline {\lambda}$ is not an 
isomorphism. One can take $h_0 = ch - \lambda$. 

Assuming (II), the existence of $h_0$ follows from the Bilinear Map Lemma 
\cite[Lemma~5.1]{ha}. 
\end{proof}

The theorem below shows that if the restriction of $\sce$ to a general plane 
is stable then its spectrum satisfies two additional properties. 
Property (vi) is the analogous of a property proven by 
Hartshorne \cite[Prop.~5.1]{ha2} for stable rank 2 reflexive sheaves. We 
provide, for completeness, a simplified version of his arguments. 

\begin{theorem}\label{T:vvi}
Let $\sce$ be a stable rank $3$ reflexive sheaf on $\piii$ with $c_1 = 0$. 
Assume that there exists a plane $H \subset \piii$ avoinding the singular 
points of $\sce$ such that $\sce_H$ is stable. Let $k_\sce := (k_1, \ldots , 
k_m)$ be the spectrum of $\sce$. 

\emph{(v)} If $0$ does not occur in the spectrum then either
$k_{m-2} = k_{m-1} = k_m = -1$ or $k_1 = k_2 = k_3 = 1\, ;$ 

\emph{(vi)} If, for some $i$ with $2 \leq i \leq m-1$, one has
$k_{i-1} < k_i < k_{i+1} \leq 0$ then $\sce$ has an unstable plane of order
$-k_1$ and $k_1 < k_2 < \cdots < k_i$. 
\end{theorem}

\begin{proof} 
(v) If $0$ does not occur in the spectrum then $n_{-1} = n_{-2}$. 
Thm.~\ref{T:nsubh1f}(b)(1) implies that either $n_{-1} = n_{-2} = 0$ or
$n_{-1} = n_{-2} = c_2$. In the former case $n_{-i} = 0$, $\forall \, i \geq 1$
(by Theorem~\ref{T:nsubh1f}(a)(ii)), $q_{-1} = \h^1(\sce_H(-1)) = c_2$, and 
$q_0 \leq \h^1(\sce_H) = c_2 - 3$ (by Riemann-Roch and the fact that 
$\tH^0(\sce_H) = 0$) hence $q_{-1} - q_0 \geq 3$. 

In the latter case, $q_{-1} = 0$ hence $q_i = 0$ for $i \geq -1$, and
$n_{-3} \leq \h^1(\sce_H(-3)) = c_2 - 3$ hence $n_{-2} - n_{-3} \geq 3$. 

(vi) Let $j \geq -1$ be the integer defined by $-j-2 = k_i$. The hypothesis 
says that $q_j - q_{j+1} = 1$. By the analogue of 
Thm.~\ref{T:nsubh1f}(a)(iii),(b)(2) for quotient modules of $\tH^1_\ast(F)$ 
(with $F = \sce_H$) it follows that there exists a non-zero linear form 
$\ell \in \tH^0(\sco_H(1))$ such that multiplication by
$\ell \colon Q_{l-1} \ra Q_l$ is the zero map, $\forall \, l \geq j$. In
particular, multiplication by $\ell \colon Q_{-k_1-3} \ra Q_{-k_1 -2}$ is the zero
map. Lemma~\ref{L:criterionforunstplane}(I) implies that $\sce$ has an unstable 
plane $H_0$ of order $-k_1$. 

Let us show, now, that $q_l - q_{l+1} = 1$ for $-k_1 - 2 \geq l \geq j$. By 
the definition of an unstable plane, there exists an epimorphism
$\sce_{H_0} \ra \sci_{Z , H_0}(k_1)$, for some 0-dimensional subscheme $Z$ of
$H_0$. One can assume that $H \cap Z = \emptyset$. Let $L_0$ be the
intersection line of $H$ and $H_0$. One has an exact sequence$\, :$ 
\[
0 \ra F^{\, \prime} \lra \sce_H \lra \sco_{L_0}(k_1) \ra 0\, , 
\]
with $F^{\, \prime}$ a stable rank 3 vector bundle on $H$ with 
$c_1(F^{\, \prime}) = -1$. Using the commutative diagram$\, :$ 
\[
\SelectTips{cm}{12}\xymatrix{\sce\ar[r]\ar[d] & \sco_{H_0}(k_1)\ar[d]\\
\sce_H\ar[r] & \sco_{L_0}(k_1)}
\]
one sees that the composite map $\tH^1_\ast(\sce) \ra \tH^1_\ast(\sce_H) 
\ra \tH^1_\ast(\sco_{L_0}(k_1))$ is zero. One deduces an exact sequence$\, :$ 
\[
\tH^1_\ast(F^{\, \prime}) \lra Q \lra \tH^1_\ast(\sco_{L_0}(k_1)) \lra 
\tH^2_\ast(F^{\, \prime})\, . 
\]
Since $F^{\, \prime}$ is stable, $\tH^2(F^{\, \prime}(l)) = 0$ for $l \geq -2$. 
Let $Q^\prime$ be the image of $\tH^1_\ast(F^{\, \prime}) \ra Q$ and put 
$q_l^\prime := \dim_kQ^\prime_l$. One has $q_l = q^\prime_l + 
\h^1(\sco_{L_0}(k_1 + l))$ for $l \geq -2$. 

Using Remark~\ref{R:quotient}(c), 
one gets that $q^\prime_l \geq q^\prime_{l+1}$ for $l \geq -1$ with equality if 
and only if both numbers are 0. Since $q_j - q_{j+1} = 1$ it follows that 
$q^\prime_j = 0$ hence $q^\prime_l = 0$, $\forall \, l \geq j$, hence 
$q_l - q_{l+1} = \h^1(\sco_{L_0}(k_1 + l)) - \h^1(\sco_{L_0}(k_1 + l + 1)) = 1$, 
for $j \leq l \leq -k_1 - 2$. 
\end{proof} 

In the remaining part of this appendix we will show that the properties (v) 
and (vi) from Thm.~\ref{T:vvi} are, actually, satisfied by the spectrum of any 
stable rank 3 reflexive sheaf $\sce$ on $\piii$ with $c_1 = 0$. According to 
the main result of the paper of Ein, Hartshorne and Vogelaar 
\cite[Thm.~0.1]{ehv}, if there is no plane $H \subset \piii$ avoiding the 
singular points of $\sce$ such that $\sce_H$ is stable then either $\sce$ can 
be realized as an extension$\, :$ 
\begin{equation}\label{E:omegasceoh0} 
0 \lra \Omega_\piii(1) \lra \sce \lra \sco_{H_0}(-c_2+1) \lra 0\, , 
\end{equation}
for some plane $H_0 \subset \piii$, or $\sce^\vee$ can be realized as such an 
extension, or $\sce$ is the second symmetric power of a nullcorrelation 
bundle, or $c_2 \leq 3$. 

If $\sce$ can be realized as an extension \eqref{E:omegasceoh0} then 
$\h^1(\sce(-1)) = 1$ and $\h^2(\sce(l)) = \h^2(\sco_{H_0}(-c_2 + 1 + l))$ for 
$l \geq -3$ hence the spectrum of $\sce$ is $(-c_2 + 1 , \ldots , -1 , 0)$. 

If $\sce \simeq S^2N$, for some nullcorrelation bundle $N$, then taking the 
second symmetric power of the monad $0 \ra \sco_\piii(-1) \ra 4\sco_\piii \ra 
\sco_\piii(1) \ra 0$ whose cohomology sheaf is $N$ one gets that $\sce$ is 
the cohomology sheaf of a monad of the form$\, :$ 
\[
0 \ra \sco_\piii(-1) \otimes 4\sco_\piii \ra S^2(4\sco_\piii) \oplus 
(\sco_\piii(-1) \otimes \sco_\piii(1)) \lra 4\sco_\piii \otimes \sco_\piii(1) 
\ra 0\, . 
\]
It follows that $\tH^i(\sce(-2)) = 0$, $i = 1,\, 2$, and $\h^1(\sce(-1)) = 4$ 
hence the spectrum of $\sce$ is $(0 , 0 , 0 , 0)$. 

\vskip2mm 

The following result is the Proposition on page 72 of \cite{c2}. We include an 
argument, for completeness. 

\begin{proposition}\label{P:scedual}
Let $\sce$ be a stable rank $3$ reflexive sheaf on $\piii$ with $c_1 = 0$, 
$c_2 \geq 3$ such that $\sce^\vee$ can be realized as an extension$\, :$ 
\[
0 \lra \Omega_\piii(1) \lra \sce^\vee \lra \sco_{H_0}(-c_2+1) \lra 0\, , 
\]
for some plane $H_0 \subset \piii$, and let $(k_1 , \ldots , k_m)$ be the 
spectrum of $\sce$. Then$\, :$ 

\emph{(a)} ${\fam0 H}^0(\sce_H) = 0$, for the general plane
$H \subset \piii\, ;$ 

\emph{(b)} If $0$ does not occur in the spectrum then
$k_{m-2} = k_{m-1} = k_m = -1\ ;$ 

\emph{(c)} If $k_1 \leq -2$ then $\sce$ has an unstable plane of order 
$-k_1\, ;$ 

\emph{(d)} If, for some $i$ with $2 \leq i \leq m-1$, one has
$k_{i-1} < k_i < k_{i+1} \leq 0$ then $k_1 < k_2 < \cdots < k_i$. 
\end{proposition}

\begin{proof} 
Dualizing the extension from the statement, one gets an exact sequence$\, :$ 
\[
0 \lra \sce \lra \text{T}_\piii(-1) \overset{\phi}{\lra} \sci_{Z , H_0}(c_2) 
\lra 0\, , 
\]
for some 0-dimensional subscheme $Z$ of $H_0$ such that
$\sco_Z(c_2) \simeq \sce xt^1(\sce^\vee , \sco_\piii)$.
Let $\phi_0 \colon \text{T}_\piii(-1)_{H_0} \ra \sci_{Z , H_0}(c_2)$
be the restriction of $\phi$ to $H_0$ and let $G$ be the 
kernel of $\phi_0$. $G$ is a rank 2 vector bundle on $H_0$ with
$c_1(G) = -c_2 + 1$. 

(a) Let $H \subset \piii$ be a general plane. Assume, in particular, that 
$H \neq H_0$ and $H \cap Z = \emptyset$. Let $L$ be the line $H \cap H_0$. 
One has an exact sequence$\, :$ 
\[
0 \lra \sce_H \lra \text{T}_\piii(-1)_H \overset{\phi \vb H}{\lra} 
\sco_L(c_2) \lra 0\, . 
\]
Since $c_1(G) \leq -2$ and $\tH^0(G) = 0$ (because
$\tH^0(\text{T}_\piii(-1)) \izo \tH^0(\text{T}_\piii(-1)_{H_0})$
and $\tH^0(\sce^\vee) = 0$), the theorem of 
Grauert-M\"{u}lich (see, for example, \cite[Thm.~0.1]{c3}) implies that, for 
a general line $L \subset H_0$, $\tH^0(G_L) = 0$. In this case $\tH^0(\phi 
\vb L) \colon \tH^0(\text{T}_\piii(-1)_L) \ra \tH^0(\sco_L(c_2))$ is injective. 
Since $\tH^0(\text{T}_\piii(-1)_H) \izo \tH^0(\text{T}_\piii(-1)_L)$ one 
deduces that $\tH^0(\phi \vb H)$ is injective hence $\tH^0(\sce_H) = 0$. 

(b) Using the notation from the proof of (a), one has exact sequences$\, :$ 
\[
\tH^0(\text{T}_\piii(-1)_H(i)) \ra \tH^0(\sco_L(c_2 + i)) \ra \tH^1(\sce_H(i)) 
\ra \tH^1(\text{T}_\piii(-1)_H(i))\, . 
\]
Using the commutative diagram$\, :$ 
\[
\SelectTips{cm}{12}\xymatrix{\tH^1(\sce(i))\ar[r]\ar[d] & 
\tH^1(\text{T}_\piii(i-1))\ar[d]\\
\tH^1(\sce_H(i))\ar[r] & \tH^1(\text{T}_\piii(-1)_H(i))} 
\]
one deduces that $N_i \subseteq \tH^0(\sco_L(c_2 + i))$ for $i \leq -1$ (see 
Definition~\ref{D:spectrum} for the notation). Moreover, for $i \geq -1$, 
$\tH^1(\sce_H(i))$ and, consequently, $Q_i$ is a quotient of 
$\tH^0(\sco_L(c_2 + i))$. 

Now, if $0$ does not occur in the spectrum then $n_{-1} = n_{-2}$. Since 
$N_i \subseteq \tH^0(\sco_L(c_2 + i))$ for $i \leq -1$, it follows that 
$n_i = 0$ for $i \leq -1$. This implies that the spectrum contains only 
negative integers. Moreover, $q_{-1} = \h^1(\sce_H(-1)) = c_2$ 
and $q_0 \leq \h^1(\sce_H) = c_2 - 3$ (because $\h^0(\sce_H) = 0$ by (a)) hence 
$q_{-1} - q_0 \geq 3$, i.e., $-1$ occurs at least three times in the spectrum. 

(c) As we saw in the proof of (b), $Q_i$ is a quotient of 
$\tH^0(\sco_L(c_2 + i))$ for $i \geq -1$. It follows that if
$\ell \in \tH^0(\sco_H(1))$ is an equation of the line $L = H \cap H_0$, then 
multiplication by $\ell \colon Q_{-k_1-3} \ra Q_{-k_1-2}$ is the zero map. 
Lemma~\ref{L:criterionforunstplane}(I) implies, now, that $\sce$ has an 
unstable plane of order $-k_1$. 

(d) By (c), $\sce$ has an unstable plane $H_1$ of order $-k_1$ hence 
one has an exact sequence$\, :$ 
\[
0 \lra \sce^\prime \lra \sce \lra \sci_{Z_1 , H_1}(k_1) \lra 0\, , 
\]
where $Z_1$ is a subscheme of dimension $\leq 0$ of $H_1$ and $\sce^\prime$ is 
a stable rank 3 reflexive sheaf with $c_1(\sce^\prime) = -1$. It follows from 
(a) that, for the general plane $H \subset \piii$, one has 
$\tH^0(\sce^\prime_H) = 0$. This implies that $\sce^\prime$ is not isomorphic to 
$\Omega_\piii(1)$ hence, by the restriction theorem of Schneider \cite{sch} 
(see, also, \cite[Thm.~3.4]{ehv}), the restriction of $\sce^\prime$ to a general 
plane $H \subset \piii$ is stable. Assuming that $H \cap Z_1 = \emptyset$, 
one has an exact sequence$\, :$ 
\[
0 \lra \sce^\prime_H \lra \sce_H \lra \sco_{L_1}(k_1) \lra 0\, , 
\] 
where $L_1 := H \cap H_1$. One can conclude, now, as in the proof of 
Thm.~\ref{T:vvi}(vi) (with $F^{\, \prime} = \sce^\prime_H$). 
\end{proof}

\begin{lemma}\label{L:scewithunstplane}  
Let $\sce$ be a stable rank $3$ reflexive sheaf on $\piii$ with $c_1 = 0$, and 
let $(k_1 , \ldots , k_m)$ be the spectrum of $\sce$. Assume that $\sce$ has an 
unstable plane. If $0$ does not occur in the spectrum of $\sce$ then
$c_2 \geq 3$ and $k_{m-2} = k_{m-1} = k_m = -1$. 
\end{lemma}

\begin{proof} 
One must have $\tH^0(\sce_H) = 0$ for the general plane $H \subset \piii$ 
because, otherwise, \cite[Prop.~5.1]{ehv} would imply that $\sce$ can be 
realized as an extension \eqref{E:omegasceoh0} and, in this case, as we saw 
above, $\sce$ would have spectrum $(-c_2+1 , \ldots , -1 , 0)$. This implies, 
in particular, that $c_2 \geq 3$ because, by Riemann-Roch,
$\h^1(\sce_H) = c_2 - 3$ (for a general plane $H \subset \piii$ such that
$\sce_H$ is semistable and $\tH^0(\sce_H) = 0$). 

Now, let $H_0$ be an unstable plane for $\sce$ and let $r > 0$ be its order. 
One has an exact sequence$\, :$ 
\[
0 \lra \sce^\prime \lra \sce \lra \sci_{Z , H_0}(-r) \lra 0\, , 
\]
where $Z$ is a subscheme of $H_0$ of dimension $\leq 0$ and $\sce^\prime$ is a 
stable rank 3 reflexive sheaf with $c_1(\sce^\prime) = -1$. By what has been 
shown above, $\tH^0(\sce^\prime_H) = 0$, for the general plane $H \subset 
\piii$. This implies that $\sce^\prime$ is not isomorphic to 
$\Omega_\piii(1)$ hence, by the restriction theorem of Schneider \cite{sch} 
(see, also, \cite[Thm.~3.4]{ehv}), $\sce^\prime_H$ is stable, for a general 
plane $H \subset \piii$. Assuming that $H \cap Z = \emptyset$, one has an 
exact sequence$\, :$ 
\[
0 \lra \sce^\prime_H \lra \sce_H \lra \sco_{L_0}(-r) \lra 0\, , 
\] 
where $L_0 := H \cap H_0$. Since the composite morphism
$\sce \ra \sce_H \ra \sco_{L_0}(-r)$ factorizes through $\sco_{H_0}(-r)$, the
composite map 
$\tH^1(\sce(i)) \ra \tH^1(\sce_H(i)) \ra \tH^1(\sco_{L_0}(-r+i))$ is zero. 
It follows that $N_i \subseteq \tH^1(\sce^\prime_H(i))$ for $i \leq 0$. 

Now, since 0 does not occur in the spectrum of $\sce$ one has $n_{-1} = n_{-2}$. 
Applying Remark~\ref{R:quotient}(b) to $F^{\, \prime} := \sce^\prime_H$ on 
$H \simeq \pii$, one gets that $n_i = 0$ for $i \leq -1$. This implies that 
the spectrum of $\sce$ consists only on negative integers. On the other hand, 
$q_{-1} = \h^1(\sce_H(-1)) = c_2$ and $q_0 \leq \h^1(\sce_H) = c_2 - 3$ hence 
$q_{-1} - q_0 \geq 3$, i.e., $-1$ occurs at least three times in the spectrum. 
\end{proof}

\begin{proposition}\label{c2leq3} 
Let $\sce$ be a stable rank $3$ reflexive sheaf on $\piii$ with $c_1 = 0$. 

\emph{(a)} If $c_2 = 2$ then the possible spectra of $\sce$ are $(-1 , 0)$, 
$(0 , 0)$, and $(0 , 1)$. 

\emph{(b)} If $c_2 = 3$ and $0$ does not occur in the spectrum of $\sce$ then 
this spectrum is either $(-1 , -1 , -1)$ or $(1 , 1 , 1)$. 
\end{proposition}

\begin{proof} 
(a) Taking into account properties (i)--(iv) from Definition~\ref{D:spectrum} 
and the fact that $-2 \leq c_3 \leq 2$ one sees that one has to eliminate the 
spectrum $(-1 , 1)$. If $\sce$ would have this spectrum then one would have 
$\h^2(\sce(-3)) = 2$, $\h^2(\sce(-2)) = 1$, $\h^2(\sce(-1)) = 0$, and 
Lemma~\ref{L:criterionforunstplane}(II) would imply that $\sce$ has an 
unstable plane of order 1. But this would \emph{contradict} 
Lemma~\ref{L:scewithunstplane}. 

(b) Taking into account the properties (i)--(iv) from 
Definition~\ref{D:spectrum} and the fact that $-6 \leq c_3 \leq 6$, one has to 
eliminate the spectra$\, :$ $(-2 , -1 , 1)$, $(-1 , -1 , 1)$, 
$(-1 , 1 , 1)$ and $(-1 , 1 , 2)$. If the spectrum of $\sce$ would be among 
these ones then Lemma~\ref{L:criterionforunstplane}(II) would imply that 
$\sce$ has an unstable plane (of order $-k_1$). Then 
Lemma~\ref{L:scewithunstplane} would imply that the spectrum of $\sce$ is 
$(-1 , -1 , -1)$ which is a \emph{contradiction}. 
\end{proof}

\end{document}